\documentclass[10pt,leqno]{article}  
\usepackage{amsmath,amstext, amsthm,amsfonts,amssymb,amsthm}  
 \numberwithin{equation}{section} 
\def\R{\mathbb{ R}}  
  
\usepackage{epsfig}
\newtheorem{thm}{Theorem}[section]
\newtheorem{lem}[thm]{Lemma}

\newtheorem{cor}[thm]{Corollary}

\newcommand{\be}{\begin{equation}}
\newcommand{\ee}{\end{equation}}
\newcommand\qbi[3]{{{#1}\atopwithdelims[]{#2}}_{#3}}
\newcommand{\ba}{\begin{array}}
\newcommand{\ea}{\end{array}}

\newcommand{\bg}{\begin{gathered}}
\newcommand{\eg}{\end{gathered}}

\newcommand{\al}{\alpha}
\newcommand{\bt}{\beta}

\newcommand{\f}{\phi}
\renewcommand{\t}{\theta}

\newcommand{\gauss}[2]{\genfrac{[}{]}{0pt}{}{#1}{#2}_q}

\newtheorem{rem}[thm]{Remark}
\newcommand{\bea}{\begin{eqnarray}}
\newcommand{\eea}{\end{eqnarray}}

\newcommand{\Sum}{\sum_{n=0}^\infty}
\renewcommand{\(}{\left( } \renewcommand{\)}{\right) }

\begin{document}

\title{ Integral and Series Representations   of $q$-Polynomials and Functions: Part I}

\author{ Mourad E. H.  Ismail 
 \thanks{Research partially supported  by the DSFP of King Saud 
 University and by  the National Plan for Science, Technology and 
 innovation (MAARIFAH), King Abdelaziz City for Science and 
 Technology, Kingdom of Saudi Arabia, Award number 
14-MAT623-02.} \\
 \and Ruiming Zhang \thanks{Corresponding author, research partially supported by National 
 Science Foundation of China, grant No. 11371294.}}
\maketitle

\begin{abstract}
By applying an integral representation for $q^{k^{2}}$ we systematically
derive a large number of new Fourier and Mellin transform pairs and
establish new integral representations for a variety of $q$-functions
and polynomials that naturally arise from combinatorics, analysis, and orthogonal polynomials 
corresponding to indeterminate
moment problems. These functions include $q$-Bessel functions, the
Ramanujan function, Stieltjes--Wigert polynomials, $q$-Hermite
and $q^{-1}$-Hermite polynomials, and the $q$-exponential functions
$e_{q}$, $E_{q}$ and $\mathcal{E}_{q}$. Their representations are
in turn used to derive many new identities involving $q$-functions
and polynomials. In this work we also present contour integral representations
for the above mentioned functions and polynomials.

\end{abstract}

 Filename:  IntegralSumPartIVer3
 
 {\bf AMS Subject Classification 2010} Primary:   33D15, 33D45, 33D60. 
Secondary: 44A20, 42A38 .\\
{\bf Key words and phrases}:    Ramanujan function, 
$q$-Bessel functions,  $q$-exponential functions, $q^{-1}$-Hermite polynomials,  
 Stieltjes--Wigert polynomials,     $q$-Laguerre polynomials,
  basic hypergeometric series, bilateral hypergeometric series,  
  Mellin transforms, Fourier transforms, multiplication formulas, 
  connection relations. 

 \tableofcontents
 
 \setcounter{equation}{0}
\section{Introduction}

This paper grew out of our attempts to derive asymptotics of $q$-orthogonal
polynomials. In \cite{Ism:Zha1}  we studied a Plancherel-Rotach
type asymptotics for certain orthogonal polynomials of the form $\sum_{k\ge0}a_{k}q^{\alpha k^{2}}z^{k}$.
All the main terms of these asymptotics exhibit an interesting common
feature, when the parameter $q$ in one interval they have the Ramanujan function $A_{q}(z)=\sum_{k\ge0}q^{k^{2}}z^{k}/(q;q)_{k}$,
while it in another interval they contain the theta function $\sum_{k=-\infty}^{\infty}q^{k^{2}}z^{k}$
. It is well known that both functions possess certain modular properties,
hence it suggests that we should look at these type of functions
from the point of view of $\log q$. To do that we realized
that the key is to find suitable integral representations for $q^{\alpha k^{2}}$,
and eventually we found two such representations. In the first integral
representation we express $q^{\alpha k^{2}}$ as a regular Fourier
transform of the lognormal distribution, whereas the second integral
representation is to express $q^{\alpha k^{2}}$ as a contour integral
with a theta function as the integrand. These two integral representations
enable us to derive integral representations for some important $q$-orthogonal
polynomials and special functions. By exploiting these integral representations
we could systematically produce many expansion formulas, series transformations,
asymptotic results, and identities. It is these integral representations
and the resulting identities that form the bulk of this work. As the
volume of the results grew it became clear that it will be best to
write a series of articles on the subject and the present paper is
the first part.

 In Section 2 we introduce the notations and terminology used in the 
 rest of the paper and in the future parts.  In particular we fix the 
 notation for the $q$-orthogonal polynomials and functions used in this 
 work. We also state the summation 
  theorems and identities  applied in the sequel. 
 We also derive two  integral representations of $q^{x^2}$. One is a 
 contour integral 
  on the unit circle and 
  the other is an integral on $\R$.  These are \eqref{eqcontJ} and \eqref{eq2ndIR}. In Section  3 we use  
   the contour integral representation \eqref{eqcontJ}  to establish  several contour integral  representations  
    for $q$-functions and polynomials.  Section  4 contains Mellin transform type integral representations for 
     several $q$-functions.  
  This is done through deforming the contour integrals in Section   2.  Our results provide new 
   entries for the 
   Mellin Transform Tables, 
 \cite{Obe}, \cite{Erd:Mag:Obe:Tri}.     The integral on $\R$ can be manipulated to establish a very large 
  number of inverse pairs of Fourier integrals \cite{Erd:Mag:Obe:Tri}. 
  The other integral leads to integral representations of $q$-orthogonal 
  polynomials.  Section 5 contains several inverse Fourier transform 
  pairs some of which involves change of the base $q$. This is followed by Section 6 which uses integral representations for functions and 
  polynomials to derive Plancherel-Rotach asymptotics for the 
  functions involved. The Plancherel-Rotach asymptotics easily 
  come out of the integral representations and demonstrate the 
  usefulness of the integral representation approach. We must note that so far the Riemann-Hilbert 
  problem approach and the nonlinear steepest descent method, \cite{Dei}  have not been able to treat 
  $q$-polynomials arising from indeterminate moment problem with fixed $q$.  Section 7 
  contains a large number of identities involving the $q$-exponential 
 function ${\mathcal E}_q$, the Ramanujan function $A_q$, the 
 $q$-Bessel functions of Jackson, the Stieltjes-Wigert polynomials, 
 the $q$-Laguerre polynomials, and $q$ and $q^{-1}$ Hermite polynomials. 

 In the last section we record some identities for bilateral hypergeometric series. 
 
 It is important to note that many of the results in this work can be 
 proved directly by power series manipulations once we know what the 
 formula looks like. It is important however to emphasize  that the 
 techniques of integral representations produces these series and 
 integral identities with proofs. Throughout this work we also  provide
  alternate proofs of some of our results.

Due to the length of this paper  we felt mentioning where the main 
results are will be helpful to the reader.  The main results of \S 3 are 
formulas \eqref{eqconhn},  \eqref{eqontSn},  \eqref{eqmphim},
\eqref{eqmpsim}, \eqref{eqcontintAq}, \eqref{eqjnu2-2}, 
and \eqref{eq3.12}. The main results in Section 4 are formulas 
\eqref{eq:ramanujan},  \eqref{eq:Stieltjes--Wigert},
\eqref{eq:q-Laguerre}, \eqref{eq:Ismail-Masson}, and
\eqref{eq:ConfluentHalf}, and
\eqref{eq:ConfluentOne}.   Section 5 is very lengthy and its main 
results are Theorems  
\ref{thm:Eq-eq}, \ref{thm:izE}, \ref{thm:airy fourier pairs}, 
\ref{thm:The-Fourier-pair sw}, \ref{thm:The-Fourier-pair q-hermite}
\ref{thm:The-Fourier-pair laguerre}, \ref{thm5.9}, \ref{thm5.10},
\ref{thm5.5.3}, \ref{thm5.7.1}, \ref{thm5.7.2}, \ref{airybeta}. Each 
theorem contains several formulas.  The asymptotic results 
involving orthoognal polynomials proved in 
Section 6 are \eqref{eq*1},  \eqref{eqS2n},  \eqref{eqS2n+1},
\eqref{eqaymbhn1},  \eqref{eqaymbh2n},  \eqref{eqaymbh2n+1}, 
\eqref{eqqLagasym1},   \eqref{eqqLagasym2}, \eqref{eqqLagasym2n}, 
\eqref{eqqLagasym2n+1}.  In addition we prove 
asymptotic results for transcendental functions  which are recorded as 
\eqref{eqAqAsymp},  \eqref{eqJnuqAsymp1}, \eqref{eqJnuqAsymp2},
\eqref{eqfirst1phi1asym}, and \eqref{eqsecond1phi1asym}.  The last 
section contains proofs of \eqref{eqseries1},
\eqref{eq:series and identities airy 1}, 
\eqref{eq:series and identities airy 3},
\eqref{eq:series and identities airy 4}, 
\eqref{eq:series and identities sw 1} , and
\eqref{eq:series and identities sw 3}. In addition we establish Theorems 
\ref{thm721}--\ref{thm7.8}. Each theorem contains several identities.

 
 \setcounter{equation}{0}
\section{Notation and Preliminaries}  

 In this section   we fix the definitions and notations used in the rest of 
 the paper. We also derive 
  the integral representations which we use to compute the asymptotics 
  of the polynomials which we study.  We follow the notation for $q$-shifted factorials in \cite{And:Ask:Roy}, \cite{Gas:Rah}, and 
  \cite{Ismbook}.  
Recall the definition of the $q$-shifted factorials
\begin{eqnarray}
(a;q)_\infty=\prod_{k=0}^{\infty}(1-aq^k),\, (a;q)_n=\frac{(a;q)_\infty}{(aq^n;q)_\infty},\, (a_1, a_2, \cdots, a_m;q)_n = \prod_{j=1}^m (a_j;q)_n,
\end{eqnarray}
and the $q$-binomial coefficient
\begin{eqnarray}
\gauss{n}{k} = \frac{(q;q)_n}{(q;q)_k(q;q)_{n-k}}. 
\end{eqnarray}

There are several $q$-analogues of the exponential function 
$e^{z}$. Two of them are 
\begin{equation}
e_{q}\left(z\right)=\frac{1}{\left(z;q\right)_{\infty}}=\sum_{n=0}^{\infty}\frac{z^{n}}{\left(q;q\right)_{n}},
\quad\left|z\right|<1
\label{eq:1.2}
\end{equation}
and 
\begin{equation}
E_{q}\left(z\right)=\left(-z;q\right)_{\infty}=\sum_{n=0}^{\infty}\frac{q^{\binom{n}{2}}z^{n}}{\left(q;q\right)_{n}}
,\quad z\in\mathbb{C}. 
\label{eq:1.3}
\end{equation}
In \cite{Ism:Zha} the present authors introduced a third 
$q$-exponential function,  
$\mathcal{E}_{q}\left(x;t\right)$, which is   defined by 
\begin{eqnarray}
\begin{gathered}
\mathcal{E}_{q}\left(x;t\right) = \frac{(t^2;q^2)_\infty}{(qt^2;q^2)_\infty}  
 \Sum  \frac{(-it)^n}{(q;q)_n} q^{n^2/4} 
  (-ie^{i\theta}q^{(1-n)/2}, -ie^{-i\theta}q^{(1-n)/2};q)_n. 
\end{gathered}
\label{eqdefIZEQ}
\end{eqnarray}

The analogue of the Chu-Vandermonde sum is 
\begin{eqnarray}
\label{eqqVander}
{}_2\phi_1(q^{-n}, a; c; q, q) = \frac{(c/a;q)_n}{(c;q)_n} a^n
\end{eqnarray}

The Ramanujan ${}_1\psi_1$ sum is \cite[(II.29)]{Gas:Rah}, 
\cite[Theorem 12.3.1]{Ismbook}
\begin{eqnarray}
\label{eqRam1psi1}
\sum_{n=-\infty}^\infty \frac{(a;q)_n}{(b;q)_n} z ^n 
=  \frac{(q, b/a, az, q/az;q)_\infty}
{(b, q/a, z, b/az;q)_\infty}, \quad  |b/a| < |z| < 1,  
\end{eqnarray}
and a limiting case is the Jacobi triple product identity 
 \cite[(II.28)]{Gas:Rah}, \cite[Theorem 12.3.2]{Ismbook}
\begin{eqnarray}
\label{eqJtriple}
\sum_{n=-\infty}^\infty q^{n^2} z^n = (q^2, -qz, -q/z;q^2)_\infty. 
\end{eqnarray}

In a series of papers from 1903 till 1905 F. H. Jackson  introduced 
 $q$-analogues of 
Bessel functions. Their modern notations are, \cite{Ism82},
\begin{eqnarray}
J_\nu^{(1)}(z;q) &=&  \frac{(q^{\nu+1};q)_\infty} {(q;q)_\infty} 
\Sum \frac{(-1)^n  (z/2)^{\nu+2n}}{(q, q^{\nu+1};q)_n}, \quad    |z| < 2, 
\label{eqJnu1}\\
J_\nu^{(2)}(z;q) &=& \frac{(q^{\nu+1};q)_\infty} {(q;q)_\infty}  
\Sum \frac{(-1)^n q^{n(n+\nu)}}{(q, q^{\nu+1};q)_n} (z/2)^{\nu+2n}.
\label{eqJnu2}\\
J_{\nu}^{(3)}(z;q)&=&\frac{\left(q^{\nu+1};q\right)_{\infty}}{(q;q)_{\infty}}
\sum_{n=0}^{\infty}\frac{q^{\binom{n+1}{2}}(-1)^{n}}{\left(q,q^{\nu+1}
;q\right)_{n}}\left(\frac{z}{2}\right)^{\nu+2n}\label{eq:b3-1}
\end{eqnarray}
The  $J_\nu^{(1)}$ and $J_\nu^{(2)}$ are related via 
\begin{eqnarray}
\label{eqJni-Jnu2}
J_\nu^{(1)}(z;q) =\frac{ J_\nu^{(2)}(z;q)}{(-z^2/4;q)_\infty}. 
\end{eqnarray}
Formula \eqref{eqJni-Jnu2} analytically continues $J_\nu^{(1)}$ to 
a meromorphic function in the complex plane. We also need two 
q-analogues of the Bessel functions $I_\nu(z)$, they can be defined 
via Jackson's q-Bessel functions,
\begin{eqnarray}
I_\nu^{(k)}(z;q) &=& e^{-i\pi \nu /2}J_\nu^{(k)}(iz;q),\qquad k=1,2.
\label{eqInu}
\end{eqnarray}
The $q$-Hermite polynomials \cite{Ismbook} 
\begin{equation}
H_{n}\left(\cos\theta\vert q\right)=\sum_{k=0}^{n}\frac{\left(q;q\right)_{n}
 e^{i\left(n-2k\right)\theta}}{\left(q;q\right)_{k}\left(q;q\right)_{n-k}},
\label{eq:1.10}
\end{equation}
 have the generating function 
\begin{equation}
\sum_{n=0}^{\infty}\frac{H_{n}\left(\cos\theta\vert q\right)t^{n}}
{\left(q;q\right)_{n}}
 =\frac{1}{\left(te^{i\theta},te^{-i\theta};q\right)_{\infty}},  
\label{eq:1.11}
\end{equation}
for $\left|t\right|<1$ and $\theta\in\mathbb{R}$, and the $q$-exponential 
 generating function \cite{Ism:Zha}
\begin{equation}
\left(qt^{2};q^{2}\right)_{\infty}\mathcal{E}_{q}\left(x;t\right)
 =\sum_{n=0}^{\infty}\frac{q^{n^{2}/4}t^{n}}{\left(q;q\right)_{n}}
 H_{n}\left(x\vert q\right),
 \label{eq:1.12}
\end{equation}
where $\left|t\right|<1$. From (\ref{eq:1.11}) and 
(\ref{eq:1.12}) it is easy to see that the   orthogonality relation 
\begin{equation}
\frac{1}{2\pi}\int_{1}^{1}H_{m}\left(x\vert q\right)H_{n}
\left(x\vert q\right)w\left(x\vert q\right)dx=\frac{\left(q;q\right)_{n}}
{\left(q;q\right)_{\infty}}\delta_{m,n}\label{eq:1.13}
\end{equation}
  is equivalent to 
\begin{equation}
\frac{1}{2\pi}\int_{-1}^{1}\mathcal{E}_{q}\left(x;u\right)
 \mathcal{E}_{q}\left(x;v\right)w\left(x\vert q\right)dx
 =\frac{\left(-uvq^{1/2};q\right)_{\infty}}{\left(q;q\right)_{\infty}
 \left(qu^{2},qv^{2};q^{2}\right)_{\infty}},
\label{eq:1.14}
\end{equation}
where
\begin{equation}
w\left(x\vert q\right)=\frac{\left(e^{2i\theta},e^{-2i\theta};q\right)_{\infty}}
{\sqrt{1-x^{2}}},\quad 
x=\cos\theta,\ \theta\in\left(0,\pi\right)\label{eq:1.15}
\end{equation}
 and $\left|u\right|<1$ and $\left|v\right|<1$,  In particular, we
have
\begin{equation}
\frac{1}{2\pi}\int_{-1}^{1}\mathcal{E}_{q}\left(x;u\right)
 \mathcal{E}_{q}\left(x;uq^{1/2}\right)w\left(x\vert q\right)dx
  =\frac{\left(-u^{2}q;q\right)_{\infty}}{\left(q,u^{2}q;q\right)_{\infty}}, 
\label{eq:1.16}
\end{equation}
for $\left|u\right|<1$. It is worth noting that \eqref{eq:1.14} is clearly 
 another $q$-analogue of the classical beta integral.

For $z\neq0$, the theta function is defined by  \cite{Whi:Wat} 
\begin{equation}
 \vartheta(z;q) =  \vartheta_4\left(z;q\right)=\sum_{n=-\infty}^{\infty}q^{n^{2}}\left(-z\right)^{n}.
 \label{eq:1.4}
\end{equation}
The Jacobi triple product identity \eqref{eqJtriple} implies 
\begin{equation}
  \vartheta_4\left(z;q\right)=\left(q^{2},qz,q/z;q^{2}\right)_{\infty}.\label{eq:1.5}
\end{equation}
All the material on theta functions contained here is from Whittaker 
and Watson \cite{Whi:Wat}. 
 We shall define the modulus $\tau$ of a theta function by
\bea
\label{eqdefmodulus}
q=e^{\pi\tau i},  \quad \Im\,   \tau >0.
\eea
 The remaining three theta functions are defined by 
\begin{eqnarray}
\vartheta_{3}(v|\tau)  =  \sum_{k=-\infty}^{\infty}q^{k^{2}}e^{2k\pi iv}
   =   \left(q^{2},-qe^{2\pi iv},-qe^{-2\pi iv};q^{2}\right)_{\infty},
   \label{eqtheta3}
\end{eqnarray}
\begin{eqnarray}
\label{eqtheta2}
\quad \vartheta_{2}\left(v\vert\tau\right)  =  \sum_{k=-\infty}^{\infty}
q^{\left(k+1/2\right)^{2}}e^{\left(2k+1\right)\pi iv} 
  =  2q^{1/4}\cos\pi v\left(q^{2},-q^{2}e^{2\pi iv},-q^{2}
  e^{-2\pi iv};q^{2}\right)_{\infty}.
\end{eqnarray}
The theta functions have the periodicity and quasi periodicity properties 
\bea
\vartheta_{3}\left(v+1\mid\tau\right)=\vartheta_{3}\left(v\mid\tau\right),
\label{eqtheta3period}
\eea
\bea
\label{eqtransformtto1/t}
\vartheta_{3}\left(\frac{v}{\tau}\mid-\frac{1}{\tau}\right)
=\sqrt{\frac{\tau}{i}}e^{\pi iv^{2}/\tau}\vartheta_{3}\left(v\mid\tau\right),
\eea
\bea
\label{theat2quaisper}
\vartheta_{3}\left(v+\frac{2n+1}{2}\tau\mid\tau\right)
=\frac{e^{-\left(2n+1\right)\pi iv}\vartheta_{2}\left(v\vert\tau\right)}
{q^{\left(2n+1\right)^{2}/4}}. 
\eea
 The zeros of the the theta functions are given explicitly by 
\bea
\label{eqzerosoft2}
\vartheta_{2}\left(j+\frac{1}{2}+k\tau\mid\tau\right)&=&0,
\quad j,k\in\mathbb{Z}, \\
\vartheta_{3}\left(j+\frac{1}{2}+\left(k+\frac{1}{2}\right)
\tau\mid\tau\right)&=&0,\quad j,k\in\mathbb{Z}.
\label{eqzerosoft3}
\eea
The following transformation will also be used 
\bea
\label{eqtheta3quasiper}
\vartheta_{3}\left(v+n\tau\mid\tau\right)=q^{-n^{2}}e^{-2n\pi vi}\vartheta_{3}\left(v\mid\tau\right).
\eea

We also need the partial $\theta$-function 
\begin{equation}
\omega\left(v;q\right)=\sum_{n=0}^{\infty}q^{n^{2}}v^{n}\label{eq:1.6}
\end{equation}

The   Ramanujan function \cite{Ram1}, \cite{Ism} is 
\begin{eqnarray}
A_q(z) = \sum_{n=0}^\infty  \frac{q^{n^2}}{(q;q)_n} (-z)^n,
\label{eqRamanujanF}
\end{eqnarray} 
was studied extensively by Ramanujan in the lost notebook 
\cite{Ram1}.  
The values $A_q(-1)$ and $A_q(-q)$ give the Rogers-Ramanujan 
identities, \cite{And:Ask:Roy}, \cite{Ismbook}, \cite{Ram2}. Ismail 
indicated that   $A_q(z)$ plays the role of 
 Airy function in Plancherel-Rotach asymptotics of 
$q$-orthogonal polynomials, \cite{Ism}. One can think of 
$A_q$ also as  a  $q$-analogue of $\exp(-x)$, since 
$\lim_{q\to 1} A_q((1-q)z) = \exp(-x)$.
The $q$-gamma fucntion is, \cite{And:Ask:Roy}, \cite{Gas:Rah},  
\begin{eqnarray}
\label{eqqGamma}
\Gamma_q(x) =  (1-q)^{1-x}\frac{(q;q)_\infty}{(q^x;q)_\infty}.
\end{eqnarray}

The $q^{-1}$-Hermite polynomials \cite{Ism:Mas}, the $q$-Laguerre  
polynomials \cite{Ismbook}, \cite{Gas:Rah}, and the Stieltjes--Wigert 
polynomials \cite{Ismbook}, \cite{Sze}, are defined by 
\begin{eqnarray}
h_n{(\sinh\xi\,|\, q)} &=& \sum^n_{k=0}\frac{(q;q)_n}
{(q;q)_k(q;q)_{n-k}}\,
(-1)^k q^{k(k-n)} e^{(n-2k)\xi},  
\label{eqqH} \\
L_n^{(\alpha)}(x;q) &=&  (q^{\alpha+1};q)_n
\sum_{k=0}^n  \frac{q^{\alpha k+k^2}}{(q;q)_k(q;q)_{n-k}}
 \frac{(-x)^k}{\(q^{\alpha+1};q\)_k}, 
 \label{eqqL}
  \end{eqnarray}
and 
\begin{eqnarray}
\label{eq:stieltjes1}
\begin{gathered}
S_{n}(x;q)  =  \frac{1}{(q;q)_{n}}\sum_{k=0}^{n}\left[\begin{array}{c}
n\\
k
\end{array}\right]_{q}q^{k^{2}}\left(-x\right)^{k}  \\
  =  \frac{1}{(q;q)_{n}}\sum_{k=0}^{n}\frac{(q^{-n};q)_{k}}{(q;q)_{k}}
  q^{\binom{k+1}{2}}\left(xq^{n}\right)^{k},   
\end{gathered}
\end{eqnarray}
respectively. The $q^{-1}$-Hermite polynomials were  introduced by 
Askey in \cite{Ask} but 
 were studied in great details by Ismail and Masson in \cite{Ism:Mas} 
 where their moment 
  problem was completely solved, their $N$-extremal measures 
  (\cite{Akh})  were found 
   together with an infinite family of weight functions.   The 
   Stieltjes--Wigert and $q$-Laguerre 
    polynomials are much older, \cite{Sze}, \cite{Gas:Rah}, 
    \cite{Koe:Swa}, \cite{And:Ask:Roy} 
     but there is very little known about the solution of their moment 
     problem, \cite{Ism:Rah},    \cite{Chr1}, \cite{Chr2}. 

Ismail and C. Zhang \cite{Ism:ZhaC} proved the following symmetry 
relation 
for the Stieltjes--Wigert polynomials
\bea
\label{eqsym}
q^{n^2}(-t)^nS_n(q^{-2n}/t;q)= S_n(t;q).
\eea

The $q^{-1}$-Hermite polynomials 
 have the generating function
\begin{equation}
\left(-te^{\xi},te^{-\xi};q\right)_{\infty}=\sum_{n=0}^{\infty}
\frac{q^{\binom{n}{2}}t^{n}}{\left(q;q\right)_{n}}
h_{n}\left(\sinh\xi\vert q\right),\label{eq:1.18}
\end{equation}
 where $x=\sinh\xi$ and $\xi,t\in\mathbb{C}$. Their 
 orthogonality relation  is
\begin{eqnarray}
\int_\R h_m(x|q)  h_n(x|q) d\mu_H(x)  &=&  q^{-n(n+1)/2} (q;q)_n\; 
\delta_{m,n}, 
\end{eqnarray}
where $\mu_H$ is a probability measure that solves the $q^{-1}$ 
Hermite  moment problem, \cite[(21.5.6)]{Ismbook}. Two such 
measures are absolutely 
  continuous and their Radon-Nikodym derivatives are  
\begin{gather}
\label{eq21.6.13}
w_1(x)=  \frac{1}{-   \ln q \; (q;q)_\infty} \frac{(1+x^2)^{-1/2}}
{(-qe^{2\xi},-qe^{-2\xi};q)_\infty},  \\
w_2(x)= \frac{e^{-c^2/4}}{c \sqrt{\pi}} \exp\left(\frac2{\ln q}
\left[\ln(x+\sqrt{x^2+1})\right]^2\right),
  \quad c^2 = - \frac{1}{2}\ln q,
\label{Atawf}
\end{gather}
The Poisson kernel for the $q^{-1}$-Hermite polynomials is, 
\cite{Ism:Mas},
  \cite[(21.2.9)]{Ismbook}
\begin{eqnarray}
\label{eqPoissonKhn}
\Sum h_n(\sinh \xi|q) h_n(\sinh \eta|q) \frac{q^{\binom{n}{2}}}
{(q;q)_n} \, t^n 
= \frac{(-te^{\xi+\eta}, -te^{-\xi-\eta}, t^{\xi-\eta}, t^{-\xi+\eta};q)_\infty}
{(t^2/q;q)_\infty}. 
\end{eqnarray}

Orthogonality relations for the Stieltjes--Wigert and $q$-Laguerre 
polynomials are
\begin{eqnarray}
&{}& \int_0^\infty L_m^{(\alpha)}(x;q)L_n^{(\alpha)}(x;q) 
\frac{x^\alpha }{(-x;q)_\infty} \; dx
 =  \frac{-\pi}{\sin (\pi \alpha)}\frac{(q^{-\alpha };q)_\infty}{(q;q)_\infty}
\frac{(q^{\alpha +1};q)_n}{q^n(q;q)_n}\; \delta_{m,n}, \label{eqqLor} \\
&{}& \int_0^\infty S_m(x;q)S_n(x;q) \exp(- c^2\ln^2(xq^{-1/2}) dx =
 \frac{\sqrt{\pi}q^{-n-1/2}}{c(q;q)_n}\delta_{m,n}, 
\label{eqSWor}
\end{eqnarray}
respectively, where $c = -1/(2\ln q)$, see  (21.8.4)] and  (21.8.46)
 in \cite{Ismbook}. 

The following is a curious relation connecting the $q^{-1}$-Hermite polynomials and the Stieltjes--Wigert polynomials 
\begin{eqnarray}
\label{eqqHasSW}
h_{n}\left(\sinh\xi\vert q\right)=e^{n\xi}\left(q;q\right)_{n}S_{n}\left(e^{-2\xi}q^{-n};q\right).
\end{eqnarray}
It clearly follows from \eqref{eqqH} and \eqref{eq:stieltjes1}.

  Throughout the rest of this section we shall assume that 
 $C$ is a closed contour containing $z =0$ in its interior. 
 
  It is clear from the Jacobi triple product  \eqref{eqJtriple} that 
 \begin{eqnarray}
 \label{eqcontJ}
 q^{ck^2} u^k =  \frac{1}{2\pi i} \oint_C (q^{2c}, -q^c zu, -q^c/zu;q^{2c})_\infty 
 \frac{dz}{z^{k+1}}.
 \end{eqnarray}
 This is the first integral representation which we will use. The second integral representation is 
\begin{equation}
q^{\alpha^{2}/2} =\frac{1}{\sqrt{\pi\log q^{-2}}}\int_{-\infty}^{\infty}\exp\left(\frac{y^{2}}{\log q^{2}}+i\alpha y\right)dy.
\label{eq2ndIR}
\end{equation}
 This is essentially  the evaluation  of the moments of the lognormal distribution, 
  see Example 2.5.3 in \cite{Ismbook}.

  We shall also the Heine transformations \eqref{eqHeine3}  \cite[(III.1)--(III.3)]{Gas:Rah}
\bea
{}_2\phi_1  \left(\left. \begin{matrix} 
 A , B \\
C
\end{matrix}\, \right|q,Z\right) 
&=& \frac{(B, AZ;q)_\infty}{(C, Z;q)_\infty} {}_2\phi_1  \left(\left. \begin{matrix} 
 C/B , Z \\
AZ
\end{matrix}\, \right|q,B \right) 
\label{eqHeine1}
 \\
&=& \frac{(C/B, BZ;q)_\infty}{(C, Z;q)_\infty} 
\; {}_2\phi_1  \left(\left. \begin{matrix} 
 ABZ/C , B \\
BZ
\end{matrix}\, \right|q,\frac{C}{B}\right) 
\label{eqHeine2}\\
 &=& \frac{(ABZ/C;q)_\infty}{(Z;q)_\infty} 
\; {}_2\phi_1  \left(\left. \begin{matrix} 
 C/A , C/B \\
C
\end{matrix}\, \right|q,\frac{ABZ}{C}\right), 
\label{eqHeine3}
\eea
and the ${}_2\phi_1-{}_2\phi_2$ transformation \cite[(III.4)]{Gas:Rah}
\bea
\label{eq2F22F2tran}
\; {}_2\phi_1  \left(\left. \begin{matrix} 
A , B \\
C
\end{matrix}\, \right|q,\ Z\right) = 
 \frac{(AZ;q)_\infty}{(Z;q)_\infty} 
\; {}_2\phi_2  \left(\left. \begin{matrix} 
A , C/B \\
C, AZ
\end{matrix}\, \right|q, BZ\right). 
\eea
  
  \setcounter{equation}{0}
  
 \section{Contour Integral Representations} 
 
 In this section   we use \eqref{eqcontJ} to derive contour integral representations for certain $q$-orthogonal polynomials and unilateral and bilateral hypergeometric functions. Theorem \ref{thm3.1} gives contour integral representations for the Stieltjes--Wigert and $q^{-1}$-Hermite polynomials.  
 We also establish contour integral representations for $_{r+1}\phi_{s+1}$ in \eqref{eqmphim}, 
 for ${}_m\psi_m$ in   \eqref{eqmpsim}, for $J_\nu^{(2)}$ in \eqref{eqjnu2}  and \eqref{eqjnu2-2},  and for $A_q$ in \eqref{eqcontintAq}. In addition we derive contour integral representations for various power and Laurent series in \eqref{eq3.13}--\eqref{eq3.15}. 
 
 \subsection{Orthogonal Polynomials}
 
\begin{thm}\label{thm3.1}
The $q^{-1}$-Hermite and Stieltjes--Wigert polynomials have the integral representations 
\begin{eqnarray}
h_n(\sinh \xi |q) &= & \frac{(-1)^n e^{-n\xi}}{2\pi i} \, q^{-n(n+1)/2}
 \oint_C (q, -qz, -1/z;q)_\infty (qze^{2\xi};q)_n \frac{dz}{z},  \label{eqconhn}\\
 S_n(x;q) &
= &  \frac{1}{2\pi i\, (q;q)_n}  
 \oint_C (q, -qz, -1/z;q)_\infty (x/z;q)_n \frac{dz}{z},
\label{eqontSn}
\end{eqnarray}
respectively, where $\mathcal{C}=\left\{ z=re^{i\theta}|0\le\theta\le2\pi\right\} $
with $r>0$. 
\end{thm}
\begin{proof}
First apply  \eqref{eqcontJ} to the definition of $h_n$  then use 
\begin{eqnarray} 
(q;q)_n/(q;q)_{n-k} = (-1)^k(q^{-n};q)_n q^{nk - \binom{k}{2}}
\notag
\end{eqnarray}
  to find that 
\begin{eqnarray}
\notag
\begin{gathered}
h_n(\sinh \xi|q) = \frac{ e^{n\xi}}{2\pi i} \oint_C (q, -qz, -1/z;q)_\infty \sum_{k=0}^n 
\frac{(q^{-n};q)_k}{(q;q)_k} e^{-2k\xi} z^{-k} \frac{dz}{z} \\
= \frac{ e^{n\xi}}{2\pi i} \oint_C (q, -qz, -1/z;q)_\infty (q^{-n}e^{-2\xi}/z;q)_n \;  \frac{dz}{z},
\end{gathered}
\end{eqnarray}
where we used  the $q$-binomial theorem in the last step. Now \eqref{eqconhn}  follows from this and some simple  manipulations.  The proof of \eqref{eqontSn} is similar.  It also follows 
from \eqref{eqqHasSW} and \eqref{eqconhn}. 
\end{proof}

\subsection{Transcendental Functions}

Applying Cauchy's theorem to  the Ramanujan ${}_1\psi_1$ sum  \eqref{eqRam1psi1}  
we obtain the representation 
\begin{eqnarray}
\label{eqan/bn}
\frac{(a;q)_{n}}{(b;q)_{n}}=\frac{1}{2\pi i}\oint_{\mathcal{C}}\frac{(q,b/a,az,q/az;q)_{\infty}}{(b,q/a,z,b/az;q)_{\infty}}\frac{dz}{z^{n+1}}
\end{eqnarray}
where $n \in\mathbb{Z}$ and $\mathcal{C}$ is the circular contour 
$\left\{ z=\rho e^{i\theta}|0\le\theta\le2\pi\right\}$
with $|b/a|<\rho<1$. We now use the above integral relation, \eqref{eqan/bn},  to derive contour integral representations for ${}_{r+1}\phi_{s+1}$ and  ${}_{r+1}\psi_{s+1}$. Indeed  we have  
\begin{eqnarray*}
\begin{gathered}
{} _{r+1}\phi_{s+1}\left( \begin{array}{cc}
\begin{array}{c}
a_{1},a_{2},\dots,a_{r+1}\\
b_{1},b_{2},\dots,b_{s+1}
\end{array}   \bigg| q,x\end{array}\right)   \qquad \qquad \qquad \qquad  \qquad \qquad \qquad 
\\
= \sum_{k=0}^{\infty}\frac{\left(a_{1},a_{2},\dots,a_{r};q\right)_{k}}{\left(q,b_{1},b_{2},\dots,b_{s};q\right)_{k}}
 \left(-q^{\frac{k-1}{2}}\right)^{k\left(s+1-r\right)}x^{k}\frac{\left(a_{r+1};q\right)_{k}}{\left(b_{s+1};q\right)_{k}}   \qquad \qquad \qquad   \\
 = \sum_{k=0}^{\infty}\frac{\left(a_{1},a_{2},\dots,a_{r};q\right)_{k}\left(-q^{\frac{k-1}{2}}\right)^{k\left(s+1-r\right)}}{\left(q,b_{1},b_{2},\dots,b_{s};q\right)_{k}}    \qquad \qquad \qquad  \qquad \qquad \qquad   \\
  \times \frac{1}{2\pi i}\oint_{\mathcal{C}}\left(\frac{x}{z}\right)^{k}\frac{(q,b_{s+1}/a_{r+1},a_{r+1}z,q/\left(a_{r+1}z\right);q)_{\infty}}{(b_{s+1},q/a_{r+1},z,b_{s+1}/\left(a_{r+1}z\right);q)_{\infty}}\frac{dz}{z}. 
   \qquad \qquad 
  \end{gathered}
\end{eqnarray*}
Therefore we have established the integral representation 
\begin{eqnarray}
\label{eqmphim}
\begin{gathered}
_{r+1}\phi_{s+1}\left(\begin{array}{cc}
\begin{array}{c}
a_{1},a_{2},\dots,a_{r+1}\\
b_{1},b_{2},\dots,b_{s+1}
\end{array}\bigg|  q,x\end{array}\right)  \qquad \qquad      \qquad \qquad \qquad   \\
=  \frac{\left(q,b_{s+1}/a_{r+1};q\right)_{\infty}}{\left(b_{s+1},q/a_{r+1};q\right)_{\infty}}
   \frac{1}{2\pi i}\oint_{\mathcal{C}}{}_{r}\phi_{s}\left(\begin{array}{cc}
\begin{array}{c}
a_{1},a_{2},\dots,a_{r}\\
b_{1},b_{2},\dots,b_{s}
\end{array}\bigg|  q,\frac{x}{z}\end{array}\right)  \qquad \qquad \qquad \\
  \times  \frac{\left(a_{r+1}z,q/\left(a_{r+1}z\right);q\right)_{\infty}}{\left(z,b_{s+1}/\left(a_{r+1}z\right);q\right)_{\infty}}\frac{dz}{z}, \qquad \qquad \qquad  \qquad \qquad \qquad 
 \end{gathered}
\end{eqnarray}
where $\mathcal{C}$ is $\left\{ z=\rho e^{i\theta}|0\le\theta\le2\pi\right\} $
with $|b_{s+1}/a_{r+1}|<\rho<1,|x|<\rho$. In particular we conclude that 
\begin{eqnarray*}
_{2}\phi_{1}\left(\begin{array}{cc}
\begin{array}{c}
a,b\\
c
\end{array}\bigg| & q,x\end{array}\right) & = & \frac{\left(q,c/b;q\right)_{\infty}}{\left(c,q/b;q\right)_{\infty}}\oint_{\mathcal{C}}\frac{\left(ax/z,bz,q/\left(bz\right);q\right)_{\infty}}{\left(x/z,z,c/\left(bz\right);q\right)_{\infty}}\frac{dz}{2\pi iz},
\end{eqnarray*}
where $\mathcal{C}$ is the above mentioned circular contour and 
 $|c/b|<\rho<1,|x|<\rho$. Similarly we find that 
\begin{eqnarray*}
\begin{gathered}
{}_{m+1}\psi_{m+1}\left(\begin{array}{cc}
\begin{array}{c}
a_{1},\dots,a_{m+1}\\
b_{1},\dots,b_{m+1}
\end{array}\bigg|  q,x\end{array}\right)  =  \sum_{k=-\infty}^{\infty}\frac{\left(a_{1},\dots,a_{m};q\right)_{k}x^{k}}{\left(b_{1},\dots,b_{m};q\right)_{k}}\frac{\left(a_{m+1};q\right)_{k}}{\left(b_{m+1};q\right)_{k}}\\
  =  \frac{\left(q,b_{m+1}/a_{m+1};q\right)_{\infty}}{\left(b_{m+1},q/a_{m+1};q\right)_{\infty}}\sum_{k=-\infty}^{\infty}\frac{\left(a_{1},\dots,a_{m};q\right)_{k}}{\left(b_{1},\dots,b_{m};q\right)_{k}}
   \oint_{\mathcal{C}}\frac{(a_{m+1}z,q/a_{m+1}z;q)_{\infty}}{(z,b_{m+1}/a_{m+1}z;q)_{\infty}}\left(\frac{x}{z}\right)^{k}\frac{dz}{2\pi iz},
\end{gathered}
\end{eqnarray*}
which leads to the integral representation 
\begin{eqnarray}
\label{eqmpsim}
\begin{gathered}
_{m+1}\psi_{m+1}\left(\begin{array}{cc}
\begin{array}{c}
a_{1},\dots,a_{m+1}\\
b_{1},\dots,b_{m+1}
\end{array}\bigg|  q,x\end{array}\right)  =  \frac{\left(q,b_{m+1}/a_{m+1};q\right)_{\infty}}{\left(b_{m+1},q/a_{m+1};q\right)_{\infty}}\\
  \times  \oint_{\mathcal{C}}\frac{(a_{m+1}z,q/a_{m+1}z;q)_{\infty}}{(z,b_{m+1}/a_{m+1}z;q)_{\infty}}
   {} _{m}\psi_{m}\left(\begin{array}{cc}
\begin{array}{c}
a_{1},\dots,a_{m}\\
b_{1},\dots,b_{m}
\end{array}\bigg|  q,\frac{x}{z}\end{array}\right)\frac{dz}{2\pi iz},
\end{gathered}
\end{eqnarray}
 for   $\mathcal{C}=\left\{ z=\rho e^{i\theta}|0\le\theta\le2\pi\right\} $
with $|\frac{b_{m+1}}{a_{m+1}}|<\rho<1,|x|>|\frac{b_{1}\cdot b_{2}\cdots b_{m}}{a_{1}\cdot a_{2}\cdots a_{m}}|$.
Another  special case worth noting is  the integral representation 
\begin{eqnarray}
\begin{gathered}
_{2}\psi_{2}\left(\begin{array}{cc}
\begin{array}{c}
a_{1},a_{2}\\
b_{1},b_{2}
\end{array}\bigg|  q,x\end{array}\right) \qquad \qquad \qquad \qquad \qquad  \qquad \qquad  \\
=   \frac{\left(q,q,b_{1}/a_{1},b_{2}/a_{2};q\right)_{\infty}}{\left(b_{1},b_{2},q/a_{1},q/a_{2};q\right)_{\infty}}
   \oint_{\mathcal{C}}\frac{(a_{2}z,q/\left(a_{2}z\right),a_{1}x/z,qz/\left(a_{1}x\right);q)_{\infty}}{(z,x/z,b_{2}/\left(a_{2}z\right),b_{1}z/\left(a_{1}x\right);q)_{\infty}}\frac{dz}{2\pi iz},
\end{gathered}
\end{eqnarray}
where $\mathcal{C}=\left\{ z=\rho e^{i\theta}|0\le\theta\le2\pi\right\} $
and we assume that  $|\frac{b_{2}}{a_{2}}|<\rho<1,|x|>|\frac{b_{1}}{a_{1}}|$. 

Next we apply Cauchy's theorem to \eqref{eq:1.3}
and find that 
\[
\frac{q^{\binom{n}{2}}\left(-1\right)^{n}}{\left(q;q\right)_{n}}=\frac{1}{2\pi i}\oint_{\mathcal{C}}\frac{\left(z;q\right)_{\infty}}{z^{n+1}}dz,
\]
where $\mathcal{C}=\left\{ z=re^{i\theta}|0\le\theta\le2\pi\right\} $
for some $r>0$. Then for $\left|\frac{b}{a}\right|<\left|\frac{x}{r}\right|<1$
we have
\begin{eqnarray*}
\begin{gathered}
_{1}\phi_{1}\left(a;b;q,x\right)  =  \sum_{k=0}^{\infty}\frac{\left(a;q\right)_{k}q^{k\left(k-1\right)/2}\left(-x\right)^{k}}{\left(q,b;q\right)_{k}}\\
  =  \sum_{k=-\infty}^{\infty}\frac{\left(a;q\right)_{k}q^{k\left(k-1\right)/2}\left(-x\right)^{k}}
  {\left(q,b;q\right)_{k}}
  =  \sum_{k=-\infty}^{\infty}\frac{1}{2\pi i}\oint_{\mathcal{C}}\frac{\left(a;q\right)_{k}}{\left(b;q\right)_{k}}\left(\frac{x}{z}\right)^{k}\frac{\left(z;q\right)_{\infty}dz}{z}\\
  =  \frac{1}{2\pi i}\oint_{\mathcal{C}}\sum_{k=-\infty}^{\infty}\frac{\left(a;q\right)_{k}}{\left(b;q\right)_{k}}\left(\frac{x}{z}\right)^{k}\frac{\left(z;q\right)_{\infty}dz}{z}
  =   \frac{1}{2\pi i}\oint_{\mathcal{C}}\frac{\left(q,b/a,z,ax/z,qz/\left(ax\right);q\right)_{\infty}}{\left(b,q/a,x/z,bz/\left(ax\right);q\right)_{\infty}}\frac{dz}{z},
\end{gathered}
\end{eqnarray*}
 that is, for any $r>0$ and $r\left|b\right|<\left|ax\right|<ar$
\[
_{1}\phi_{1}\left(a;b;x\right)=\frac{1}{2\pi i}\oint_{\mathcal{C}}\frac{\left(q,b/a,z,ax/z,qz/\left(ax\right);q\right)_{\infty}}{\left(b,q/a,x/z,bz/\left(ax\right);q\right)_{\infty}}\frac{dz}{z},
\]
 where $\mathcal{C}=\left\{ z=re^{i\theta}|0\le\theta\le2\pi\right\} $.
In particular,
\[
\left(aq;q\right)_{\infty}=\frac{1}{2\pi i}\oint_{\mathcal{C}}\frac{\left(q,\sqrt{q}/z,\sqrt{q}z;q\right)_{\infty}}{\left(a\sqrt{q}/z;q\right)_{\infty}}\frac{dz}{z}
\]
Let $wq=ax$ and $a\to\infty$ to obtain
\[
\sum_{k=0}^{\infty}\frac{q^{k^{2}}w^{k}}{\left(q,b;q\right)_{k}}=\frac{1}{2\pi i}\oint_{\mathcal{C}}\frac{\left(q,z,qw/z,z/w;q\right)_{\infty}}{\left(b,bz/\left(qw\right);q\right)_{\infty}}\frac{dz}{z},
\]
 where $\mathcal{C}=\left\{ z=re^{i\theta}|0\le\theta\le2\pi\right\} $
and $\left|w\right|>r\left|b\right|q^{-1}$ for any $r>0$. In particular,
\begin{eqnarray}
\label{eqjnu2}
J_{\nu}^{\left(2\right)}\left(w\right)=\frac{\left(w/2\right)^{\nu}}{2\pi i}\oint_{\mathcal{C}}\frac{\left(z,-\frac{q^{\nu+1}w^{2}}{4z},-\frac{4z}{w^{2}q^{\nu}};q\right)_{\infty}}{\left(-4z/w^{2};q\right)_{\infty}}\frac{dz}{z}.
\end{eqnarray}
 This technique is constructive and systematic; and can be used in many cases as an algorithm.  For example, 
\begin{eqnarray*}
A_{q}\left(w\right) & = & \sum_{k=0}^{\infty}
\frac{q^{\left(k^{2}-k\right)/2}\left(-w\sqrt{q}\right)^{k}}
{\left(q;q\right)_{k}}q^{k^{2}/2}\\
 & = & \sum_{k=0}^{\infty}\frac{q^{\left(k^{2}-k\right)/2}
 \left(-w\sqrt{q}\right)^{k}}{\left(q;q\right)_{k}2\pi i}\oint_{\mathcal{C}}
 \frac{\left(q,-\sqrt{q}z,-\sqrt{q}/z;q\right)_{\infty}}{z^{k}}\frac{dz}{z}\\
 & = & \oint_{\mathcal{C}}\left(q,-\sqrt{q}z,-\sqrt{q}/z;q\right)_{\infty}
 \sum_{k=0}^{\infty}\frac{q^{\binom{k}{2}}}{\left(q;q\right)_{k}}\left(-
 \frac{w\sqrt{q}}{z}\right)^{k}.
\end{eqnarray*}
The series is summed by \eqref{eq:1.3} and we arrive at the integral 
representation
\begin{eqnarray}
A_{q}\left(w\right)=\frac{1}{2\pi i}\oint_{\mathcal{C}}\left(q,-\sqrt{q}z,-
\sqrt{q}/z,\sqrt{q}w/z;q\right)_{\infty}\frac{dz}{z},
\label{eqcontintAq}
\end{eqnarray}
where the contour $\mathcal{C}$ is a circle of radius $r$, $r >0$.  
For  $\nu>0$, we rewrite the series representation \eqref{eqJnu2} in 
the form 
\[
J_{\nu}^{\left(2\right)}\left(z\right)\left(\frac{z}{2}\right)^{-\nu}=\frac{\left(q^{\nu+1};q\right)_{\infty}}{\left(q;q\right)_{\infty}}\sum_{k=0}^{\infty}\frac{\left(-q;q\right)_{k}}{\left(q^{\nu+1};q\right)_{k}}\frac{q^{k^{2}}\left(-1\right)^{k}}{\left(q^{2};q^{2}\right)_{k}}\left(\frac{q^{\nu}z^{2}}{4}\right)^{k}, 
\]
then apply \eqref{eqan/bn}  with $a =-q, b = q^{\nu+1}$ to find that 
\begin{eqnarray*}
J_{\nu}^{\left(2\right)}\left(w\right)\left(\frac{w}{2}\right)^{-\nu} & = & 
\sum_{k=0}^{\infty}\frac{1}{2\pi i}\oint_{\mathcal{C}}\frac{q^{k^{2}}}
{\left(q^{2};q^{2}\right)_{k}}\left(-\frac{q^{\nu}w^{2}}{4z}\right)^{k}
\frac{(-q^{\nu},-qz,-1/z;q)_{\infty}}{(-1,z,-q^{\nu}/z;q)_{\infty}}\frac{dz}{z}\\
 & = & \frac{1}{2\pi i}\oint_{\mathcal{C}}\sum_{k=0}^{\infty}
 \frac{q^{k^{2}-k}}{\left(q^{2};q^{2}\right)_{k}}
 \left(-\frac{q^{\nu+1}w^{2}}{4z}\right)^{k}
 \frac{(-q^{\nu},-qz,-1/z;q)_{\infty}}{(-1,z,-q^{\nu}/z;q)_{\infty}}\frac{dz}{z}.
\end{eqnarray*}
where $\mathcal{C}$ is a circular contour of radius $r$.  
Here again the series is summed by \eqref{eq:1.3}  and we establish 
the integral representation 
\begin{eqnarray}
\label{eqjnu2-2}
J_{\nu}^{\left(2\right)}\left(w\right)\left(\frac{w}{2}\right)^{-\nu}=
\oint_{\mathcal{C}}\frac{(-q^{\nu},-qz,-1/z;q)_{\infty}
(q^{\nu+1}w^{2}/(4z);q^{2})_{\infty}}{4\pi i(-q,z,-q^{\nu}/z;q)_{\infty}}
\frac{dz}{z},
\end{eqnarray}
where $\mathcal{C}$ is the circle 
$\left\{ z=re^{i\theta}|0\le\theta\le2\pi\right\}$
and  $\nu$ and $r$ are positive. 

Given $\nu$ with $\Re \, \nu>0$, the same  technique can be used
 to obtain
the following formulas
\begin{eqnarray}
\begin{gathered}
\sum_{k=0}^{\infty}\frac{\left(a_{1},a_{2},\dots,a_{r};q\right)_{k}
\left(-x\right)^{k}q^{\nu k^{2}}}{\left(q,b_{1},b_{2},\dots,b_{s};q\right)_{k}}
q^{\left(s+1-r\right)\binom{k}{2}} \\
=   \oint_{\mathcal{C}}{}_{r}\phi_{s}\left(\begin{array}{cc}
\begin{array}{c}
a_{1},a_{2},\dots,a_{r}\\
b_{1},b_{2},\dots,b_{s}
\end{array}\bigg|  q, 1/z \end{array}\right)
   \left(q^{2\nu},-q^{\nu}xz,-q^{\nu}/\left(xz\right);q^{2\nu}\right)_\infty
   \frac{dz}{2\pi iz},
   \end{gathered}
\end{eqnarray}
 
\begin{eqnarray}
\omega\left(x;q\right)=\frac{1}{2\pi i}\oint_{\mathcal{C}}
\left(q^{2},-qxz,-q/\left(xz\right);q^{2}\right)\frac{dz}{z-1},
\end{eqnarray}
 where $\mathcal{C}=\left\{ z=re^{i\theta}|0\le\theta\le2\pi\right\} $
with $r>1$ and for $x\neq0$
\begin{eqnarray}
\label{eq3.12}
\begin{gathered}
\sum_{k=-\infty}^{\infty}\frac{\left(a_{1},\dots,a_{m};q\right)_{k}x^{k}q^{\nu k^{2}}}{\left(b_{1},\dots,b_{m};q\right)_{k}}  \qquad \qquad \qquad \\
\qquad =\frac{1}{2\pi i}  \oint_{\mathcal{C}}{}_{m}\psi_{m}\left(\begin{array}{cc}
\begin{array}{ccc}
a_{1},\dots,a_{m}\\
b_{1},\dots,b_{m}
\end{array}\bigg|  q, 1/z \end{array}\right)  
   \left(q^{2\nu},-q^{\nu}xz,-q^{\nu}/\left(xz\right);q^{2\nu}\right)_\infty\frac{dz}{z},
\end{gathered}
\end{eqnarray}
where $\mathcal{C}=\left\{ z=re^{i\theta}|0\le\theta\le2\pi\right\} $
with $1<r<\left|\frac{a_{1}\cdots a_{m}}{b_{1}\cdots b_{m}}\right|$. In  particular, for $x\neq0$ have proved the representation  
\begin{eqnarray}
\sum_{k=-\infty}^{\infty}\frac{\left(a;q\right)_{k}}{\left(b;q\right)_{k}}q^{\binom{k}{2}}x^{k}=\oint_{\mathcal{C}}\frac{\left(q,q,b/a,a/z,qz/a,-xz,-q/\left(xz\right); q\right)_{\infty} dz}
 {2\pi iz\left(b,q/a,1/z,bz/a;q\right)_{\infty}},
\label{eq3.13}
\end{eqnarray}
where $\mathcal{C}=\left\{ z=re^{i\theta}|0\le\theta\le2\pi\right\} $
with $1<r<\left|\frac{a_{1}}{b_{1}}\right|$, and
\begin{eqnarray}
\label{eq3.14}
\sum_{k=-\infty}^{\infty}\left(a;q\right)_{k}q^{\binom{k}{2}}x^{k}=\oint_{\mathcal{C}}
 \frac{\left(q,q,a/z,qz/a,-xz,-q/\left(xz\right);q\right)_{\infty}dz}{2\pi iz\left(q/a,1/z;q\right)_{\infty}},
\end{eqnarray}
 and
\begin{eqnarray}
\label{eq3.15}
\sum_{k=-\infty}^{\infty}\left(a;q\right)_{k}q^{\binom{k+1}{2}}\left(-1\right)^{k}=\frac{\left(q;q\right)^{2}}{\left(q/a;q\right)_{\infty}}\oint_{\mathcal{C}}\frac{\left(a/z,qz/a,qz;q\right)_{\infty}dz}{2\pi iz},
\end{eqnarray}
where $\mathcal{C}$ is a circle of radius $r$,  $r>1$. 

 \setcounter{equation}{0}

\section{Mellin Transform Type Representations}
In this section  we derive integral representations for various polynomials and functions. The integral representations can be rewritten as Mellin transforms and are indeed new entries for 
 Mellin transforms. 
 
 \subsection{The Ramanujan  Function $A_q$}
\begin{thm}
\label{thm:ramanujan}Let $\left|\arg(x)\right|<\pi$ and $0<q<1$
we have 
\begin{equation}
\frac{A_{q}\left(x\right)}{(q,-q,-q;q)_{\infty}}=\frac{1}{2\pi i}\int_{\rho-i\infty}^{\rho+i\infty}\frac{z^{-\log\left(qx\right)/\log q^{2}}dz}{(z;q^{2})_{\infty}(-qz^{-1/2};q)_{\infty}},\label{eq:ramanujan}
\end{equation}
 where $0<\rho<1$ and we take the principal branch of $z^{1/2}$with
$\Re(z^{1/2})>0$ for $\left|\arg(z)\right|<\pi$.\end{thm}
\begin{proof}
Since both sides of \eqref{eq:ramanujan} are analytic functions in
the domain specified, we only need to prove theorem when  $1>x>0$,
then   analytically continue the result to the whole domain. 
 For $q\in(0,1)$ and $x>0$, observe that 
\[
f(z)=\frac{z^{-\log(qx)/\log q^{2}}}{(z;q^{2})_{\infty}(-qz^{-1/2};q)_{\infty}}
\]
it is meromorphic in the proper right half plane with simple poles
\[
z_{n}=q^{-2n},\quad n\in\mathbb{N}\cup\left\{ 0\right\} 
\]
 with residue 
\[
-\frac{q^{n^{2}}\left(-x\right)^{n}}{(q;q)_{n}}\frac{1}{(q,-q,-q;q)_{\infty}}
\]
and an essential singularity at $z=\infty$. Let $M$ be a large positive
integer and any positive number $0<\rho<1$, we define 
\begin{eqnarray}
\notag
\begin{gathered}
C(M,1)=\left\{ q^{-2M-1}\exp(i\theta)\vert-\frac{\pi}{2}<\theta<\frac{\pi}{2}\right\},\\
C(M,2)=\left\{ z\vert q^{-2M-1}\ge\Im(z)\ge-q^{-2M-1},\ \Re(z)=\rho\right\},
\end{gathered}
\end{eqnarray}
 and 
\[
C(M)=C(M,1)\cup C(M,2),
\]
 then we have 
\begin{eqnarray*}
&{}& \frac{1}{2\pi i}\int_{C(M)}f(z)dz  =  \sum_{n=0}^{M}\mbox{Residue of }f(z)\mbox{ at }q^{-2n} =  \frac{-1}{(q,-q,-q;q)_{\infty}}\sum_{n=0}^{M}\frac{q^{n^{2}}\left(-x\right)^{n}}{(q;q)_{n}}\\
&{}  & =-  \frac{A_{q}\left(x\right)}{(q,-q,-q;q)_{\infty}}+\frac{1}{(q,-q,-q;q)_{\infty}}\sum_{n=M+1}^{\infty}\frac{q^{n^{2}}\left(-x\right)^{n}}{(q;q)_{n}},
\end{eqnarray*}
 by applying Cauchy's theorem. From 
\begin{eqnarray*}
&{}& \frac{1}{2\pi i}\int_{C(M)}f(z)dz  =  \frac{-1}{2\pi i}\int_{\rho-i\infty}^{\rho+i\infty}\frac{z^{-\log\left(qx\right)/\log q^{2}}dz}{(z;q^{2})_{\infty}(-qz^{-1/2};q)_{\infty}}\\
 &{} &  + \frac{1}{2\pi i}\int_{C(M,1)}f(z)dz+\frac{1}{2\pi i}\int_{\rho+iq^{-2M-1}}^{\rho+i\infty}f(z)dz 
  +  \frac{1}{2\pi i}\int_{\rho-i\infty}^{\rho-iq^{-2M-1}}f(z)dz,
\end{eqnarray*}
 we conclude that 
\begin{eqnarray*}
&{}& \frac{1}{2\pi i}\int_{\rho-i\infty}^{\rho+i\infty}\frac{z^{-\log\left(qx\right)/\log q^{2}}dz}{(z;q^{2})_{\infty}(-qz^{-1/2};q)_{\infty}} \\
&{}& =  \frac{A_{q}\left(x\right)}{(q,-q,-q;q)_{\infty}}
  -  \frac{1}{(q,-q,-q;q)_{\infty}}\sum_{n=M+1}^{\infty}\frac{q^{n^{2}}\left(-x\right)^{n}}{(q;q)_{n}}\\
 &{} &\quad +  \frac{1}{2\pi i}\int_{C(M,1)}f(z)dz+\frac{1}{2\pi i}\int_{\rho+iq^{-2M-1}}^{\rho+i\infty}f(z)dz
  + \frac{1}{2\pi i}\int_{\rho-i\infty}^{\rho-iq^{-2M-1}}f(z)dz.
\end{eqnarray*}
 On $C(M,1)$ we derive the estimates,
\begin{eqnarray}
\notag
\begin{gathered}
\left|z^{-\log\left(qx\right)/\log q^{2}}\right|\le\left(qx\right)^{M+1/2}, \\
\left|(-qz^{-1/2};q)_{\infty}\right|\ge\left(q^{M+3/2};q\right)_{\infty}>(q^{3/2};q)_{\infty}, 
\end{gathered}
\end{eqnarray}
 and 
\begin{eqnarray*}
\left|(z;q^{2})_{\infty}\right| & = & \left|(z;q^{2})_{M+1}(q^{2M+2}z;q^{2})_{\infty}\right|\\
 & \ge & q^{-\left(M+1\right)^{2}}(q;q^{2})_{M+1}(q;q^{2})_{\infty} 
 >  q^{-\left(M+1\right)^{2}}(q;q^{2})_{\infty}^{2}, 
\end{eqnarray*}
 by applying $\left|a+b\right|\ge\left|\left|a\right|-\left|b\right|\right|.$
Hence, 
\[
\left|\frac{1}{2\pi i}\int_{C(M,1)}f(z)dz\right|\le\frac{q^{\left(M+1\right)^{2}}\left(qx\right)^{M+1/2}}{2(q^{3/2};q)_{\infty}(q,q^{2})_{\infty}^{2}}.
\]
 Observe that the inequalities
\[
\frac{1-q^{k}}{1-q}\ge kq^{k-1},\quad\frac{q^{k^{2}}}{\left(q;q\right)_{k}}\le\frac{q^{n^{2}/2}}{n!}\left(\frac{\sqrt{q}}{1-q}\right)^{n},
\]
imply
\begin{eqnarray*}
\left|\frac{1}{(q,-q,-q;q)_{\infty}}\sum_{n=M+1}^{\infty}\frac{q^{n^{2}}\left(-x\right)^{n}}{(q;q)_{n}}\right| & \le & \frac{1}{(q,-q,-q;q)_{\infty}}\sum_{n=M+1}^{\infty}\frac{q^{n^{2}/2}}{n!}\left(\frac{x\sqrt{q}}{1-q}\right)^{n}\\
 & \le & \frac{q^{M\left(M+1\right)/2}\left(\frac{xq}{1-q}\right)^{M+1}\exp\left(\frac{xq}{1-q}\right)}{(q,-q,-q;q)_{\infty}\left(M+1\right)!}.
\end{eqnarray*}
 On the vertical line segment $
z=\rho+iy, y\ge q^{-2M-1}$ we write  $z=\sqrt{\rho^{2}+y^{2}}\exp\left(i\arctan (y/\rho)\right) 
$, so that  
\[
z^{1/2}=\sqrt[4]{\rho^{2}+y^{2}}\exp\left(\frac{i}{2}\arctan\frac{y}{\rho}\right).
\]
 Therefore  
\[
\left|(-qz^{-1/2};q)_{\infty}\right|\ge\left|\left(\frac{q}{\sqrt[4]{\rho^{2}+y^{2}}};q\right)_{\infty}\right|\ge\left(q\sqrt{\rho};q\right)_{\infty}.
\]
 The observation  
\[
\left|1-zq^{2k}\right|\ge yq^{2k},\quad\left|1-zq^{2k}\right|\ge\left|1-\rho q^{2k}\right|,
\]
implies  
\begin{eqnarray*}
\begin{gathered}
\left|(z;q^{2})_{\infty}\right|  =  \left|(z;q^{2})_{M}\right|\cdot\left|(q^{2M}z;q^{2})_{\infty}\right| \qquad \qquad \\
  \qquad \qquad \ge  y^{M}q^{M(M-1)}(q^{2M}\rho;q^{2})_{\infty} \; 
  \ge  y^{M}q^{M(M-1)}(\rho;q^{2})_{\infty}.
\end{gathered}
\end{eqnarray*}
 Moreover it is also clear that
\[
\left|z^{-\log\left(qx\right)/\log q^{2}}\right|\le\left|z\right|^{-\log\left(qx\right)/\log q^{2}}\le\rho^{-\log\left(qx\right)/\log q^{2}},
\]
 which shows that  
\begin{eqnarray*}
\left|\frac{1}{2\pi i}\int_{\rho+iq^{-2M-1}}^{\rho+i\infty}f(z)dz\right| & \le & 
 \frac{\rho^{-\log\left(qx\right)/\log q^{2}}q^{-M\left(M-1\right)}}{2\pi\left(q\sqrt{\rho};q\right)_{\infty}\rho;q^{2})_{\infty}}\int_{q^{-2M-1}}^{\infty}y^{-M}dy\\
 & = & \frac{\rho^{-\log\left(qx\right)/\log q^{2}}q^{\left(M-1\right)^{2}}}{2\pi\left(M-1\right)
   \left(q\sqrt{\rho};q\right)_{\infty}\rho;q^{2})_{\infty}}. 
\end{eqnarray*}
 Similarly we derive the estimate 
\[
\left|\frac{1}{2\pi i}\int_{\rho-i\infty}^{\rho-iq^{-2M-1}}f(z)dz\right|\le
  \frac{\rho^{-\log\left(qx\right)/\log q^{2}}q^{\left(M-1\right)^{2}}}{2\pi\left(M-1\right)
    \left(q\sqrt{\rho};q\right)_{\infty}\rho;q^{2})_{\infty}}.
\]
 Now let $M\to\infty$ to get 
\[
\frac{1}{2\pi i}\int_{\rho-i\infty}^{\rho+i\infty}\frac{z^{-\log\left(qx\right)/\log q^{2}}dz}
  {(z;q^{2})_{\infty}(-qz^{-1/2};q)_{\infty}}=\frac{1}{(q,-q,-q;q)_{\infty}}\sum_{n=0}^{\infty}
    \frac{q^{n^{2}}\left(-x\right)^{n}}{(q;q)_{n}},
\]
 which is \eqref{eq:ramanujan}.
 \end{proof}
 
 \subsection{Orthogonal Polynomials} 
\begin{thm}
\label{thm:Stieltjes--Wigert}  Let $\left|\arg(x)\right|<\pi$ and $0<q<1$
we have 
\begin{equation}
\frac{S_{n}\left(x;q\right)}{(-q;q)_{\infty}^{2}}=\frac{1}{2\pi i}
  \int_{\rho-i\infty}^{\rho+i\infty}\frac{(q^{n+1}z^{1/2};q)_{\infty}z^{-\log\left(qx\right)/\log q^{2}}dz}{(z;q^{2})_{\infty}(-qz^{-1/2};q)_{\infty}},\label{eq:Stieltjes--Wigert}
\end{equation}
 where $0<\rho<1$ and we take the principal branch of $z^{-1/2}$with
$\Re(z^{1/2})>0$ for $\left|\arg(z)\right|<\pi$.\end{thm}
\begin{proof}
Again we notice that both sides  of \eqref{eq:Stieltjes--Wigert} are
entire functions in $x$, we need only show it holds for $0<x<1$
and then get \eqref{eq:Stieltjes--Wigert} in full generality by applying
the analytic continuation. Given an nonnegative integer $n$, let
\[
g(z)=\frac{(q^{n+1}z^{1/2};q)_{\infty}
  z^{-\log\left(qx\right)/\log q^{2}}}{(z;q^{2})_{\infty}(-qz^{-1/2};q)_{\infty}},
\]
 for $q\in(0,1)$ and $1>x>0$, it is meromorphic in the right half
plane with simple poles 
\[
z_{k}=q^{-2k},\quad k=0,1,\dots,n
\]
 with residues 
\[
\frac{-1}{(-q,-q;q)_{\infty}}\frac{q^{k^{2}}\left(-x\right)^{k}}{(q;q)_{k}(q;q)_{n-k}}.
\]
 Let us take the same contour with $M\ge n$ as in the proof of \ref{thm:ramanujan}
to obtain 
\begin{eqnarray*}
\begin{gathered}
\frac{1}{2\pi i}\int_{\rho-i\infty}^{\rho+i\infty}\frac{(q^{n+1}z^{1/2};q)_{\infty}z^{-\log\left(qx\right)/\log q^{2}}dz}{(z;q^{2})_{\infty}(-qz^{-1/2};q)_{\infty}}  =  \frac{S_{n}\left(x;q\right)}{(-q;q)_{\infty}^{2}} 
\qquad \qquad  \qquad  \qquad \\
+\frac{1}{2\pi i}\int_{C(M,1)}g(z)dz.
  +  \frac{1}{2\pi i}\int_{\rho+iq^{-2M-1}}^{\rho+i\infty}g(z)dz
  +  \frac{1}{2\pi i}\int_{\rho-i\infty}^{\rho-iq^{-2M-1}}g(z)dz.
\end{gathered}
\end{eqnarray*}
 On $C(M,1)$ we have 
\[
\left|(q^{n+1}z^{1/2};q)_{\infty}\right|\le q^{-M^{2}/2}\left(-q^{1/2};q\right)_{M}\left(-q^{1/2};q\right)_{\infty}
\]
 and 
\[
\left|\frac{1}{2\pi i}\int_{C(M,1)}g(z)dz\right|\le\frac{q^{M^{2}/2}\left(qx\right)^{M+1/2}\left(-q^{1/2},q^{3/2};q\right)_{M}\left(-q^{1/2};q\right)_{\infty}}{2(q^{3/2};q)_{\infty}(q,q^{2})_{\infty}^{2}}.
\]
 On the vertical line segment 
\[
z=\rho+iy,\quad y\ge q^{-2M-1},
\]
 we have 
\[
\left|\frac{(q^{n+1}z^{1/2};q)_{\infty}z^{-\log\left(qx\right)/\log q^{2}}}{(z;q^{2})_{\infty}(-qz^{-1/2};q)_{\infty}}\right|\le\frac{\left|z\right|^{-\log\left(qx\right)/\log q^{2}}}{\left|\left(z^{1/2};q\right)_{n+1}(-z^{1/2},-qz^{-1/2};q)_{\infty}\right|}.
\]
 Applying  $
z^{1/2}=\sqrt[4]{\rho^{2}+y^{2}}\exp\left(\frac{i}{2}\arctan\frac{y}{\rho}\right)$,
we see that
\begin{eqnarray*}
\left|1-z^{1/2}q^{k}\right| & \ge & \left|1-q^{k}\sqrt[4]{\rho^{2}+y^{2}}\cos\left(\frac{1}{2}\arctan\frac{y}{\rho}\right)\right|\ge\frac{q^{k}y^{1/2}}{2^{3/2}}
\end{eqnarray*}
 and 
\[
\left|\left(z^{1/2};q\right)_{n+1}\right|\ge\frac{q^{n(n+1)/2}y^{\left(n+1\right)/2}}{2^{3n/2}}
\]
 for $M\ggg n$.  Applying 
\[
\left|1+z^{1/2}q^{k}\right|^{2}=1+q^{2k}\left(\rho^{2}+y^{2}\right)^{1/2}+2q^{k}\sqrt[4]{\rho^{2}+y^{2}}\cos\left(\frac{1}{2}\arctan\frac{y}{\rho}\right)
\]
  we find 
\begin{eqnarray*}
\left|(-q^{M}z^{1/2};q)_{\infty}\right| & \ge & 1,\quad\left|(-z^{1/2};q)_{M}\right|\ge y^{M/2}q^{M\left(M-1\right)/2}
\end{eqnarray*}
 and 
\begin{eqnarray*}
\left|(-z^{1/2};q)_{\infty}\right| & = & \left|(-z^{1/2};q)_{M}\right|\cdot\left|(-q^{M}z^{1/2};q)_{\infty}\right|\ge y^{M/2}q^{M\left(M-1\right)/2}.
\end{eqnarray*}
 It is also clear that   
\begin{eqnarray}
\notag
\left|(-qz^{-1/2};q)_{\infty}\right|\ge\left(q\sqrt{\rho};q\right)_{\infty}, \quad \textup{
 and} \quad  
\left|z^{-\log\left(qx\right)/\log q^{2}}\right|\le\rho^{-\log\left(qx\right)/\log q^{2}},
\end{eqnarray}
whence, 
\[
\left|\frac{(q^{n+1}z^{1/2};q)_{\infty}z^{-\log\left(qx\right)/\log q^{2}}}{(z;q^{2})_{\infty}(-qz^{-1/2};q)_{\infty}}\right|\le\frac{2^{3n/2}\rho^{-\log\left(qx\right)/\log q^{2}}}{y^{\left(M+n+1\right)/2}q^{M\left(M-1\right)/2+n(n+1)/2}\left(q\sqrt{\rho};q\right)_{\infty}}.
\]
 Thus, 
\[
\left|\frac{1}{2\pi i}\int_{\rho+iq^{-2M-1}}^{\rho+i\infty}g(z)dz\right|\le\frac{2^{3n/2}\rho^{-\log\left(qx\right)/\log q^{2}}q^{\left(2nM+5M-2n^{2}-n+1\right)/4}}{2\pi\left(M+n-1\right)\left(q\sqrt{\rho};q\right)_{\infty}}. 
\]
 Similarly we establish the following estimate  
\[
\left|\frac{1}{2\pi i}\int_{\rho-i\infty}^{\rho-iq^{-2M-1}}g(z)dz\right|\le\frac{2^{3n/2}\rho^{-\log\left(qx\right)/\log q^{2}}q^{\left(2nM+5M-2n^{2}-n+1\right)/4}}{2\pi\left(M+n-1\right)\left(q\sqrt{\rho};q\right)_{\infty}},.
\]
Now  \eqref{eq:Stieltjes--Wigert} is obtained by letting $M\to\infty$.\end{proof}
\begin{thm}
\label{thm:q-Laguerre}Let $\left|\arg(x)\right|<\pi$, $\nu>-1$
and $0<q<1$ we have 
\begin{equation}
\frac{\left(q^{\nu+1+n};q\right)_{\infty}L_{n}^{(\nu)}(xq^{-\nu};q)}{(-q;q)_{\infty}^{2}}=\frac{1}{2\pi i}\int_{\rho-i\infty}^{\rho+i\infty}\frac{(q^{n+1}z^{1/2},q^{\nu+1}z^{-1/2};q)_{\infty}dz}{(z;q^{2})_{\infty}(-qz^{-1/2};q)_{\infty}z^{\log\left(qx\right)/\log q^{2}}},\label{eq:q-Laguerre}
\end{equation}
 where $0<\rho<1$ and we take the principle branch of $z^{-1/2}$with
$\Re(z^{1/2})>0$ for $\left|\arg(z)\right|<\pi$.
\end{thm}
\begin{proof}
For any nonnegative integer $n$, let 
\[
h(z)=\frac{(q^{n+1}z^{1/2},q^{\nu+1}z^{-1/2};q)_{\infty}z^{-\log\left(qx\right)/\log q^{2}}}{(z;q^{2})_{\infty}(-qz^{-1/2};q)_{\infty}},
\]
 then it is meromorphic in the right half plane with essential singularity
$z=\infty$ and simple poles 
\[
q^{-2k}\quad k=0,1,\dots n
\]
 with residues 
\[
-\frac{q^{k^{2}}\left(-x\right)^{k}}{(q,q^{\nu+1};q)_{k}\left(q;q\right)_{n-k}}\frac{\left(q^{\nu+1};q\right)_{\infty}}{\left(-q;q\right)_{\infty}^{2}}.
\]
 Let $M$ be a large positive integer and the contour $C(M)$ as defined
in the proof of Theorem \ref{thm:Stieltjes--Wigert}, then we have,
\begin{eqnarray*}
\begin{gathered}
\frac{1}{2\pi i}\int_{\rho-i\infty}^{\rho+i\infty}\frac{(q^{n+1}z^{1/2},q^{\nu+1}z^{-1/2};q)_{\infty}dz}{(z;q^{2})_{\infty}(-qz^{-1/2};q)_{\infty}z^{\log\left(qx\right)/\log q^{2}}}  \qquad \qquad \qquad \\ 
=  \frac{\left(q^{\nu+1+n};q\right)_{\infty}L_{n}^{(\nu)}(xq^{-\nu};q)}{(-q;q)_{\infty}^{2}}
  +  \frac{1}{2\pi i}\int_{C(M,1)}h(z)dz\\
\qquad   +  \frac{1}{2\pi i}\int_{\rho+iq^{-2M-1}}^{\rho+i\infty}h(z)dz 
  +  \frac{1}{2\pi i}\int_{\rho-i\infty}^{\rho-iq^{-2M-1}}h(z)dz.
  \end{gathered}
\end{eqnarray*}
 Noticing on the semi-circle $C(M,1)$ and the two vertical line segments
we have the estimate 
\[
\left|(q^{\nu+1}z^{-1/2};q)_{\infty}\right|\le\left(-q^{\nu+M+3/2};q\right)_{\infty}
 \le\left(-q^{1/2};q\right)_{\infty}
\]
 and the rest of the proof similar to the proof for Theorem \ref{thm:Stieltjes--Wigert}
\end{proof}
 One can then use  Theorem \ref{thm:Stieltjes--Wigert} and \eqref{eqqHasSW}   to  prove the following corollary
\begin{cor}
 \label{cor:Ismail-Masson} Let $\left|\Im(\xi)\right|<\frac{\pi}{2}$
and $0<q<1$ we have 
\begin{equation}
\frac{e^{-n\xi}h_{n}\left(\sinh\xi\vert q\right)}{(-q;q)_{\infty}^{2}\left(q;q\right)_{n}}
 =\frac{1}{2\pi i}
 \int_{\rho-i\infty}^{\rho+i\infty}\frac{(q^{n+1}z^{1/2};q)_{\infty}z^{\xi/\log q+\left(n-1\right)/2}dz}                      {(z;q^{2})_{\infty}(-qz^{-1/2};q)_{\infty}},\label{eq:Ismail-Masson}
\end{equation}
 where $0<\rho<1$ and we take the principle branch of $z^{-1/2}$with
$\Re(z^{1/2})>0$ for $\left|\arg(z)\right|<\pi$. \end{cor}

\subsection{Transcendental Functions} 

Given parameters 
\begin{equation}
\alpha_{1},\dots,\alpha_{s},\beta_{1},\dots,\beta_{t}\in\mathbb{C}\label{eq:Parameters}
\end{equation}
 with 
\begin{equation}
|\alpha_{1}| < 1, |\alpha_{s}|< 1, |\beta_{1})| ,\dots, |\beta_{t})| < 1, 
\label{eq:ParameterConditions}
\end{equation}
and $\ell>0$, we define an infinite series with weight $\ell$
and 
\begin{eqnarray}
\label{eqmfunction}
m\left(\begin{array}{cc}
\begin{array}{c}
\alpha_{1},\dots,\alpha_{s}\\
\beta_{1},\dots,\beta_{t}
\end{array} \bigg| q, l,z \end{array}\right)=\sum_{k=0}^{\infty}
 \frac{(\alpha_{1}, \dots  \alpha_{s}; q)_{k}
q^{\ell k^{2}}\left(-z\right)^{k}}{(q,\beta_{1}, \dots, \beta_{t};q)_{k}}.
\end{eqnarray}
It must be noted this function is not new and can be written in terms of ${}_r\phi_s$ 
 with additional zero entries, if necessary,  but we prefer to use this more convenient notation.  

\begin{thm}
\label{thm:ConfluentHalf} Given two sets of parameters as in \eqref{eq:Parameters}
and \eqref{eq:ParameterConditions}, let $\rho$ be a positive number
such that 
\[
\max\left\{ |\alpha_{1}|, \dots,  |\alpha_{s}|, |\beta_{1}|,\dots,  |\beta_{t}| \right\} <\rho<1.
\]
 Then we have 
\begin{eqnarray}
\label{eq:ConfluentHalf}
 \begin{gathered}   \frac{\left(\beta_{1},\dots, \beta_{t};q\right)_{\infty}}{\left(q, \alpha_{1},\dots, \alpha_{s};q\right)_{\infty}}m\left(\begin{array}{cc}
\begin{array}{c}
\alpha_{1},\dots,\alpha_{s}\\
\beta_{1},\dots,\beta_{t}
\end{array} \bigg| q,  \frac{1}{2},x\end{array}\right)    \\
  =  \frac{1}{2\pi i}\int_{\rho-i\infty}^{\rho+i\infty}\frac{\prod_{j=1}^{s}(\beta_{j}/z;q)_{\infty}z^{-\log\left(\sqrt{q}x\right)/\log q}}{(z;q)_{\infty}\prod_{k=1}^{t}( \alpha_{k}/z;q)_{\infty}}dz  
\end{gathered}
\end{eqnarray}
 for $\left|\arg(x)\right|<\pi$. \end{thm}
\begin{proof}
Observe that the function $F(z)$  defined by 
\[
F(z)=\frac{\prod_{j=1}^{s}(\beta_{j}/z;q)_{\infty}z^{-\log\left(\sqrt{q}x\right)/\log q}}{(z;q)_{\infty}\prod_{k=1}^{t}(\alpha_{k}/z;q)_{\infty}}
\]
  is meromorphic in the right half plane with simple poles
\[
z=q^{-n},\quad n=0,1,\dots
\]
 with residues 
\[
-\frac{\left(\beta_{1},\dots, \beta_{t} ; q\right)_{\infty}}{\left(q, \alpha_{1},\dots, \alpha_{s}; q\right)_{\infty}}\frac{\left( \alpha_{1},\dots, \alpha_{s} ; q\right)_{n} q^{n^{2}/2}\left(-x\right)^{n}}
{\left(q, \beta_{1},\dots, \beta_{t}; q\right)_{n}}
\]
 The rest of the proof is similar to the proof for Theorem \ref{thm:q-Laguerre} and will be omitted.
 \end{proof}
\begin{thm}
\label{thm:ConfluentOne} Given two sets of parameters as in \eqref{eq:Parameters}
and \eqref{eq:ParameterConditions}, let $\rho$ be a positive number
such that 
\[
\max\left\{ |\alpha_{1}| , \dots,  |\alpha_{s}|, |\beta_{1}|, \dots,  |\beta_{t}|\right\} <\rho<1.
\]
 Then we have 
\begin{eqnarray}
\label{eq:ConfluentOne}
 \begin{gathered}
  \frac{\left(\beta_{1},\dots, \beta_{t}; q\right)_{\infty}}{\left(q,-q,-q, \alpha_{1},\dots, 
 \alpha_{s}; q\right)_{\infty}}m\left(\begin{array}{cc}
\begin{array}{c}
\alpha_{1},\dots,\alpha_{s}\\
\beta_{1},\dots,\beta_{t}
\end{array} & \bigg|q,1,x\end{array}\right)   \\
  =  \frac{1}{2\pi i}\int_{\rho-i\infty}^{\rho+i\infty}\frac{\prod_{j=1}^{s}(\beta_{j}z^{-1/2};q)_{\infty}z^{-\log\left(qx\right)/\log q^{2}}}{(z;q^{2})_{\infty}(-qz^{-1/2};q)_{\infty}\prod_{k=1}^{t}(\alpha_{k}z^{-1/2};q)_{\infty}}dz  
  \end{gathered}
\end{eqnarray}
 for $\left|\arg(x)\right|<\pi$ and we take the principle branch
of $z^{-1/2}$with $\Re(z^{1/2})>0$ for $\left|\arg(z)\right|<\pi$. \end{thm}
\begin{proof}
Notice that 
\[
G(z)=\frac{\prod_{j=1}^{s}(\beta_{j}z^{-1/2};q)_{\infty}z^{-\log\left(qx\right)/\log q^{2}}}{(z;q^{2})_{\infty}
(-qz^{-1/2};q)_{\infty}\prod_{k=1}^{t}(\alpha_{k}z^{-1/2};q)_{\infty}}
\]
 has simple poles 
\[
q^{-2n},\quad n=0,1,\dots
\]
 with residues 
\[
-\frac{\left(\beta_{1} ,\dots, \beta_{t};q\right)_{n}}{\left(q,-q,-q, \alpha_{1},\dots, \alpha_{s};q\right)_{\infty}}\frac{\left( \alpha_{1},\dots, \alpha_{s};q\right)_{n}q^{n^{2}}\left(-x\right)^{n}}
{\left(q, \beta_{1},\dots, \beta_{t};  q  \right)_{n}},
\]   
 the rest of the proof is very similar to the proof for Theorem \ref{thm:ramanujan}. 
\end{proof}
Given a nonnegative integer $n$, the following polynomials 
\[
p_{n}\left(\begin{array}{cc}
\begin{array}{c}
\alpha_{1},\dots,\alpha_{s}\\
\beta_{1},\dots,\beta_{t}
\end{array} \bigg| q,  z \end{array}\right)=\sum_{k=0}^{n}
\frac{(q^{\alpha_{1}},\dots \alpha_{s};q)_{k}q^{k^{2}}\left(-z\right)^{k}}{(q, \beta_{1},\dots,\beta_{s};q)_{k}\left(q;q\right)_{n-k}},
\]
are defined and studied in \cite{Ismbook}, notice here we used a different
normalization has a different notation. Then we have the following
result:
\begin{thm}
\label{thm:Confluent polynomials}  Given a nonnegative integer $n$
and under the conditions as in Theorem \ref{thm:ConfluentOne} we
have 
\begin{eqnarray}
\label{eq4.10}
\bg
  \frac{( \beta_{1},\dots,  \beta_{s};q)_{\infty}}{(-q,-q, \alpha_{1},\dots \alpha_{s};q)_{\infty}}
 p_{n}\left(\begin{array}{cc}
\begin{array}{c}
\alpha_{1},\dots,\alpha_{s}\\
\beta_{1},\dots,\beta_{t}
\end{array} \bigg| q,  x\end{array}\right)\\
  =  \frac{1}{2\pi i}\int_{\rho-i\infty}^{\rho+i\infty}\frac{(q^{n+1}z^{1/2}, \beta_{1} z^{-1/2},\dots, 
 \beta_{t} z^{-1/2};q)_{\infty}z^{-\log\left(qx\right)/\log q^{2}}dz}{(z;q^{2})_{\infty}(-qz^{-1/2}, \alpha_{1}
 z^{-1/2},\dots, \alpha_{s} z^{-1/2};q)_{\infty}}
 \eg
\end{eqnarray}
 for $\left|\arg(x)\right|<\pi$ and we take the principle branch
of $z^{-1/2}$with $\Re(z^{1/2})>0$ for $\left|\arg(z)\right|<\pi$.\end{thm}
\begin{proof}
Notice that under the conditions, the function 
\[
f(z)=\frac{(q^{n+1}z^{1/2}, \beta_{1} z^{-1/2},\dots, \beta_{t}  z^{-1/2};q)_{\infty} 
z^{-\log\left(qx\right)/\log q^{2}}}{(z;q^{2})_{\infty}(-qz^{-1/2},  \alpha_{1}z^{-1/2},\dots, \alpha_{s}z^{-1/2};q)_{\infty}}
\]
 is meromorphic on the proper right half plane with simple poles at
\[
q^{-2k},\quad k=0,1,\dots n
\]
 with residues 
\[
-\frac{( \beta_{1}, \dots,  \beta_{s};q)_{\infty}}{(-q,-q, \alpha_{1}, \dots,  \alpha_{s} ; q)_{\infty}} 
\frac{(\alpha_{1},  \dots,  \alpha_{s}; q)_{k}q^{k^{2}}\left(-x\right)^{k}}{(q, \beta_{1}, \dots, 
\beta_{s};q)_{k}\left(q;q\right)_{n-k}}
\]
 and the rest of the proof is similar to the proof for Theorem \ref{thm:q-Laguerre}. \end{proof}

One question we address is the expansion of a function in a series of polynomials. The following simple remark is useful. 
\begin{rem}
Let $\{p_n(x)\}$ be a sequence of polynomials where $p_n(x)$ is of exact degree $n$. If 
\begin{eqnarray}
x^n = \sum_{k=0}^n c_{n,k} p_k(x)
\label{eqconncntopn}
\end{eqnarray}
then 
\begin{eqnarray}
\Sum a_n x^n = \sum_{k=0}^\infty \left(\sum_{n=k}^\infty c_{n,k}a_n \right) \; p_k(x),
\label{eqexp}
\end{eqnarray}
provided that we can interchange the $k$ and $n$ sums. 
\end{rem}

\section{Fourier Transform Pairs}
 
\subsection{Preliminaries} 
Let $0 < q < 1$ and $m>0$. In this section  we apply the Fourier transform formula \eqref{eq2ndIR}
in the form 
\begin{eqnarray}
\label{eq:fourier1}
\bg
q^{mk^{2}}  =  \frac{1}{\sqrt{2\pi}}\int_{-\infty}^{\infty}
\exp\left(-\frac{x^{2}}{2}+ikx\sqrt{\log q^{-2m}}\right)dx
\\
  =  \frac{1}{\sqrt{\pi\log q^{-4m}}}\int_{-\infty}^{\infty}\exp\left(\frac{x^{2}}{\log q^{4m}}+ikx\right)dx\eg 
\end{eqnarray}
to evaluate some integrals related to $q$-special functions.

One useful identity is the following $q$-analogue of $(1-1)^n$. 
\begin{eqnarray}
\label{eqqbinom1} 
\sum_{k=0}^n \qbi{n}{k}q  (-1)^k = \begin{cases}  0 & \textup{if}\;  n\;  \textup{is odd} \\
\frac{(q;q)_{n}}{(q^2;q^2)_{n/2}}, & \textup{if}\;  n\;  \textup{is even}.
\end{cases}
\end{eqnarray}
To prove \eqref{eqqbinom1} denote the left-hand side by $f_n$. Clearly 
\begin{eqnarray}
\notag
\Sum f_n \frac{t^n}{(q;q)_n} = \sum_{\infty > n \ge k \ge 0} \frac{(-t)^k}{(q;q)_k}
\frac{t^{n-k}}{(q;q)_{n-k}} 
= \frac{1}{(-t;q)_\infty} \; \frac{1}{(t;q)_\infty} = \frac{1}{(t^2; q^2)_\infty},
 \end{eqnarray}
and the result follows. It  must be noted that \eqref{eqqbinom1} is essentially the evaluation of a 
$q$-Hermite polynomial 
$H_n(x|q)$ at $x=0$, see \eqref{eq:1.10}. 

Some  of the proofs needs the following evaluation 
\begin{eqnarray}
\sum_{n=0}^s \frac{(-1)^n q^{n(n-s)}}{(q;q)_n(q;q)_{s-n}} 
 = \begin{cases} 0 & \textup{if} \; n \;  \textup{is odd}\\
             \frac{(-1)^{s/2}q^{-s^2/4}}{(q^2; q^2)_{s/2}}   & \textup{if} \; n \; \textup{is even}
                \end{cases}
\label{eqqbinom2}
\end{eqnarray}
To prove \eqref{eqqbinom2} we let 
 \begin{eqnarray}
 \notag
 g_s = q^{s^2/2} \sum_{n=0}^s \frac{(-1)^n q^{n(n-s)}}{(q;q)_n(q;q)_{s-n}} 
\end{eqnarray}
It is clear that 
\begin{eqnarray}
\sum_{s=0}^\infty g_s t^s = \sum_{s=0}^\infty t^s 
\sum_{n=0}^s \frac{(-1)^nq^{n^2/2}}{(q;q)_n} \frac{q^{(s-n)^2/2}}{(q;q)_{s-n}}
 = (-\sqrt{q}t, \sqrt{q}t;q)_\infty = (qt^2;q^2)_\infty.  
\notag
\end{eqnarray}
This proves \eqref{eqqbinom2}. Also note that \eqref{eqqbinom2} is equivalent to the evaluation 
of $h_n(0|q)$, see \eqref{eqqH}. 

The last identity we need is 
\begin{eqnarray}
\label{eqqbinom3}
\sum_{k=0}^n \frac{q^{k/2}}{(q;q)_k(q;q)_{n-k}} = \frac{1}{(\sqrt{q}; \sqrt{q})_n}. 
\end{eqnarray}
The proof is via generating functions. Again \eqref{eqqbinom3} is the evaluation of $H_n((q^{1/2}+q^{-1/2})/2|q)$. 
  
  We next state the Fourier transform version of Parseval's theorem. 
  \begin{thm}\label{Parseval}
  If $f$ and $g$ belong to $L_1(\R)$ and 
  \bea
  F(x) = \frac{1}{\sqrt{2\pi}}\int_\R f(t) e^{ixt} \, dt, \quad G(x) = \frac{1}{\sqrt{2\pi}}\int_\R g(t) e^{ixt} \, dt, 
  \notag
  \eea
   then  
  \bea
  \int_\R F(x)G(x) = \int_\R f(t)g(-t) dt.
  \label{eqPar}
  \eea
  \end{thm}
This is Theorem 35 on page  54 of Titchmarsh's book \cite{Tit}. 

\subsection{The $q$-Exponential Functions $E_{q}\left(z\right)$ and 
\textmd{$e_{q}\left(z\right)$} }
\begin{thm}
\label{thm:Eq-eq}For $\left|z\right|<1$, the Fourier transform pairs
for $E_{q}\left(z\right)$ are
\begin{eqnarray}
q^{\alpha^{2}/2}E_{q}\left(zq^{\alpha+1/2}\right) & = & \sqrt{\frac{\log q^{-1}}{2\pi}}
\int_{-\infty}^{\infty}e_{q}\left(zq^{iy}\right)q^{y^{2}/2+i\alpha y}dy,\label{eq:fourier2}\\
e_{q}\left(zq^{iy}\right)q^{y^{2}/2} & = & \sqrt{\frac{\log q^{-1}}{2\pi}}
\int_{-\infty}^{\infty}q^{\alpha^{2}/2}E_{q}\left(zq^{\alpha+1/2}\right)q^{-i\alpha y}
d\alpha,\label{eq:fourier3}
\end{eqnarray}
 the modulus double and half formulas for $E_{q}\left(z\right)$ are,
\begin{eqnarray}
E_{q^{2}}\left(qz^{2}\right) & = & \sqrt{\frac{\log q^{-1}}{\pi}}\int_{-\infty}^{\infty}q^{\alpha^{2}}
E_{q}\left(-zq^{1/2+\alpha}\right)E_{q}\left(zq^{1/2-\alpha}\right)d\alpha,\label{eq:fourier4}\\
E_{q}\left(zq^{1/2}\right) & = & \sqrt{\frac{\log q^{-2}}{\pi}}\int_{-\infty}^{\infty}q^{2\alpha^{2}}
E_{q^{2}}\left(zq^{1-2\alpha}\right)E_{q^{2}}\left(zq^{2+2\alpha}\right)d\alpha.
\label{eq:fourier5}
\end{eqnarray}
 the modulus double and half formulas for $e_{q}\left(z\right)$ 
\begin{eqnarray}
e_{q^{2}}\left(-z^{2}\right) & = & \sqrt{\frac{\log q^{-1}}{\pi}}
\int_{-\infty}^{\infty}e_{q}\left(zq^{iy}\right)e_{q}\left(-zq^{-iy}\right)q^{y^{2}}dy,
\label{eq:fourier6}\\
e_{q}\left(z\right) & = & \sqrt{\frac{\log q^{-1}}{2\pi}}\int_{-\infty}^{\infty}e_{q^{2}}
\left(zq^{iy-1/2}\right)e_{q^{2}}\left(zq^{-iy+1/2}\right)q^{y^{2}/2}dy,\label{eq:fourier7}
\end{eqnarray}
 where $\left|z\right|<1$. 
 \end{thm}
\begin{proof}
It is clear that 
\begin{eqnarray*}
q^{\alpha^{2}/2}\left(-zq^{\alpha+1/2};q\right)_{\infty} & = &
 \sum_{n=0}^{\infty}\frac{z^{n}}{(q;q)_{n}}q^{(n+\alpha)^{2}/2}.
\end{eqnarray*}
Now apply  \eqref{eq:fourier1}  and when  $\left|z\right|<1$ sum the series. The result is  
\begin{equation}
q^{\alpha^{2}/2}\left(-zq^{\alpha+1/2};q\right)_{\infty}=\frac{1}{\sqrt{2\pi\log q^{-1}}}\int_{-\infty}^{\infty}
\frac{\exp\left(\frac{x^{2}}{\log q^{2}}+i\alpha x\right)}{\left(ze^{ix};q\right)_{\infty}}dx,
\label{eq:fourier11}
\end{equation}
whose  inverse transform is 
\begin{equation}
\frac{\exp\left(\frac{x^{2}}{\log q^{2}}\right)}{\left(ze^{ix};q\right)_{\infty}}=\sqrt{\frac{\log q^{-1}}{2\pi}}
\int_{-\infty}^{\infty}q^{\alpha^{2}/2}(-zq^{\alpha+1/2};q)_{\infty}\exp\left(-i\alpha x\right)d\alpha. 
 \label{eq:fourier12}
\end{equation}
These are \eqref{eq:fourier2}-\eqref{eq:fourier3}.  

 We now start with \eqref{eq:fourier11}, then  replace $q$ by $q^2$. Then  we make  the 
 choices $(z,\alpha ) \to (\overline{z},\alpha )$,   and $(z,\alpha ) \to ( z,\alpha +1)$, respectively.  We then  
 use Parseval's formula \eqref{eqPar}  and conclude that  
\begin{align}
\int_{-\infty}^{\infty}q^{2\alpha^{2}}\left(zq^{1-2\alpha},zq^{2+2\alpha};q^{2}\right)_{\infty}d\alpha 
& =\frac{1}{\log q^{-2}}\int_{-\infty}^{\infty}
\frac{\exp\left(\frac{x^{2}}{\log q^{2}}\right)dx}{\left(-ze^{ix};q\right)_{\infty}}.
\label{eq:fourier25}
\end{align}
The right-hand side is $(zq^{1/2};q)_\infty$ by \eqref{eq:fourier11}. Thus  
\begin{equation}
\left(zq^{1/2};q\right)_{\infty}=\sqrt{\frac{\log q^{-2}}{\pi}}\int_{-\infty}^{\infty}q^{2\alpha^{2}}\left(zq^{1-2\alpha},zq^{2+2\alpha};q^{2}\right)_{\infty}d\alpha,\label{eq:fourier26}
\end{equation}
which is \eqref{eq:fourier5}.

Similarly, from \eqref{eq:fourier12} we obtain \eqref{eq:fourier6}
and \eqref{eq:fourier7}.  
\end{proof}
 
 Later we shall use the following 
 more symmetric form, 
\begin{equation}
q^{\alpha^{2}/2}\left(-zq^{\alpha+1/2};q\right)_{\infty}=\sqrt{\frac{\log q^{-1}}{2\pi}}
\int_{-\infty}^{\infty}\frac{q^{y^{2}/2+i\alpha y}}{\left(zq^{iy};q\right)_{\infty}}dy\label{eq:fourier13}
\end{equation}
and its inverse 
\begin{equation}
\frac{q^{y^{2}/2}}{\left(zq^{iy};q\right)_{\infty}}=\sqrt{\frac{\log q^{-1}}{2\pi}}\int_{-\infty}^{\infty}
q^{\alpha^{2}/2-i\alpha y}\left(-zq^{\alpha+1/2};q\right)_{\infty}d\alpha. 
\label{eq:fourier14}
\end{equation}

Observe  that  intuitively \eqref{eq:fourier6} is the formal change of variables $q \to 1/q, 
\alpha \to i\alpha$. A similar remark applies to \eqref{eq:fourier7}. 

\begin{proof}[Alternate Proofs]  We now give direct proofs of the identities in Theorem 
\ref{thm:Eq-eq}  using series expansions. 
 Equation  \eqref{eq:fourier11} can  be proved by expanding $1/(ze^{ix};q)_\infty$ using 
 \eqref{eq:1.2},  and evaluating the resulting integral.  Moreover \eqref{eq:fourier4} and 
 \eqref{eq:fourier5} can also be proved by expanding the functions $E_q(\mp zq^{1/2\pm \alpha })$ using 
 \eqref{eq:1.3} then complete the squares in the power of $q$ and evaluate the integral over 
 $\alpha $,  which a gaussian integral.  This reduces the  right-hand side of \eqref{eq:fourier4} to 
 \begin{eqnarray}
 \notag
 \bg
 \sum_{m,n=0}^\infty \frac{(-1)^n \, z^{m+n}}{(q;q)_m(q;q)_n} q^{(m^2+n^2)/2} q^{-(m-n)^2/4} = 
 \sum_{s=0}^\infty \frac{z^s}{(q;q)_s} q^{s^2/4}\; \sum_{n=0}^s \qbi{s}{n}q (-1)^n \\
 =  \sum_{s=0}^\infty \frac{z^{2s}q^{s^2}}{(q;q)_{2s}} \; \frac{(q;q)_{2s}}{(q^2;q^2)_s},
 \eg 
 \end{eqnarray}
where we used \eqref{eqqbinom1} in the last step. The above series sums to $E_{q^2}(qz^2)$ and \eqref{eq:fourier4} follows.  The integral evaluation 
\eqref{eq:fourier5} follows along the same lines. 

The direct proof of \eqref{eq:fourier6} is slightly longer.  Again expand the functions 
$e_q(\pm z q^{\pm iy})$ and perform the $y$ integral.  Thus the right hand side of 
\eqref{eq:fourier6} is 
\begin{eqnarray}
\sum_{m,n=0}^\infty \frac{(-1)^n z^{m+n}}{(q;q)_m(q;q)_n} q^{(m-n)^2/4}. 
\notag
\end{eqnarray}
Let $s = m+n$ and eliminate $m$ to see that the above sum is 
\begin{eqnarray}
\label{eqsandn}
\sum_{s=0}^\infty z^s  q^{s^2/4} \sum_{n=0}^s 
 \frac{ q^{-n(s-n)}(-1)^n}{(q;q)_n(q;q)_{s-n}} = \sum_{s=0}^\infty \frac{(-z^2)^s}{(q^2;q^2)_s}
\end{eqnarray}
where we used \eqref{eqqbinom2}.  
This establishes  \eqref{eq:fourier6}.  Next we consider 
\eqref{eq:fourier7}.  Its right-hand side is 
\begin{eqnarray}
\notag
\sum_{m,n=0}^\infty \frac{z^{m+n}}{(q^2;q^2)_m(q^2;q^2)_n}\, q^{(n-m)/2} \, q^{(m-n)^2/2}
= \sum_{s=0}^\infty z^sq^{-\binom{s}{2}} \sum_{n=0}^s \frac{q^{n^2} }{(q^2;q^2)_n}
 \frac{q^{(s-n)(s-n-1)}}{(q^2;q^2)_{s-n}}
\end{eqnarray}
Again we use 
\begin{eqnarray}
\notag
\sum_{s=0}^\infty t^s   \sum_{n=0}^s \frac{q^{n^2} }{(q^2;q^2)_n}
\frac{q^{(s-n)(s-n-1)}}{(q^2;q^2)_{s-n}}
=  (-tq; -t;q^2) = (-t;q)_\infty =  \sum_{s=0}^\infty \frac{t^s}{(q;q)_s} q^{\binom{s}{2}}. 
\end{eqnarray}
and conclude that 
\begin{eqnarray}
\notag
q^{-\binom{s}{2}} \sum_{n=0}^s \frac{q^{n^2} }{(q^2;q^2)_n}
\frac{q^{(s-n)(s-n-1)}}{(q^2;q^2)_{s-n}} = \frac{1}{(q;q)_s}.
\end{eqnarray}
 This establishes  \eqref{eq:fourier6}. 
 \end{proof}

\subsection{The $q$-Exponential Function  $\mathcal{E}_{q}\left(x;t\right)$}

\begin{thm}
\label{thm:izE}For $\left|t\right|<1$, $\theta\in\left[0,\pi\right]$
and $x=\cos\theta$, the Fourier pair for $\mathcal{E}_{q}\left(x;t\right)$
is
\begin{eqnarray}
\label{eq:fourier39}
\bg
\qquad q^{\alpha^{2}/4}E_{q^{2}}\left(-t^{2}q^{1+\alpha}\right)\mathcal{E}_{q}\left(x;tq^{\alpha/2}\right) 
=\sqrt{\frac{\log q^{-1}}{\pi}}\int_{-\infty}^{\infty}q^{x^{2}
+i\alpha x}e_{q}\left(tq^{ix}e^{i\theta}\right)e_{q}\left(tq^{ix}e^{-i\theta}\right)dx,
\eg
\end{eqnarray}
\begin{eqnarray}
\label{eq:fourier40}
\bg
\qquad q^{x^{2}}e_{q}\left(tq^{ix}e^{i\theta}\right)e_{q}\left(tq^{ix}e^{-i\theta}\right) 
=\sqrt{\frac{\log q^{-1}}{4\pi}}\int_{-\infty}^{\infty}q^{\alpha^{2}/4-i\alpha x}
E_{q^{2}}\left(-t^{2}q^{1+\alpha}\right)\mathcal{E}_{q}\left(x;tq^{\alpha/2}\right)d\alpha,
\eg
\end{eqnarray}
 and modulus double and half pair is
\begin{eqnarray}
\label{eq:fourier41}
\bg
\mathcal{E}_{q^{2}}\left(\cos2\theta;t^{2}\right)  =  \sqrt{\frac{\log q^{-1}}
{2\pi}}\int_{\infty}^{\infty}E_{q^{2}}\left(-t^{2}q^{1-\alpha}\right)
E_{q^{2}}\left(-t^{2}q^{1+\alpha}\right) \\
  \times   e_{q^{4}}\left(t^{4}q^{2}\right)
  \mathcal{E}_{q}\left(x;-tq^{-\alpha/2}\right)
  \mathcal{E}_{q}\left(x;tq^{\alpha/2}\right)q^{\alpha^{2}/2}d\alpha, 
  \eg
\end{eqnarray}
\begin{eqnarray}
\label{eq:fourier42}
\bg
\mathcal{E}_{q}\left(x;t\right) =  \sqrt{\frac{\log q^{-1}}{\pi}}\int_{-\infty}^{\infty}
E_{q^{4}}\left(-t^{2}q^{2+2\alpha}\right)E_{q^{4}}\left(-t^{2}q^{4-2\alpha}\right)  \\
  \times  e_{q^{2}}\left(t^{2}q\right)\mathcal{E}_{q^{2}}\left(x;tq^{\alpha}\right)
  \mathcal{E}_{q^{2}}\left(x;tq^{1-\alpha}\right)q^{\alpha^{2}}d\alpha.  
\eg
\end{eqnarray}
 \end{thm}
\begin{proof}
From \eqref{eq:1.11}, \eqref{eq:1.12} and \eqref{eq2ndIR}  
we get
\begin{eqnarray}
\label{eq:2.31}
\bg
q^{\alpha^{2}/4}\left(t^{2}q^{1+\alpha};q^{2}\right)_{\infty}\mathcal{E}_{q}
 \left(x;tq^{\alpha/2}\right)  \\
=  \sum_{n=0}^{\infty}\frac{q^{\left(n+\alpha\right)^{2}/4}t^{n}H_{n}\left(x\vert q\right)}
 {\left(q;q\right)_{n}} =  \sum_{n=0}^{\infty}\int_{-\infty}^{\infty}\frac{\exp\left(\frac{y^{2}}
  {\log q}+iy\alpha\right)\left(e^{iy}t\right)^{n}H_{n}\left(x\vert q\right)dy}
   {\sqrt{\pi\log q^{-1}}\left(q;q\right)_{n}}   \\
  =  \frac{1}{\sqrt{\pi\log q^{-1}}}\int_{-\infty}^{\infty}\frac{\exp\left(\frac{y^{2}}
    {\log q}+iy\alpha\right)dy}
     {\left(te^{i\left(y+\theta\right)},te^{i\left(y-\theta\right)};q\right)_{\infty}},  
\eg
\end{eqnarray}
that is
\begin{equation}
q^{\alpha^{2}/4}\left(t^{2}q^{1+\alpha};q^{2}\right)_{\infty}\mathcal{E}_{q}\left(x;tq^{\alpha/2}\right)=\frac{1}{\sqrt{2\pi\log q^{-1/2}}}\int_{-\infty}^{\infty}\frac{\exp\left(\frac{y^{2}}{\log q}+iy\alpha\right)dy}{\left(te^{i\left(y+\theta\right)},te^{i\left(y-\theta\right)};q\right)_{\infty}}.  
  \label{eq:2.32}
\end{equation}
and
\begin{equation}
\frac{\exp\left(\frac{y^{2}}{\log q}\right)}
 {\left(te^{i\left(y+\theta\right)},te^{i\left(y-\theta\right)};q\right)_{\infty}}=\sqrt{\frac{\log q^{-1/2}}{2\pi}}\int_{-\infty}^{\infty}q^{\alpha^{2}/4}\left(t^{2}q^{1+\alpha};q^{2}\right)_{\infty}
   \mathcal{E}_{q}\left(x;tq^{\alpha/2}\right)e^{-iy\alpha}d\alpha
    \label{eq:2.33}
\end{equation}
 for $\left|t\right|<1$ and $\theta\in\left(0,2\pi\right)$. From
\eqref{eq:2.32} we get
\[
\begin{aligned} & \int_{\infty}^{\infty}q^{\alpha^{2}/2}\left(t^{2}q^{1-\alpha},t^{2}q^{1+\alpha};q^{2}\right)_{\infty}\mathcal{E}_{q}\left(x;-tq^{-\alpha/2}\right)\mathcal{E}_{q}\left(x;tq^{\alpha/2}\right)d\alpha\\
 & =\int_{-\infty}^{\infty}\frac{\exp\left(\frac{2y^{2}}{\log q}\right)dy}{\left(t^{2}e^{2i\left(y+\theta\right)},t^{2}e^{2i\left(y-\theta\right)};q^{2}\right)_{\infty}\log q^{-1/2}}\\
 & =\int_{-\infty}^{\infty}\frac{\exp\left(\frac{y^{2}}{\log q^{2}}\right)dy}{\left(t^{2}e^{i\left(y+2\theta\right)},t^{2}e^{i\left(y-2\theta\right)};q^{2}\right)_{\infty}\log q^{-1}}\\
 & =\sqrt{\frac{2\pi}{\log q^{-1}}}\left(t^{4}q^{2};q^{4}\right)_{\infty}\mathcal{E}_{q^{2}}\left(\cos2\theta;t^{2}\right),
\end{aligned}
\]
that is,
\[
\mathcal{E}_{q^{2}}\left(\cos2\theta;t^{2}\right)=\sqrt{\frac{\log q^{-1}}{2\pi}}
\int_{\infty}^{\infty}\frac{\left(t^{2}q^{1-\alpha},t^{2}q^{1+\alpha};q^{2}\right)_{\infty}
 \mathcal{E}_{q}
\left(x;-tq^{-\alpha/2}\right)\mathcal{E}_{q}\left(x;tq^{\alpha/2}\right)d\alpha}
{\left(t^{4}q^{2};q^{4}\right)_{\infty}q^{-\alpha^{2}/2}}
\]
for $x=\cos\theta$, $\theta\in\mathbb{R}$ and $t\in\left(-q^{-1/2},q^{-1/2}\right)$.
Similarly,
\[
\begin{aligned} & \int_{-\infty}^{\infty}q^{\alpha^{2}}\left(t^{2}q^{2+2\alpha},t^{2}q^{4-2\alpha};q^{4}\right)_{\infty}\mathcal{E}_{q^{2}}\left(x;tq^{\alpha}\right)
\mathcal{E}_{q^{2}}\left(x;tq^{1-\alpha}\right)d\alpha\\
 & =\int_{-\infty}^{\infty}\frac{\exp\left(\frac{y^{2}}{\log q}\right)dy}{\left(te^{i\left(y+\theta\right)},te^{i\left(y-\theta\right)};q\right)_{\infty}\log q^{-1}}\\
 & =\sqrt{\frac{\pi}{\log q^{-1}}}\left(t^{2}q^{1};q^{2}\right)_{\infty}
 \mathcal{E}_{q}\left(x;t\right),
\end{aligned}
\]
that is,
\[
\mathcal{E}_{q}\left(x;t\right)=\int_{-\infty}^{\infty}\frac{\left(t^{2}q^{2+2\alpha},t^{2}q^{4-2\alpha};q^{4}\right)_{\infty}\mathcal{E}_{q^{2}}\left(x;tq^{\alpha}\right)
\mathcal{E}_{q^{2}}\left(x;tq^{1-\alpha}\right)d\alpha}{\sqrt{\pi/\log q^{-1}}\left(t^{2}q;q^{2}\right)_{\infty}q^{-\alpha^{2}}}
\]
 for $t,x\in\left(-1,1\right)$. 
 \end{proof}
 
 One can easily prove \eqref{eq:fourier39}  by expanding the $e_q$ functions in the integrand evaluate the integral then use \eqref{eq:1.10} and \eqref{eq:1.12}.  Now 
 \eqref{eq:fourier40} follows from the Fourier inversion formula. 
 
 \begin{proof}[Another Proof of \eqref{eq:fourier41}--\eqref{eq:fourier42}]To prove \eqref{eq:fourier41} expand $E_{q^2}(-t^2q^{1 \pm \al}) \mathcal{E}_{q}\left(x;t q^{\pm \al}\right)$ using 
  \eqref{eq:1.12}. Thus the right-hand side of  \eqref{eq:fourier41} equals 
 \begin{eqnarray}
 \bg
 \frac{1}{(t^4q^2;q^4)_\infty} \sum_{m,n=0}^\infty \frac{H_m(\cos \t)}{(q;q)_m}
  \frac{H_n(\cos \t)}{(q;q)_n}\, (-1)^m t^{m+n}\, q^{(m^2+n^2)/4}\, q^{-(n-m)^2/8} \\
  =  \frac{1}{(t^4q^2;q^4)_\infty} \sum_{s=0}^\infty t^s q^{s^2/8} \sum_{m=0}^s 
  \frac{H_m(\cos \t)}{(q;q)_m}
  \frac{H_{s-n}(\cos \t)}{(q;q)_{s-n}} \, (-1)^m. 
 \eg
 \notag
 \end{eqnarray}
 On the other hand 
 \begin{eqnarray}
 \notag
 \bg
 \sum_{s=0}^\infty  w^s \sum_{s=0}^\infty t^s q^{s^2/8} \sum_{m=0}^s 
  \frac{H_m(\cos \t|q)}{(q;q)_m}
  \frac{H_{s-n}(\cos \t|q)}{(q;q)_{s-n}} \, (-1)^m 
    = \frac{1}{(we^{i\t}, we^{-i\t}, -we^{i\t}, -we^{-i\t};q)_\infty} \\
     = \frac{1}{(w^2e^{2i\t}, w^2 e^{2i\t};q^2)_\infty} = \sum_{s=0}^\infty w^{2s}
       \frac{H_s(\cos 2\t)|q^2)}{(q^2;q^2)_s}. 
 \eg
 \end{eqnarray}
 This shows that the right-hand side of  \eqref{eq:fourier41} is 
 \begin{eqnarray}
 \notag
   \frac{1}{(t^4q^2;q^4)_\infty} \sum_{s=0}^\infty t^{2s} q^{s^2/2}\, 
     \frac{H_s(\cos 2\t)|q^2)}{(q^2;q^2)_s}.
 \end{eqnarray}
 which is its left-hand side. The proof of \eqref{eq:fourier42} is similar and will be omitted. 
 \end{proof}

\subsection{The Ramanujan Function $A_{q}\left(z\right)$ }
\begin{thm}
\label{thm:airy fourier pairs}The first Fourier pair for $A_{q}\left(z\right)$
is

\begin{equation}
q^{\alpha^{2}/2}A_{q}\left(q^{\alpha}z\right)
=\sqrt{\frac{\log q^{-1}}{2\pi}}\int_{-\infty}^{\infty}
E_{q}\left(-zq^{1/2+iy}\right)q^{y^{2}/2+i\alpha y}dy,
 \label{eq:fourier airy 1}
\end{equation}
 
\begin{equation}
E_{q}\left(-zq^{1/2+iy}\right)q^{y^{2}/2}
 =\sqrt{\frac{\log q^{-1}}{2\pi}}\int_{-\infty}^{\infty}q^{\alpha^{2}/2-i\alpha y}
  A_{q}\left(q^{\alpha}z\right)d\alpha,\label{eq:fourier airy 2}
\end{equation}
 its related double and half modulus pair is 
\begin{equation}
A_{q^{2}}\left(z^{2}\right)=\sqrt{\frac{\log q^{-1}}{\pi}}\int_{-\infty}^{\infty}q^{\alpha^{2}}A_{q}\left(q^{\alpha}z\right)A_{q}\left(-q^{-\alpha}z\right)d\alpha,\label{eq:fourier airy 3}
\end{equation}

\begin{equation}
A_{q}\left(z\right)=\sqrt{\frac{\log q^{-1}}{2\pi}}
 \int_{-\infty}^{\infty}q^{\alpha^{2}/2}A_{q^{2}}\left(q^{\alpha+1/2}z\right)
 A_{q^{2}}\left(q^{-\alpha-1/2}z\right)d\alpha.
  \label{eq:fourier airy 4}
\end{equation}
 The second Fourier pair for $A_{q}\left(z\right)$ is 
\begin{equation}
q^{\alpha^{2}}A_{q}\left(q^{2\alpha}z\right)=\sqrt{\frac{\log q^{-1}}{4\pi}}
 \int_{-\infty}^{\infty}q^{y^{2}/4+i\alpha y}e_{q}\left(-zq^{yi}\right)dy,
  \label{eq:fourier airy 5}
\end{equation}
\begin{equation}
q^{y^{2}/4}e_{q}\left(-zq^{iy}\right)=\sqrt{\frac{\log q^{-1}}{\pi}}
 \int_{-\infty}^{\infty}q^{\alpha^{2}-i\alpha y}A_{q}\left(q^{2\alpha}z\right)d\alpha\label{eq:fourier airy 6}
\end{equation}
 and its related double and half modulus pair is 
\begin{equation}
A_{q^{2}}\left(-z^{2}\right)=\sqrt{\frac{\log q^{-1}}{2\pi}}\int_{-\infty}^{\infty}q^{\alpha^{2}/2}A_{q}\left(q^{-\alpha}z\right)A_{q}\left(-q^{\alpha}z\right)d\alpha\label{eq:fourier airy 7}
\end{equation}
\begin{equation}
A_{q}\left(z\right)=\sqrt{\frac{\log q^{-1}}{\pi}}\int_{-\infty}^{\infty}q^{\alpha^{2}}A_{q^{2}}\left(q^{2\alpha}z\right)A_{q^{2}}\left(q^{1-2\alpha}z\right)d\alpha.\label{eq:fourier airy 8}
\end{equation}
 \end{thm}
\begin{proof}
From
\begin{eqnarray*}
q^{\alpha^{2}/2}A_{q}(q^{\alpha}z) & = & \sum_{n=0}^{\infty}\frac{q^{n^{2}/2}(-z)^{n}}{(q;q)_{n}}q^{(n+\alpha)^{2}/2}
\end{eqnarray*}
to get 
\begin{equation}
q^{\alpha^{2}/2}A_{q}\left(q^{\alpha}z\right)=\frac{1}{\sqrt{2\pi}}\int_{-\infty}^{\infty}\frac{\left(zq^{1/2}e^{ix};q\right)_{\infty}\exp\left(\frac{x^{2}}{\log q^{2}}+i\alpha x\right)}{\sqrt{\log q^{-1}}}dx\label{eq:fourier airy 9}
\end{equation}
 and its inverse transform 
\begin{equation}
\left(zq^{1/2}e^{ix};q\right)_{\infty}\exp\left(\frac{x^{2}}{\log q^{2}}\right)=\sqrt{\frac{\log q^{-1}}{2\pi}}\int_{-\infty}^{\infty}q^{\alpha^{2}/2}A_{q}\left(q^{\alpha}z\right)\exp\left(-i\alpha x\right)d\alpha,\label{eq:fourier airy 10}
\end{equation}
which gives the first Fourier pair. From \eqref{eq:fourier airy 9}
to obtain
\[
q^{\alpha^{2}/2}A_{q}\left(-q^{-\alpha}\overline{z}\right)=
 \frac{1}{\sqrt{2\pi}}\int_{-\infty}^{\infty}\frac{\left(-\overline{z}q^{1/2}e^{-ix};q\right)_{\infty}
 \exp\left(\frac{x^{2}}{\log q^{2}}+i\alpha x\right)}{\sqrt{\log q^{-1}}}dx,
\]
this equation and \eqref{eq:fourier airy 9} imply 
\begin{eqnarray*}
\bg
\int_{-\infty}^{\infty}q^{\alpha^{2}}A_{q}\left(q^{\alpha}z\right)A_{q}\left(-q^{-\alpha}z\right)d\alpha  =  \frac{1}{\log q^{-1}}\int_{-\infty}^{\infty}\left(z^{2}qe^{2ix};q^{2}\right)_{\infty}
 \exp\left(\frac{x^{2}}{\log q}\right)dx\\
 =  \frac{1}{\log q^{-2}}\int_{-\infty}^{\infty}\left(z^{2}qe^{ix};q^{2}\right)_{\infty}
  \exp\left(\frac{x^{2}}{\log q^{4}}\right)dx  
  = \sqrt{\frac{\pi}{\log q^{-1}}}A_{q^{2}}\left(z^{2}\right),
\eg
\end{eqnarray*}
which is \eqref{eq:fourier airy 3}. Next we use  
\[
q^{\alpha^{2}}A_{q^{2}}\left(q^{2\alpha}z\right)=\frac{1}{\sqrt{2\pi}}\int_{-\infty}^{\infty}\frac{\left(zqe^{ix};q^{2}\right)_{\infty}\exp\left(\frac{x^{2}}{\log q^{4}}+i\alpha x\right)}
 {\sqrt{\log q^{-2}}}dx, 
\]
and 
\[
q^{\alpha^{2}}A_{q^{2}}\left(q^{1-2\alpha}\overline{z}\right)
 =\frac{1}{\sqrt{2\pi}}\int_{-\infty}^{\infty}
  \frac{\left(\overline{z}q^{2}e^{-ix};q^{2}\right)_{\infty}
   \exp\left(\frac{x^{2}}{\log q^{4}}+i\alpha x\right)}{\sqrt{\log q^{-2}}}dx, 
\]
 to establish 
\begin{eqnarray*}
\int_{-\infty}^{\infty}q^{2\alpha^{2}}A_{q^{2}}\left(q^{2\alpha}z\right)
 A_{q^{2}}\left(q^{1-2\alpha}z\right)d\alpha   \qquad    \qquad \qquad 
  \qquad  \qquad \qquad \\
\qquad \qquad =  \int_{-\infty}^{\infty}\frac{\left(zqe^{ix};q\right)_{\infty}
 \exp\left(\frac{x^{2}}{\log q^{2}}\right)}{\log q^{-2}}dx 
  =  A_{q}\left(q^{1/2}z\right)\sqrt{\frac{\pi}{\log q^{-2}}}. 
\end{eqnarray*}
which is \eqref{eq:fourier airy 4}. 

Next we start with 
\[
q^{\left(n+\alpha\right)^{2}}=\frac{1}{\sqrt{2\pi\log q^{-2}}}\int_{-\infty}^{\infty}
  \exp\left(\frac{x^{2}}{\log q^{4}}+i\left(n+\alpha\right)x\right)dx, 
\]
 and prove that   
\begin{eqnarray*}
q^{\alpha^{2}}A_{q}(q^{2\alpha}z) & = & \sum_{n=0}^{\infty}\frac{(-z)^{n}}{(q;q)_{n}}
 q^{(n+\alpha)^{2}}. 
\end{eqnarray*}
Thus we have proved  
\begin{equation}
q^{\alpha^{2}}A_{q}\left(q^{2\alpha}z\right)=\frac{1}{\sqrt{2\pi}}
 \int_{-\infty}^{\infty}\frac{\exp\left(\frac{x^{2}}{\log q^{4}}+i\alpha x\right)}
 {\left(-ze^{ix};q\right)_{\infty}\sqrt{\log q^{-2}}}dx, 
\label{eq:fourier airy 11}
\end{equation}
 and its inverse, namely 
\begin{equation}
\frac{\exp\left(\frac{x^{2}}{\log q^{4}}\right)}{\left(ze^{ix};q\right)_{\infty}\sqrt{\log q^{-2}}}
=\frac{1}{\sqrt{2\pi}}\int_{-\infty}^{\infty}q^{\alpha^{2}}A_{q}\left(-q^{2\alpha}z\right)
\exp\left(-i\alpha x\right)d\alpha,                \label{eq:fourier airy 12}
\end{equation}
for all $\left|z\right|<1$ and $x\in\mathbb{R}$, this equations
are \eqref{eq:fourier airy 5} and \eqref{eq:fourier airy 6}.

From \eqref{eq:fourier airy 11} and 
\[
q^{\alpha^{2}}A_{q}\left(-q^{-2\alpha}\overline{z}\right)=\frac{1}{\sqrt{2\pi}}
\int_{-\infty}^{\infty}\frac{\exp\left(\frac{x^{2}}{\log q^{4}}+i\alpha x\right)}
{\left(\overline{z}e^{-ix};q\right)_{\infty}\sqrt{\log q^{-2}}}dx
\]
 we find that 
\begin{eqnarray*}
\bg
\int_{-\infty}^{\infty}q^{2\alpha^{2}}A_{q}\left(q^{2\alpha}z\right)
A_{q}\left(-q^{-2\alpha}z\right)d\alpha 
\qquad \qquad \qquad \\
 =  \int_{-\infty}^{\infty}\frac{\exp\left(\frac{x^{2}}{\log q^{8}}\right)}
  {\left(z^{2}e^{ix};q^{2}\right)_{\infty}\log q^{-4}}dx  
  =  A_{q^{2}}\left(-z^{2}\right)\sqrt{\frac{\pi}{\log q^{-2}}},
\eg
\end{eqnarray*}
which is \eqref{eq:fourier airy 7}. Now apply  
\[
q^{2\alpha^{2}}A_{q^{2}}\left(q^{4\alpha}z\right)=\frac{1}{\sqrt{2\pi}}\int_{-\infty}^{\infty}\frac{\exp\left(\frac{x^{2}}{\log q^{8}}+i\alpha x\right)}{\left(-ze^{ix};q^{2}\right)_{\infty}\sqrt{\log q^{-4}}}dx
\]
 and 
\[
q^{2\alpha^{2}}A_{q^{2}}\left(q^{1-4\alpha}\overline{z}\right)=\frac{1}{\sqrt{2\pi}}\int_{-\infty}^{\infty}\frac{\exp\left(\frac{x^{2}}{\log q^{8}}+i\alpha x\right)}{\left(-\overline{z}qe^{-ix};q^{2}\right)_{\infty}\sqrt{\log q^{-4}}}dx
\]
to obtain 
\begin{eqnarray*}
\bg
\int_{-\infty}^{\infty}q^{4\alpha^{2}}A_{q^{2}}\left(q^{4\alpha}z\right)A_{q^{2}}\left(q^{1-4\alpha}z\right)d\alpha \qquad \qquad  \qquad \qquad \\
\qquad  =  \int_{-\infty}^{\infty}\frac{\exp\left(\frac{x^{2}}{\log q^{4}}\right)}
{\left(-ze^{ix};q\right)_{\infty}\log q^{-4}}dx  
  =  A_{q}\left(z\right)\sqrt{\frac{\pi}{\log q^{-4}}},
  \eg
\end{eqnarray*}
which gives \eqref{eq:fourier airy 8}. This completes the proof of this theorem. 
\end{proof} 
It must be noted that formulas \eqref{eq:fourier airy 1},  \eqref{eq:fourier airy 2}, 
 \eqref{eq:fourier airy 5}
and    \eqref{eq:fourier airy 6} are straight 
forward  to prove directly by expanding the functions in the integrands.  

A proof of 
 \eqref{eq:fourier airy 3} using series manipulations is as follows. The right-hand side of 
  \eqref{eq:fourier airy 3} is 
  \begin{eqnarray}
  \notag
  \sum_{m,n=0}^\infty \frac{(-1)^m z^{m+n}}{(q;q)_m(q;q)_n}\, q^{m^2+n^2- (m-n)^2/4}
    = \sum_{s=0}^\infty \frac{z^{2s}}{(q^2;q^2)_s} (-1)^s q^{2s^2}, 
  \end{eqnarray}
where we used \eqref{eqqbinom2} in the last step.  This gives 
  \eqref{eq:fourier airy 3}. 
   The proof of  \eqref{eq:fourier airy 4} is along the same lines. 
  Its right-hand side is  
\begin{eqnarray}
\notag
\bg
\sum_{m,n=0}^\infty \frac{(-z)^{m+n}\, q^{2m^2+2n^2}}{(q^2;q^2)_m(q^2;q^2)_n} q^{(m-n)/2}
q^{-(m-n)^2/2} \\
= \sum_{s=0}^\infty  (-z)^s q^{s(s+1)/2} 
 \sum_{n=0}^s \frac{q^{n(n-1)}}{(q^2;q^2)_n} \frac{q^{(s-n)^2}}{(q^2;q^2)_{s-n}}. 
\eg
\end{eqnarray}
Let $u_s$ denote the $n$ sum. Then 
\begin{eqnarray}
\notag
\sum_{s=0}^\infty u_s w^s = (-w, -qw;q^2)_\infty = (-w;q)_\infty 
 = \sum_{s=0}^\infty \frac{q^{s(s-1)/2}}{(q^2;q^2)_s} w^s.
\end{eqnarray}
Simple manipulations now show that the right-hand side of \eqref{eq:fourier airy 4} 
 reduces to its left-hand side. Formula \eqref{eq:fourier airy 5}  follows from the series expansions of the Ramanujan function and \eqref{eqqbinom1} while \eqref{eq:fourier airy 5} follows from the definition of the $A_q$ function and \eqref{eqqbinom3}.   
 
 \subsection{Orthogonal  Polynomials }  
 In this section we derive several Fourier pairs involving the Stieltjes--Wigert, $q^{-1}$-Hermite,  
 and $q$-Laguerre   polynomials.
 
\begin{thm}
\label{thm:The-Fourier-pair sw}The Fourier pair for Stieltjes-Wigert
polynomials is
\begin{equation}
q^{\alpha^{2}/2}S_{n}\left(xq^{\alpha-1/2};q\right)
 =\sqrt{\frac{\log q^{-1}}{2\pi}}\int_{-\infty}^{\infty}\frac{\left(xq^{iy};q\right)_{n}
  q^{y^{2}/2+i\alpha y}dy}{\left(q;q\right)_{n}}dy,
   \label{eq:fouriersw 1}
\end{equation}
 
\begin{equation}
\frac{\left(xq^{1/2+iy};q\right)_{n}q^{y^{2}/2}}{\left(q;q\right)_{n}}
 =\sqrt{\frac{\log q^{-1}}{2\pi}}\int_{-\infty}^{\infty}S_{n}\left(xq^{\alpha};q\right)
  q^{\alpha^{2}/2-i\alpha y}d\alpha,
   \label{eq:fouriersw 2}
\end{equation}
 the double and half modulus pair is

\begin{equation}
S_{n}\left(x^{2};q^{2}\right)=\frac{\left(q;q\right)_{n}}{\left(-q;q\right)_{n}}\sqrt{\frac{\log q^{-1}}{\pi}}\int_{-\infty}^{\infty}q^{\alpha^{2}}S_{n}\left(-xq^{-\alpha};q\right)S_{n}\left(xq^{\alpha};q\right)d\alpha,\label{eq:fouriersw 3}
\end{equation}
 \textup{
\begin{equation}
S_{2n}\left(x;q\right)=\frac{\left(q^{2};q^{2}\right)_{n}}{\left(q;q^{2}\right)_{n}}\sqrt{\frac{\log q^{-2}}{\pi}}\int_{-\infty}^{\infty}q^{2\alpha^{2}}S_{n}\left(xq^{1/2-2\alpha};q^{2}\right)S_{n}\left(xq^{2\alpha-1/2};q^{2}\right)d\alpha.\label{eq:fouriersw 4}
\end{equation}
}\end{thm}
\begin{proof}
From \eqref{eq:stieltjes1} and \eqref{eq:fourier1} we get 
\begin{eqnarray}
\label{eq:fouriersw 5}
\bg
q^{\alpha^{2}/2}S_{n}\left(xq^{\alpha-1/2};q\right)  
=  \frac{1}{(q;q)_{n}}\sum_{k=0}^{n}\frac{(q^{-n};q)_{k}}{(q;q)_{k}}\left(xq^{n}\right)^{k}q^{\left(k+\alpha\right)^{2}/2},  \\
  =  \frac{1}{\sqrt{\pi\log q^{-2}}}\int_{-\infty}^{\infty}\frac{\left(xe^{iy};q\right)_{n}}{\left(q;q\right)_{n}}\exp\left(\frac{y^{2}}{\log q^{2}}+i\alpha y\right)dy,
 \eg
\end{eqnarray}
which is \eqref{eq:fouriersw 1}. Its inverse transform, namely 
\begin{equation}
\frac{\left(xe^{iy};q\right)_{n}\exp\left(\frac{y^{2}}{\log q^{2}}\right)}{\left(q;q\right)_{n}\sqrt{\log q^{-1}}}=\frac{1}{\sqrt{2\pi}}\int_{-\infty}^{\infty}q^{\alpha^{2}/2}S_{n}\left(xq^{\alpha-1/2};q\right)e^{-i\alpha y}d\alpha,\label{eq:fouriersw 6}
\end{equation}
gives \eqref{eq:fouriersw 2}. 

From \eqref{eq:fouriersw 5} to get 
\begin{eqnarray*}
\bg
\int_{-\infty}^{\infty}q^{\alpha^{2}}S_{n}\left(-xq^{-\alpha-1/2};q\right)S_{n}\left(xq^{\alpha-1/2};q\right)d\alpha  =  \frac{1}{\log q^{-1}}\int_{-\infty}^{\infty}\frac{\left(x^{2}e^{2iy};q^{2}\right)_{n}}{\left(q;q\right)_{n}^{2}}\exp\left(\frac{y^{2}}{\log q}\right)dy\\
  =  \frac{1}{\log q^{-2}}\int_{-\infty}^{\infty}\frac{\left(x^{2}e^{iy};q^{2}\right)_{n}}{\left(q;q\right)_{n}^{2}}\exp\left(\frac{y^{2}}{2\log q^{2}}\right)dy  
  =  \frac{\left(q^{2};q^{2}\right)_{n}}{\left(q;q\right)_{n}^{2}}\frac{\sqrt{\pi\log q^{-4}}}{\log q^{-2}}S_{n}\left(x^{2}q^{-1};q^{2}\right),
 \eg
\end{eqnarray*}
 which is \eqref{eq:fouriersw 3}.   Similarly, 
\begin{eqnarray*}
\bg
\int_{-\infty}^{\infty}q^{2\alpha^{2}}S_{n}\left(xq^{-2\alpha};q^{2}\right)S_{n}\left(xq^{2\alpha-1};q^{2}\right)d\alpha   =   \int_{-\infty}^{\infty}\frac{\left(xe^{iy},qxe^{iy};q^{2}\right)_{n}\exp\left(\frac{y^{2}}{\log q^{2}}\right)dy}{\log q^{-2}\left(q^{2};q^{2}\right)_{n}^{2}}\\
   =   \int_{-\infty}^{\infty}\frac{\left(xe^{iy};q\right)_{2n}\exp\left(\frac{y^{2}}{\log q^{2}}\right)dy}{\log q^{-2}\left(q^{2};q^{2}\right)_{n}^{2}}   
   =   \frac{\left(q;q\right)_{2n}}{\left(q^{2};q^{2}\right)_{n}^{2}}\frac{\sqrt{\pi\log q^{-2}}}{\log q^{-2}}S_{2n}\left(xq^{-1/2};q\right),
 \eg
\end{eqnarray*}
 which yields  \eqref{eq:fouriersw 4}.
 \end{proof}
 
 Observe that \eqref{eq:fouriersw 1} and \eqref{eq:fouriersw 2} follow directly from the definition of 
 the Stieltjes-Wigert polynomials.

\begin{thm}
\label{thm:The-Fourier-pair q-hermite}
The Fourier pair for $q^{-1}$-Hermite
polynomial $h_{n}\left(x\vert q\right)$ is
\begin{eqnarray}
\bg
\frac{q^{\left(\alpha^{2}-n\alpha+n^{2}\right)/2}h_{n}\left(\sinh\left(\frac{\xi}{2}-\log q^{\left(\alpha-n\right)/2}\right)\bigg|q\right)}{\left(-e^{-\xi/2}\right)^{n}\sqrt{\log q^{-1}}} \\
\qquad \qquad 
=\frac{1}{\sqrt{2\pi}}\int_{-\infty}^{\infty}q^{x^{2}/2+i\alpha x}\left(e^{\xi}q^{1/2-ix};q\right)_{n}dx,
\eg
\label{eq:fourierh 1}
\end{eqnarray}
\begin{eqnarray}
\bg
\frac{q^{x^{2}/2-n^{2}/2}\left(e^{\xi}q^{1/2-ix};q\right)_{n}}{\left(-e^{\xi/2}\right)^{n}\sqrt{\log q^{-1}}} 
\qquad \qquad\qquad \qquad \\
\qquad \qquad =\frac{1}{\sqrt{2\pi}}\int_{-\infty}^{\infty}q^{\left(\alpha^{2}-n\alpha-2i\alpha x\right)/2}
 h_{n}\left(\sinh\left(\frac{\xi}{2}-\log q^{\left(\alpha-n\right)/2}\right)\bigg|q\right)d\alpha,
\eg
\label{eq:fourierh 2}
\end{eqnarray}
 the double and half modulus pair is
\begin{eqnarray}
\bg
h_{n}\left(\sinh\xi\bigg|q^{2}\right)=\sqrt{\frac{\log q^{-1}}{\pi}}\int_{-\infty}^{\infty}q^{\alpha^{2}}
 h_{n}\left(\sinh\left(\frac{\xi}{2}-\log q^{\alpha/2}\right)\bigg|q\right)    \\
\qquad \qquad  \qquad\times \; 
i^{n}h_{n}\left(i\cosh\left(\frac{\xi}{2}+\log q^{\alpha/2}\right)\bigg|q\right)d\alpha, 
\eg
\label{eq:fourierh 3}
\end{eqnarray}
\begin{eqnarray}
\bg
q^{n/4}h_{2n}\left(\sinh\xi\bigg|q\right)=\sqrt{\frac{\log q^{-2}}{\pi}}
 \int_{-\infty}^{\infty}q^{2\alpha^{2}}
  h_{n}\left(\sinh\left(\xi-\log q^{\alpha+1/4}\right)\bigg|q^{2}\right) \\
\qquad \qquad  \qquad \times \; h_{n}\left(\sinh\left(\xi+\log q^{\alpha}\right)\bigg|q^{2}\right)d\alpha. 
\eg\label{eq:fourierh 4}
\end{eqnarray}
\end{thm}
\begin{proof} 
From  \eqref{eq:fourier1} and the definition of $h_{n}\left(x|q\right)$
we  see that  
\begin{eqnarray*}
\bg
q^{\left(\alpha^{2}+n\alpha\right)/2}e^{-n\xi}h_{n}\left(\sinh\left(\xi-\log q^{\alpha/2}\right)\vert q\right)  =  \sum_{k=0}^{n}\frac{(q^{-n};q)_{k}q^{\left(k+\alpha\right)^{2}/2}}{(q;q)_{k}}\left(\frac{\sqrt{q}}{e^{2\xi}}\right)^{k}\\
  =  \sum_{k=0}^{n}\int_{-\infty}^{\infty}
    \exp\left(\frac{y^{2}}{\log q^{2}}+i\alpha y\right)\frac{(q^{-n};q)_{k}}
      {(q;q)_{k}}\frac{\left(e^{-2\xi+iy}q^{1/2}\right)^{k}}{\sqrt{\pi\log q^{-2}}}dy\\
  =  \int_{-\infty}^{\infty}\exp\left(\frac{y^{2}}{\log q^{2}}+i\alpha y\right)
    \frac{\left(e^{-2\xi+iy}q^{1/2-n};q\right)_{\infty}}{\sqrt{\pi\log q^{-2}}
      \left(e^{-2\xi+iy}q^{1/2};q\right)_{\infty}}dy\\
  =  \int_{-\infty}^{\infty}\exp\left(\frac{y^{2}}{\log q^{2}}+i\alpha y\right)
    \frac{\left(e^{-2\xi+iy}q^{1/2-n};q\right)_{n}}{\sqrt{\pi\log q^{-2}}}dy\\
  =  \frac{\left(-e^{-2\xi}\right)^{n}}{q^{n^{2}/2}\sqrt{\pi\log q^{-2}}}
    \int_{-\infty}^{\infty}\exp\left(\frac{y^{2}}{\log q^{2}}+i\left(\alpha+n\right)y\right)
      \left(e^{2\xi-iy}q^{1/2};q\right)_{n}dy. 
  \eg
\end{eqnarray*}
Therefore 
\begin{eqnarray}
\bg
q^{\left(\alpha^{2}-n\alpha+n^{2}\right)/2}
  h_{n}\left(\sinh\left(\frac{\xi}{2}-\log q^{\left(\alpha-n\right)/2}\right)\bigg|q\right)   
     \qquad \qquad \qquad  \\
=\frac{\left(-e^{-\frac{\xi}{2}}\right)^{n}}{\sqrt{\pi\log q^{-2}}}\int_{-\infty}^{\infty}
  \exp\left(\frac{y^{2}}{\log q^{2}}+i\alpha y\right)\left(e^{\xi-iy}q^{1/2};q\right)_{n}dy,
\eg
\label{eq:fourierh 5}
\end{eqnarray}
 which is \eqref{eq:fourierh 1}. Its inverse Fourier transform is 
\begin{eqnarray}
\bg
\frac{\exp\left(\frac{y^{2}}{\log q^{2}}\right)\left(e^{\xi-iy}q^{1/2};q\right)_{n}}
  {\left(-e^{\xi/2}\right)^{n}\sqrt{\log q^{-1}}}   \qquad \qquad \qquad   \qquad \qquad   \\
\qquad =\frac{1}{\sqrt{2\pi}}\int_{-\infty}^{\infty}
  q^{\left(\alpha^{2}-n\alpha+n^{2}\right)/2}h_{n}
    \left(\sinh\left(\frac{\xi}{2}-\log q^{\left(\alpha-n\right)/2}\right)\vert q\right)
      e^{-i\alpha y}d\alpha,
\eg
\label{eq:fourierh 6}
\end{eqnarray}
which is \eqref{eq:fourierh 2}. From 
\begin{eqnarray}
\notag
\bg
q^{\left(\alpha^{2}-n\alpha+n^{2}\right)/2}h_{n}\left(\sinh\left(\frac{\xi}{2}-\log q^{\left(\alpha-n\right)/2}\right)\bigg|q\right)   \qquad \qquad \qquad \\
\qquad \qquad =\frac{\left(-e^{-\frac{\xi}{2}}\right)^{n}}{\sqrt{\pi\log q^{-2}}}\int_{-\infty}^{\infty}\exp\left(\frac{y^{2}}{\log q^{2}}+i\alpha y\right)\left(e^{\xi-iy}q^{1/2};q\right)_{n}dy,
\eg
\end{eqnarray}
\begin{eqnarray}
\notag
\bg
q^{\left(\alpha^{2}+n\alpha+n^{2}\right)/2}h_{n}\left(i\cosh\left(\frac{\xi}{2}-\log q^{\left(-\alpha-n\right)/2}\right)\bigg|q\right)  \qquad \qquad \qquad \\
\qquad \qquad 
=\frac{\left(ie^{-\frac{\xi}{2}}\right)^{n}}{\sqrt{\pi\log q^{-2}}}\int_{-\infty}^{\infty}\exp\left(\frac{y^{2}}{\log q^{2}}+i\alpha y\right)\left(-e^{\xi+iy}q^{1/2};q\right)_{n}dy
\eg
\end{eqnarray}
 we obtain 
\begin{eqnarray*}
\bg
    \int_{-\infty}^{\infty}q^{\alpha^{2}}
      h_{n}\left(\sinh\left(\frac{\xi}{2}-\log q^{\left(\alpha-n\right)/2}\right)\bigg|q\right)
        h_{n}\left(i\cosh\left(\frac{\xi}{2}+\log q^{\left(\alpha+n\right)/2}\right)\bigg|q\right)d\alpha\\
  =  \frac{\left(e^{-\xi}q^{-n}\right)^{n}}{2i^{n}\log q^{-1}}
    \int_{-\infty}^{\infty}\exp\left(\frac{y^{2}}{\log q^{4}}\right)\left(e^{2\xi-iy}q;q^{2}\right)_{n}dy\\
  =  i^{-n}h_{n}\left(\sinh\left(\xi+\log q^{n}\right)\bigg|q^{2}\right)\sqrt{\frac{\pi}{\log q^{-1}}},
  \eg  
\end{eqnarray*}
and we arrive at  \eqref{eq:fourierh 3}. From 
\begin{eqnarray}
\notag
\bg
q^{\left(\alpha^{2}-n\alpha+n^{2}\right)}h_{n}
  \left(\sinh\left(\frac{\xi}{2}-\log q^{\left(\alpha-n\right)}\right)\bigg|q^{2}\right)
      \qquad \qquad \qquad  \\
 \qquad \qquad \qquad 
=\frac{\left(-e^{-\frac{\xi}{2}}\right)^{n}}{\sqrt{\pi\log q^{-4}}}
  \int_{-\infty}^{\infty}\exp\left(\frac{y^{2}}{\log q^{4}}+i\alpha y\right)
    \left(e^{\xi-iy}q;q^{2}\right)_{n}dy
\eg
\end{eqnarray}
 and
\begin{eqnarray}
\bg
q^{\left(\alpha^{2}+n\alpha+n^{2}\right)}
 h_{n}\left(\sinh\left(\frac{\xi}{2}+\log q^{\left(\alpha+n+1/4\right)}\right)\bigg|q^{2}\right)
      \qquad \qquad \qquad  \\ 
 \qquad \qquad =\frac{\left(-e^{-\frac{\xi}{2}}q^{-1/4}\right)^{n}}{\sqrt{\pi\log q^{-4}}}
   \int_{-\infty}^{\infty}\exp\left(\frac{y^{2}}{\log q^{4}}+i\alpha y\right)
     \left(e^{\xi+iy}q^{2};q^{2}\right)_{n}dy,
\eg
\notag
\end{eqnarray}
to get
\begin{eqnarray*}
 &  & \int_{-\infty}^{\infty}q^{2\alpha^{2}+2n^{2}}h_{n}
   \left(\sinh\left(\frac{\xi}{2}+\log q^{\left(n-\alpha-1/4\right)}\right)\bigg|q^{2}\right)
     h_{n}\left(\sinh\left(\frac{\xi}{2}+\log q^{\left(\alpha+n\right)}\right)\bigg|q^{2}\right)
       d\alpha\\
 &  &\qquad \qquad =  \frac{e^{-n\xi}q^{n/4}}{\log q^{-2}}\int_{-\infty}^{\infty}
   \exp\left(\frac{y^{2}}{\log q^{2}}\right)\left(e^{\xi-iy}q^{1/2};q\right)_{2n}dy \\
 &  & \qquad \qquad = q^{2n^{2}+n/4}\sqrt{\frac{\pi}{\log q^{-2}}}h_{2n}
   \left(\sinh\left(\frac{\xi}{2}+\log q^{n}\right)\bigg|q\right),
\end{eqnarray*}
which is \eqref{eq:fourierh 4}. 
\end{proof}

Formulas \eqref{eq:fourierh 1}--\eqref{eq:fourierh 2} easily follow  from the 
definition of $h_n$.  We do not know how to prove \eqref{eq:fourierh 3} or 
 \eqref{eq:fourierh 4} directly. However one can combine the use of the Poisson kernel 
 \eqref{eqPoissonKhn} and,   \eqref{eq:fourierh 3} and  
 \eqref{eq:fourierh 4}, to establish the following integral evaluations 
 \begin{eqnarray}
 \bg
\frac{1}{(te^{(1+i)\xi/2}, te^{-(1+i)\xi/2} ;q)_\infty} 
\Sum h_n(\sinh 2\xi|q^2)\, \frac{q^{\binom{n}{2}}}{(q;q)_n}\, t^n \qquad \qquad \\
 \qquad \qquad = \sqrt{ \frac{\log q^{-1}}{\pi}} \int_{-\infty}^\infty q^{\al^2}
 ( -te^{(1-i)\xi/2} q^{-\al}, -te^{(i-1)\xi/2} q^{\al} ;q)_\infty \, d\al, 
\eg
\label{eqgfIE1}
 \end{eqnarray}
  \begin{eqnarray}
  \bg
\frac{1}{(-tq^{-1/4} e^{2\xi}, - t q^{1/4}e^{-2\xi} ;q)_\infty}  
\Sum h_{2n}(\sinh \xi|q) \frac{q^{n(n-1/2)}}{(q^2;q^2)_n} \\
=  \sqrt{\frac{\log q^{-2}}{\pi}}   \int_{-\infty}^\infty q^{2\al^2}
 (t q^{2\al+1/4}, te^{-2\al -1/4} ;q)_\infty\, d\al. 
 \eg
\label{eqgfIE2}
 \end{eqnarray}

\begin{thm}
\label{thm:The-Fourier-pair laguerre} The Fourier pair for 
  $L_{n}^{\left(\alpha\right)}\left(x;q\right)$
is
\begin{eqnarray}
\bg
q^{\beta^{2}/2}\left(q^{\alpha+\beta+n+\frac{1}{2}};q\right)_{\infty}
  L_{n}^{\left(\alpha+\beta-\frac{1}{2}\right)}\left(x;q\right)   \qquad \qquad  \\
 \qquad \qquad =\sqrt{\frac{\log q^{-1}}{2\pi}}\int_{-\infty}^{\infty}
   \frac{\left(xq^{\alpha+iy};q\right)_{n}q^{y^{2}/2+i\beta y}}{\left(q;q\right)_{n}
     \left(-q^{\alpha+iy};q\right)_{\infty}}dy,
      \label{eq:fourier pair laguerre 1}
\eg
\end{eqnarray}
and  
\begin{eqnarray}
\bg
\frac{\left(xq^{\alpha+iy};q\right)_{n}q^{y^{2}/2}}{\left(q;q\right)_{n}
  \left(-q^{\alpha+iy};q\right)_{\infty}}
 \qquad \qquad \qquad  \\
 \qquad \qquad =\sqrt{\frac{\log q^{-1}}{2\pi}}\int_{-\infty}^{\infty}q^{\beta^{2}/2-i\beta y}\left(q^{\alpha+\beta+n+\frac{1}{2}};q\right)_{\infty}L_{n}^{(\alpha+\beta-\frac{1}{2})}\left(x;q\right)d\beta. 
\eg
\label{eq:fourier pair laguerre 2}
\end{eqnarray}
The double and half modulus pair is 
\begin{eqnarray}
\label{eq:fourier pair laguerre 3}
\bg
L_{2n}^{(2\alpha)}\left(x;q\right)  =  \frac{\left(q^{2};q^{2}\right)_{n}}{\left(q;q^{2}\right)_{n}}\int_{-\infty}^{\infty}\frac{\left(q^{2\alpha+2\beta+2n+\frac{3}{2}},q^{2\alpha-2\beta+2n+\frac{5}{2}};q^{2}\right)_{\infty}}{\sqrt{\pi/\log q^{-2}}\left(q^{2\alpha+2n+1};q\right)_{\infty}}  \\
\qquad \qquad \qquad   \times  L_{n}^{(\alpha+\beta-\frac{1}{4})}\left(x;q^{2}\right)L_{n}^{(\alpha-\beta+\frac{1}{4})}\left(x;q^{2}\right)q^{2\beta^{2}}d\beta, 
\eg
\end{eqnarray}
 
\begin{eqnarray}
\label{eq:fourier pair laguerre 4}
\bg
L_{n}^{(\alpha)}\left(-x^{2};q^{2}\right)  =  \frac{\left(q;q\right)_{n}}{\left(-q;q\right)_{n}}\int_{-\infty}^{\infty}\frac{\left(-iq^{\alpha-\beta+n+1},iq^{\alpha+\beta+n+1};q\right)_{\infty}}{\sqrt{\pi/\log q^{-1}}\left(q^{2\alpha+2n+3};q^{2}\right)_{\infty}} \qquad  \\
\qquad \qquad  \qquad   \times  L_{n}^{(\alpha-\beta+\left(2\overline{\tau}\right)^{-1})}
\left(x;q\right)L_{n}^{(\alpha+\beta+\left(2\tau\right)^{-1})}\left(x;q\right)q^{\beta^{2}}d\beta,
 \eg 
\end{eqnarray}
where $q=e^{\pi\tau i}$ with $\Im\left(\tau\right)>0$.\end{thm}
\begin{proof}
We start with 
\begin{eqnarray*}
 &  & \frac{1}{\sqrt{\pi\log q^{-2}}}\int_{-\infty}^{\infty}\frac{\left(xe^{iy}q^{\alpha+1/2};q\right)_{n}}{\left(q;q\right)_{n}}\frac{\exp\left(\frac{y^{2}}{\log q^{2}}+i\beta y\right)}{\left(-q^{\alpha+1/2}e^{iy};q\right)_{\infty}}dy\\
 & = & \frac{1}{\left(q;q\right)_{n}}\sum_{k=0}^{n}\frac{\left(q^{-n};q\right)_{k}\left(xq^{\alpha+n+1/2}\right)^{k}}{\left(q;q\right)_{k}}\frac{1}{\sqrt{\pi\log q^{-2}}}\int_{-\infty}^{\infty}\frac{\exp\left(\frac{y^{2}}{\log q^{2}}+i\left(\beta+k\right)y\right)}{\left(-q^{\alpha+1/2}e^{iy};q\right)_{\infty}}dy\\
 & = & \frac{1}{\left(q;q\right)_{n}}\sum_{k=0}^{n}\frac{\left(q^{-n};q\right)_{k}\left(xq^{\alpha+n+1/2}\right)^{k}q^{\left(\beta+k\right)^{2}/2}\left(q^{\alpha+\beta+k+1};q\right)_{\infty}}{\left(q;q\right)_{k}}\\
 & = & \frac{\left(q^{\alpha+\beta+1};q\right)_{\infty}q^{\beta^{2}/2}}{\left(q;q\right)_{n}}\sum_{k=0}^{\infty}\frac{\left(q^{-n};q\right)_{k}q^{\binom{k+1}{2}}\left(xq^{\alpha+\beta+n}\right)^{k}}{\left(q,q^{\alpha+\beta+1};q\right)_{k}}\\
 & = & q^{\beta^{2}/2}\left(q^{\alpha+\beta+n+1};q\right)_{\infty}L_{n}^{(\alpha+\beta)}\left(x;q\right).
\end{eqnarray*}
 This shows that 
\begin{eqnarray}
\bg
q^{\beta^{2}/2}\left(q^{\alpha+\beta+n+1};q\right)_{\infty}L_{n}^{(\alpha+\beta)}\left(x;q\right)
\qquad \qquad  \\
\qquad \qquad =\frac{1}{\sqrt{2\pi}}\int_{-\infty}^{\infty}\frac{\left(xe^{iy}q^{\alpha+1/2};q\right)_{n}}{\sqrt{\log q^{-1}}\left(q;q\right)_{n}}\frac{\exp\left(\frac{y^{2}}{\log q^{2}}+i\beta y\right)}{\left(-q^{\alpha+1/2}e^{iy};q\right)_{\infty}}dy, 
\eg
\label{eq:fourier pair laguerre 5}
\end{eqnarray}
which is \eqref{eq:fourier pair laguerre 1}. Its inverse transform  is
\begin{eqnarray}
\bg
\frac{\left(xe^{iy}q^{\alpha+1/2};q\right)_{n}}{\sqrt{\log q^{-1}}
\left(q;q\right)_{n}}\frac{\exp\left(\frac{y^{2}}{\log q^{2}}\right)}
{\left(-q^{\alpha+1/2}e^{iy};q\right)_{\infty}}  \qquad \qquad \qquad 
\qquad
\\ \qquad \qquad =\frac{1}{\sqrt{2\pi}}\int_{-\infty}^{\infty}
q^{\beta^{2}/2}\left(q^{\alpha+\beta+n+1};q\right)_{\infty}
L_{n}^{(\alpha+\beta)}\left(x;q\right)e^{-i\beta y}d\beta, 
\eg
\label{eq:fourier pair laguerre 6}
\end{eqnarray}
which is \eqref{eq:fourier pair laguerre 2}. From 
\eqref{eq:fourier pair laguerre 5}
we  get
\begin{eqnarray}
\notag
\bg
q^{\beta^{2}}\left(q^{2\alpha+2\beta+2n+2};q^{2}\right)_{\infty}
L_{n}^{(\alpha+\beta)}\left(x;q^{2}\right)
   \qquad \qquad\qquad \qquad\\ 
\qquad \qquad=\frac{1}{\sqrt{2\pi}}\int_{-\infty}^{\infty}
\frac{\left(xe^{iy}q^{2\alpha+1};q^{2}\right)_{n}}{\sqrt{\log q^{-2}}
\left(q^{2};q^{2}\right)_{n}}\frac{\exp\left(\frac{y^{2}}{\log q^{4}}
+i\beta y\right)}{\left(-q^{2\alpha+1}e^{iy};q^{2}\right)_{\infty}}dy,
\eg
\end{eqnarray}
and
\begin{eqnarray}
\notag
\bg
q^{\beta^{2}}\left(q^{2\alpha-2\beta+2n+3};q^{2}\right)_{\infty}L_{n}^{(\alpha-\beta+1/2)}\left(x;q^{2}\right)  \qquad \qquad \qquad \qquad \\
\qquad \qquad  \qquad \qquad =\frac{1}{\sqrt{2\pi}}\int_{-\infty}^{\infty}\frac{\left(xe^{-iy}q^{2\alpha+2};q^{2}\right)_{n}}{\sqrt{\log q^{-2}}\left(q^{2};q^{2}\right)_{n}}\frac{\exp\left(\frac{y^{2}}{\log q^{4}}+i\beta y\right)}{\left(-q^{2\alpha+2}e^{-iy};q^{2}\right)_{\infty}}dy,
\eg
\end{eqnarray}
they imply
\begin{eqnarray*}
 &  & \int_{-\infty}^{\infty}q^{2\beta^{2}}
 \left(q^{2\alpha+2\beta+2n+2},
 q^{2\alpha-2\beta+2 +3};q^{2}\right)_{\infty}\\
 & \times & L_{n}^{(\alpha+\beta)}\left(x;q^{2}\right)
 L_{n}^{(\alpha-\beta+1/2)}\left(x;q^{2}\right)d\beta\\
 & = & \int_{-\infty}^{\infty}\frac{\left(xe^{iy}
 q^{2\alpha+1};q\right)_{2n}}
 {\log q^{-2}\left(q^{2};q^{2}\right)_{n}^{2}}
 \frac{\exp\left(\frac{y^{2}}{\log q^{2}}\right)dy}
 {\left(-q^{2\alpha+1}e^{iy};q\right)_{\infty}}\\
 & = & \frac{\left(q^{2\alpha+2n+3/2};q\right)_{\infty}
 L_{2n}^{(2\alpha+1/2)}\left(x;q\right)\sqrt{\pi}\left(q;q\right)_{2n}}
 {\sqrt{\log q^{-2}}\left(q^{2};q^{2}\right)_{n}^{2}},
\end{eqnarray*}
which implies \eqref{eq:fourier pair laguerre 3}. Let $q=e^{\pi\tau i}$
with $\Im\left(\tau\right)>0$, from equation 
\eqref{eq:fourier pair laguerre 5}
to obtain
\begin{eqnarray}
\notag
\bg
q^{\beta^{2}/2}\left(iq^{\alpha+\beta+n+1};q\right)_{\infty}
L_{n}^{(\alpha+\beta+\left(2\tau\right)^{-1})}\left(x;q\right)  
\qquad \qquad\qquad \qquad  \\
\qquad \qquad=\frac{1}{\sqrt{2\pi}}\int_{-\infty}^{\infty}
\frac{\left(ixe^{iy}q^{\alpha+1/2};q\right)_{n}}
{\sqrt{\log q^{-1}}\left(q;q\right)_{n}}
\frac{\exp\left(\frac{y^{2}}{\log q^{2}}+i\beta y\right)}
{\left(-iq^{\alpha+1/2}e^{iy};q\right)_{\infty}}dy,
\eg
\end{eqnarray}
 and 
\begin{eqnarray}
\notag
\bg
q^{\beta^{2}/2}\left(iq^{\alpha-\beta+n+1};q\right)_{\infty}
L_{n}^{(\alpha-\beta+\left(2\tau\right)^{-1})}\left(x;q\right)\
qquad\qquad \qquad  \\
\qquad\qquad \qquad   =\frac{1}{\sqrt{2\pi}}\int_{-\infty}^{\infty}
\frac{\left(ixe^{-iy}q^{\alpha+1/2};q\right)_{n}}{\sqrt{\log q^{-1}}
\left(q;q\right)_{n}}\frac{\exp\left(\frac{y^{2}}{\log q^{2}}+i\beta y\right)}
{\left(-iq^{\alpha+1/2}e^{-iy};q\right)_{\infty}}dy
\eg
\end{eqnarray}
give us
\begin{eqnarray*}
\bg 
 \int_{-\infty}^{\infty}q^{\beta^{2}}\left(-iq^{\alpha-\beta+n+1},iq^{\alpha+\beta+n+1};q\right)_{\infty} 
 \qquad\qquad \qquad \\
  \times  L_{n}^{(\alpha-\beta+\left(2\overline{\tau}\right)^{-1})}\left(x;q\right)L_{n}^{(\alpha+\beta+\left(2\tau\right)^{-1})}\left(x;q\right)d\beta\\
\qquad\qquad \qquad   =  \int_{-\infty}^{\infty}\frac{\left(-x^{2}e^{2iy}q^{2\alpha+1};q^{2}\right)_{n}}{\log q^{-1}\left(q;q\right)_{n}^{2}}\frac{\exp\left(\frac{y^{2}}{\log q}\right)}{\left(-q^{2\alpha+1}e^{2iy};q^{2}\right)_{\infty}}dy\\
\qquad\qquad \qquad   =  \int_{-\infty}^{\infty}\frac{\left(-x^{2}e^{iy}q^{2\alpha+1};q^{2}\right)_{n}}{\log q^{-2}\left(q;q\right)_{n}^{2}}\frac{\exp\left(\frac{y^{2}}{\log q^{4}}\right)}{\left(-q^{2\alpha+1}e^{iy};q^{2}\right)_{\infty}}dy\\
 \qquad\qquad \qquad \qquad   =  \left(q^{2\alpha+2n+3};q^{2}\right)_{\infty}L_{n}^{(\alpha)}\left(-x^{2};q^{2}\right)\sqrt{\frac{\pi}{\log q^{-1}}}\frac{\left(-q;q\right)_{n}}{\left(q;q\right)_{n}},
\eg
\end{eqnarray*}
which is \eqref{eq:fourier pair laguerre 4}.
 \end{proof} 
 The integral evaluations 
 \eqref{eq:fourier pair laguerre 1}--\eqref{eq:fourier pair laguerre 2}  
 also follow easily from the definition of the $q$-Laguerre polynomials 
 \eqref{eqqL}.

\subsection{The $q$-Bessel and ${}_{1}\phi_{1}$  Functions}

Before we state the integral representations of 
$J_{\nu}^{(2)}\left(z;q\right)$ we state a lemmas 
which we discovered when trying to prove 
\eqref{eq:fourier pair bessel 3} in two different ways. 
\begin{lem}
The function $J_{\nu}^{(2)}\left(z;q\right)$ has the series representation
\bea
\label{eqneqJq}
(2/z)^\nu J_{\nu}^{(2)}\left(z;q\right)  =
 \frac{1}{(q;q)_\infty} \Sum \frac{(-z^2/4;q)_n}{(q;q)_n} (-1)^n 
q^{\binom{n+1}{2}+ \nu n}.
\eea
\end{lem}
\begin{proof}
We use the $q$-binomial theorem in the form
 \bea
 \label{eq*} 
 \frac{(-z^2/4;q)_n}{(q;q)_n} = 
  \sum_{k=0}^n \frac{q^{\binom{k}{2}}(z/2)^{2k}}{(q;q)_k(q;q)_{n-k}}. 
 \eea
 Thus the right-hand side of \eqref{eqneqJq} is 
 \bea
 \notag
 \Sum  \frac{(-z^2/4;q)_n}{(q;q)_n} (-1)^n 
q^{\binom{n+1}{2}+ \nu n}.
 \eea
 We substitute the  expression for $(-z^2/4;q)_n$ in \eqref{eq*} then 
 interchange the sums. The
   $n$-sum is evaluated using \eqref{eq:1.3} and \eqref{eqneqJq} 
   follows. 
\end{proof}

It is important to note that the right-hand side of \eqref{eqneqJq} 
analytically continues the left-hand side 
 as a function of  $z$ and $\nu$ to a function which is an entire 
 function of $z$ and of $\nu$. It is also
  clear that \eqref{eqneqJq} is a $q$-formula which is not 
  a $q$-analogue of any formula for the 
  Bessel functions.

\begin{thm}\label{thm5.9}
For $\left|z\right|<1$ we have connection formulas
\begin{equation}
q^{\left(\alpha^{2}+\nu^{2}\right)/4}J_{\nu}^{\left(2\right)}
\left(zq^{\alpha/2};q\right)
 =\sqrt{\frac{\log q^{-1}}{\pi}}\int_{-\infty}^{\infty}
  J_{\nu}^{\left(1\right)}\left(zq^{iy};q\right)q^{y^{2}+i\alpha y}dy
 \label{eq:fourier pair bessel 1}
\end{equation}
 and 
\begin{equation}
J_{\nu}^{\left(1\right)}\left(zq^{iy};q\right)q^{y^{2}}
=\sqrt{\frac{\log q^{-1}}{4\pi}}\int_{-\infty}^{\infty}
q^{\left(\alpha^{2}+\nu^{2}-4i\alpha y\right)/4}
J_{\nu}^{\left(2\right)}\left(zq^{\alpha/2};q\right)d\alpha.
\label{eq:fourier pair bessel 2}
\end{equation}
The Fourier transform pair for $J_{\nu}^{(2)}\left(z;q\right)$ is
\begin{equation}
q^{\alpha^{2}/2}J_{\alpha+\nu}^{(2)}\left(z;q\right)
\left(\frac{z}{2}\right)^{-\alpha-\nu}=\sqrt{\frac{\log q^{-1}}{2\pi}}
\int_{-\infty}^{\infty}\frac{\left(\frac{q^{\nu+iy+1/2}z^{2}}
{4};q\right)_{\infty}q^{y^{2}/2+i\alpha y}}
{\left(q,-q^{\nu+iy+1/2};q\right)_{\infty}}dy, 
\label{eq:fourier pair bessel 3}
\end{equation}
and 
\begin{equation}
\frac{\left(\frac{q^{\nu+iy+1/2}z^{2}}{4};q\right)_{\infty}q^{y^{2}/2}}
{\left(q,-q^{\nu+iy+1/2};q\right)_{\infty}}=\sqrt{\frac{\log q^{-1}}{2\pi}}
\int_{-\infty}^{\infty}q^{\alpha^{2}/2-i\alpha y}J_{\alpha+\nu}^{(2)}
\left(z;q\right)\left(\frac{z}{2}\right)^{-\alpha-\nu}d\alpha.
\label{eq:fourier pair bessel 4}
\end{equation}
The double and half modulus pair for $J_{\nu}^{(2)}\left(z;q\right)$
is
\begin{equation}
\frac{J_{2\nu}^{(2)}\left(z;q\right)}{\left(-q,-q,q;q\right)_{\infty}}
=\sqrt{\frac{\log q^{-2}}{\pi}}\int_{-\infty}^{\infty}q^{2\alpha^{2}}
J_{\nu+\alpha-1/4}^{(2)}\left(z;q^{2}\right)J_{\nu-\alpha+1/4}^{(2)}
\left(z;q^{2}\right)d\alpha\label{eq:fourier pair bessel 5}
\end{equation}
and
\begin{eqnarray}
\bg
\frac{\left(\frac{z}{2}\right)^{\Re\left(\tau^{-1}\right)}
\left(-q;q\right)_{\infty}}{\left(q;q\right)_{\infty}}I_{\nu}^{(2)}
\left(\frac{z^{2}}{2};q^{2}\right) \qquad \qquad \qquad \qquad\\
\qquad \qquad \qquad \qquad=\sqrt{\frac{\log q^{-1}}{4\pi}}
\int_{-\infty}^{\infty}q^{\alpha^{2}}
J_{\nu+\alpha+\left(2\tau\right)^{-1}}^{(2)}
\left(z;q\right)J_{\nu-\alpha+\left(2\overline{\tau}\right)^{-1}}^{(2)}
\left(z;q\right)d\alpha.
\eg
\label{eq:fourier pair bessel 6}
\end{eqnarray}
 \end{thm}
\begin{proof}
Observe that 
\[
\begin{aligned} & \frac{1}{\sqrt{\pi\log q^{-1}}}\int_{-\infty}^{\infty}
J_{\nu}^{(1)}\left(ze^{ix};q\right)\left(\frac{z}{2}\right)^{-\nu}
\exp\left(\frac{x^{2}}{\log q}+i\left(\alpha-\nu\right)x\right)dx\\
 & =\frac{\left(q^{\nu+1};q\right)_{\infty}}{\left(q;q\right)_{\infty}}
 \sum_{n=0}^{\infty}\frac{1}{(q,q^{\nu+1};q)_{n}}
 \left(-\frac{z^{2}}{4}\right)^{n}\int_{-\infty}^{\infty}
 \frac{\exp\left(\frac{x^{2}}{\log q}+i\left(\alpha+2n\right)x\right)}
 {\sqrt{\pi\log q^{-1}}}dx\\
 & =\frac{\left(q^{\nu+1};q\right)_{\infty}}{\left(q;q\right)_{\infty}}
 \sum_{n=0}^{\infty}\frac{q^{\left(\alpha+2n\right)^{2}/4}}
 {(q,q^{\nu+1};q)_{n}}\left(-\frac{z^{2}}{4}\right)^{n}\\
 & =\frac{q^{\alpha^{2}/4}\left(q^{\nu+1};q\right)_{\infty}}{\left(q;q\right)_{\infty}}\sum_{n=0}^{\infty}\frac{q^{n^{2}}}{(q,q^{\nu+1};q)_{n}}\left(-\frac{q^{\alpha}z^{2}}{4}\right)^{n},
\end{aligned}
\]
 that is 
\begin{equation}
q^{\left(\alpha^{2}+\nu^{2}\right)/4}J_{\nu}^{\left(2\right)}\left(zq^{\alpha/2};q\right)=\frac{1}{\sqrt{\pi\log q^{-1}}}\int_{-\infty}^{\infty}J_{\nu}^{\left(1\right)}\left(ze^{ix};q\right)\exp\left(\frac{x^{2}}{\log q}+i\alpha x\right)dx,\label{eq:fourier pair bessel 8}
\end{equation}
 which is \eqref{eq:fourier pair bessel 1} and its inverse transform
\begin{equation}
J_{\nu}^{\left(1\right)}\left(ze^{ix};q\right)\exp\left(\frac{x^{2}}{\log q}\right)=\sqrt{\frac{\log q^{-1}}{4\pi}}\int_{-\infty}^{\infty}q^{\left(\alpha^{2}+\nu^{2}\right)/4}J_{\nu}^{\left(2\right)}\left(zq^{\alpha/2};q\right)e^{-i\alpha x}d\alpha,\label{eq:fourier pair bessel 9}
\end{equation}
is \eqref{eq:fourier pair bessel 2}.

From 
\[
\frac{q^{\left(\alpha+n\right)^{2}/2}}{\left(q^{\alpha+\nu+1};q\right)_{n}}=\frac{1}{\sqrt{\pi\log q^{-2}}}\int_{-\infty}^{\infty}\frac{\exp\left(\frac{x^{2}}{\log q^{2}}+i\left(\alpha+n\right)x\right)}{\left(q^{\alpha+\nu+1},-q^{\nu+1/2}e^{ix};q\right)_{\infty}}dx
\]
 to obtain 
\begin{eqnarray*}
\bg
q^{\alpha^{2}/2}J_{\alpha+\nu}^{(2)}\left(z;q\right)\left(\frac{z}{2}\right)^{-\alpha-\nu}  =  \frac{\left(q^{\alpha+\nu+1};q\right)_{\infty}}{\left(q;q\right)_{\infty}}\sum_{n=0}^{\infty}\frac{q^{\left(\alpha+n\right)^{2}/2}}{(q^{\alpha+\nu+1};q)_{n}}\frac{q^{n^{2}/2}}{\left(q;q\right)_{n}}\left(-\frac{q^{\nu}z^{2}}{4}\right)^{n}\\
  =  \frac{1}{\sqrt{\pi\log q^{-2}}}\sum_{n=0}^{\infty}\int_{-\infty}^{\infty}\frac{q^{\binom{n}{2}}}{\left(q;q\right)_{n}}\left(-\frac{q^{\nu+1/2}z^{2}e^{ix}}{4}\right)^{n} 
   \frac{\exp\left(\frac{x^{2}}{\log q^{2}}+i\alpha x\right)}{\left(q,-q^{\nu+1/2}e^{ix};q\right)_{\infty}}dx.
 \eg
\end{eqnarray*}
The series is summed by \eqref{eq:1.3} and the result is 
\begin{eqnarray}
\bg
q^{\alpha^{2}/2}J_{\alpha+\nu}^{(2)}\left(z;q\right)\left(\frac{z}{2}\right)^{-\alpha-\nu}  \qquad 
\qquad   \qquad \qquad \qquad\qquad
\\
\qquad\qquad=\frac{1}{\sqrt{2\pi\log q^{-1}}}\int_{-\infty}^{\infty}\frac{\left(\frac{q^{\nu+1/2}z^{2}e^{ix}}{4};q\right)_{\infty}\exp\left(\frac{x^{2}}{\log q^{2}}+i\alpha x\right)}{\left(q,-q^{\nu+1/2}e^{ix};q\right)_{\infty}}dx,
\eg
\label{eq:fourier pair bessel 10}
\end{eqnarray}
 which is equivalent to   \eqref{eq:fourier pair bessel 3}.    The inverse transform is 
\begin{equation}
\frac{\left(\frac{q^{\nu+1/2}z^{2}e^{ix}}{4};q\right)_{\infty}\exp\left(\frac{x^{2}}{\log q^{2}}\right)}{\sqrt{\log q^{-1}}\left(q,-q^{\nu+1/2}e^{ix};q\right)_{\infty}}=\frac{1}{\sqrt{2\pi}}\int_{-\infty}^{\infty}q^{\alpha^{2}/2}J_{\alpha+\nu}^{(2)}\left(z;q\right)\left(\frac{z}{2}\right)^{-\alpha-\nu}e^{-i\alpha x}d\alpha\label{eq:fourier pair bessel 11}
\end{equation}
and gives \eqref{eq:fourier pair bessel 4}. 

From \eqref{eq:fourier pair bessel 10} to get 
\[
q^{\alpha^{2}}J_{\nu+\alpha}^{(2)}\left(z;q^{2}\right)=\frac{\left(\frac{z}{2}\right)^{\nu+\alpha}}{\sqrt{2\pi\log q^{-2}}}\int_{-\infty}^{\infty}\frac{\left(\frac{z^{2}}{4}q^{2\nu+1}e^{ix};q^{2}\right)_{\infty}\exp\left(\frac{x^{2}}{\log q^{4}}+\alpha xi\right)}{\left(q^{2},-q^{2\nu+1}e^{ix};q^{2}\right)_{\infty}}dx
\]
 and 
\[
q^{\alpha^{2}}J_{\nu-\alpha+1/2}^{(2)}\left(z;q^{2}\right)=\frac{\left(\frac{z}{2}\right)^{\nu-\alpha+1/2}}{\sqrt{2\pi\log q^{-2}}}\int_{-\infty}^{\infty}\frac{\left(\frac{z^{2}}{4}q^{2\nu+2}e^{-ix};q^{2}\right)_{\infty}\exp\left(\frac{x^{2}}{\log q^{4}}+\alpha xi\right)}{\left(q^{2},-q^{2\nu+2}e^{-ix};q^{2}\right)_{\infty}}dx,
\]
 then 
\[
\begin{aligned} & \int_{-\infty}^{\infty}q^{2\alpha^{2}}J_{\nu+\alpha}^{(2)}\left(z;q^{2}\right)J_{\nu-\alpha+1/2}^{(2)}\left(z;q^{2}\right)d\alpha\\
 & =\frac{\left(\frac{z}{2}\right)^{2\nu+1/2}\left(q;q\right)_{\infty}}{\left(q^{2};q^{2}\right)_{\infty}^{2}\log q^{-2}}\int_{-\infty}^{\infty}\frac{\left(\frac{z^{2}}{4}q^{2\nu+1}e^{ix};q\right)_{\infty}\exp\left(\frac{x^{2}}{\log q^{2}}\right)dx}{\left(q,-q^{2\nu+1}e^{ix};q\right)_{\infty}}\\
 & =\frac{\left(q;q\right)_{\infty}}{\left(q^{2};q^{2}\right)_{\infty}^{2}}\sqrt{\frac{\pi}{\log q^{-2}}}J_{2\nu+1/2}^{(2)}\left(z;q\right),
\end{aligned}
\]
 that is \eqref{eq:fourier pair bessel 5}. 

Let $q=e^{\pi\tau i}$ with $\Im\left(\tau\right)>0$, from \eqref{eq:fourier pair bessel 10}
to get 
\begin{eqnarray}
\notag
\bg
q^{\alpha^{2}/2}J_{\nu+\alpha+\left(2\tau\right)^{-1}}^{(2)}\left(z;q\right)\left(\frac{z}{2}\right)^{-\nu-\alpha-\left(2\tau\right)^{-1}} \qquad\qquad \qquad   \\ 
\qquad\qquad  =\frac{1}{\sqrt{2\pi}}\int_{-\infty}^{\infty}\frac{\left(i\frac{z^{2}}{4}q^{\nu+1/2}e^{ix};q\right)_{\infty}\exp\left(\frac{x^{2}}{\log q^{2}}+\alpha xi\right)}{\sqrt{\log q^{-1}}\left(q,-iq^{\nu+1/2}e^{ix};q\right)_{\infty}}dx,
\eg
\end{eqnarray}
 and
\begin{eqnarray}
\notag
\bg
q^{\alpha^{2}/2}J_{\nu-\alpha+\left(2\tau\right)^{-1}}^{(2)}\left(z;q\right)\left(\frac{z}{2}\right)^{-\nu+\alpha-\left(2\tau\right)^{-1}}  \qquad \qquad  \\
\qquad\qquad \qquad=\frac{1}{\sqrt{2\pi}}\int_{-\infty}^{\infty}\frac{\left(i\frac{z^{2}}{4}q^{\nu+1/2}e^{-ix};q\right)_{\infty}\exp\left(\frac{x^{2}}{\log q^{2}}+\alpha xi\right)}{\sqrt{\log q^{-1}}\left(q,-iq^{\nu+1/2}e^{-ix};q\right)_{\infty}}dx,
\eg
\end{eqnarray}
then
\begin{eqnarray*}
 &  & \int_{-\infty}^{\infty}q^{\alpha^{2}}J_{\nu+\alpha+\left(2\tau\right)^{-1}}^{(2)}\left(z;q\right)J_{\nu-\alpha+\left(2\overline{\tau}\right)^{-1}}^{(2)}\left(z;q\right)d\alpha\\
 &  & \qquad =  \frac{\left(\frac{z}{2}\right)^{2\nu+\Re\left(\tau^{-1}\right)}}{\log q^{-1}\left(q;q\right)_{\infty}^{2}}\int_{-\infty}^{\infty}\frac{\left(-\frac{z^{4}}{16}q^{2\nu+1}e^{2ix};q^{2}\right)_{\infty}\exp\left(\frac{x^{2}}{\log q}\right)dx}{\left(-q^{2\nu+1}e^{2ix};q^{2}\right)_{\infty}}\\
  &  & \qquad  = \frac{\left(\frac{z}{2}\right)^{2\nu+\Re\left(\tau^{-1}\right)}\left(-q;q\right)_{\infty}}{\log q^{-1}\left(q;q\right)_{\infty}}\int_{-\infty}^{\infty}\frac{\left(-\frac{z^{4}}{16}q^{2\nu+1}e^{ix};q^{2}\right)_{\infty}\exp\left(\frac{x^{2}}{\log q^{4}}\right)dx}{\left(q^{2},-q^{2\nu+1}e^{ix};q^{2}\right)_{\infty}}\\
  &  & \qquad  =  \frac{\left(\frac{z}{2}\right)^{\Re\left(\tau^{-1}\right)}\left(-q;q\right)_{\infty}}{i^{\nu}\left(q;q\right)_{\infty}}\sqrt{\frac{4\pi}{\log q^{-1}}}J_{\nu}^{(2)}\left(\frac{z^{2}}{2}i,q^{2}\right)
\end{eqnarray*}
 which is \eqref{eq:fourier pair bessel 6}.
 \end{proof}

Formulas 
\eqref{eq:fourier pair bessel 1}--\eqref{eq:fourier pair bessel 2} also 
follow directly from the 
 series expansion of $J_{\nu}^{(1)}\left(z;q\right)$ and 
 $J_{\nu}^{(2)}\left(z;q\right)$. 
 Formulas \eqref{eq:fourier pair bessel 3}
 --\eqref{eq:fourier pair bessel 4} follow from 
  the series expansions and   \eqref{eqneqJq}.  

We next consider the $q$-Bessel function $J_{\nu}^{(3)}$. 
\begin{thm}\label{thm5.10}
The 3rd $q$-Bessel function satisfies the following Fourier type pair
\bea
q^{\alpha(3\alpha+2)/8}J_{\alpha}^{(3)}\left(zq^{-(\alpha+1)/4};q\right)
&=&\sqrt{\frac{\log q^{-2}}{\pi}}\int_{-\infty}^{\infty}J_{\alpha}^{(1)}
\left(zq^{ix};q\right)q^{2x^{2}-i\alpha x}dx.\label{eq:b1-b3}\\
q^{\alpha(7\alpha-2)/8}J_{\alpha}^{(2)}\left(zq^{(1-3\alpha)/4};q\right)
&=&\sqrt{\frac{\log q^{-2}}{\pi}}\int_{-\infty}^{\infty}J_{\alpha}^{(3)}
\left(zq^{ix};q\right)q^{2x^{2}-i\alpha x}dx.\label{eq:b3-b2}
\eea
 and the Mellin type pair
 \bea
\frac{J_{\alpha+\nu}^{(3)}\left(2zq^{\alpha/2};q\right)}{q^{\alpha\nu/2}z^{\alpha+\nu}} & = & \sqrt{\frac{\log q^{-1}}{2\pi}}\int_{-\infty}^{\infty}\frac{q^{x^{2}/2+i\alpha x}dx}{\left(q,-q^{\nu+1/2+ix},-z^{2}q^{1/2+ix};q\right)_{\infty}},\label{eq:b3-3}  \\
\frac{q^{x^{2}/2}}{\left(q,-q^{\nu+1/2+ix},-z^{2}q^{1/2+ix};q\right)_{\infty}} &=& \sqrt{\frac{\log q^{-1}}{2\pi}}\int_{-\infty}^{\infty}\frac{J_{\alpha+\nu}^{(3)}\left(2q^{\alpha/2}z;q\right)}{z^{\alpha+\nu}}q^{-i\alpha(x-i\nu/2)}d\alpha,\label{eq:b3-4}
\eea
\end{thm}
\begin{proof}
Observe that 
\[
\begin{aligned} & \int_{-\infty}^{\infty}J_{\alpha}^{(1)}\left(ze^{ix};q\right)\left(\frac{z}{2}\right)^{-\alpha}\exp\left(\frac{x^{2}}{\log q^{1/2}}-i\alpha x\right)dx\\
 & =\frac{\left(q^{\alpha+1};q\right)_{\infty}}{\left(q;q\right)_{\infty}}\sum_{n=0}^{\infty}\frac{1}{(q,q^{\alpha+1};q)_{n}}\left(-\frac{z^{2}}{4}\right)^{n}\int_{-\infty}^{\infty}\exp\left(\frac{x^{2}}{\log q^{1/2}}+i(2n-\alpha)x\right)dx\\
 & =\frac{\left(q^{\alpha+1};q\right)_{\infty}}{\left(q;q\right)_{\infty}}\sum_{n=0}^{\infty}\frac{1}{(q,q^{\alpha+1};q)_{n}}\left(-\frac{z^{2}}{4}\right)^{n}\int_{-\infty}^{\infty}\exp\left(\frac{x^{2}}{2\log q}+i(n-\alpha/2)x\right)\frac{dx}{2}\\
 & =\frac{\left(q^{\alpha+1};q\right)_{\infty}}{\left(q;q\right)_{\infty}}\sum_{n=0}^{\infty}\frac{\left(-z^{2}/4\right)^{n}}{(q,q^{\alpha+1};q)_{n}}\sqrt{\frac{\pi\log q^{-1}}{2}}q^{(n-\alpha/2)^{2}/2}\\
 & =\frac{q^{\alpha^{2}/8}\left(q^{\alpha+1};q\right)_{\infty}}{\left(q;q\right)_{\infty}}\sqrt{\frac{\pi\log q^{-1}}{2}}\sum_{n=0}^{\infty}\frac{q^{\binom{n+1}{2}}}{(q,q^{\alpha+1};q)_{n}}\left(-\frac{z^{2}}{4q^{(\alpha+1)/2}}\right)^{n}\\
 & =\sqrt{\frac{\pi\log q^{-1}}{2}}q^{\alpha^{2}/8}\left(\frac{z}{2q^{(\alpha+1)/4}}\right)^{-\alpha}J_{\alpha}^{(3)}\left(zq^{-(\alpha+1)/4};q\right)\\
 & =\sqrt{\frac{\pi\log q^{-1}}{2}}q^{\alpha(3\alpha+2)/8}\left(\frac{z}{2}\right)^{-\alpha}J_{\alpha}^{(3)}\left(zq^{-(\alpha+1)/4};q\right),
\end{aligned}
\]
and \eqref{eq:b1-b3}  follows. 
 We prove \eqref{eq:b3-b2} in a similar  fashion, that is 
\[
\begin{aligned} & \int_{-\infty}^{\infty}J_{\alpha}^{(3)}\left(ze^{ix};q\right)\left(\frac{z}{2}\right)^{-\alpha}\exp\left(\frac{x^{2}}{\log q^{1/2}}-i\alpha x\right)dx\\
 & =\frac{\left(q^{\alpha+1};q\right)_{\infty}}{(q;q)_{\infty}}\sum_{n=0}^{\infty}\frac{q^{n^{2}/2}\left(-z^{2}q^{1/2}/4\right)^{n}}{\left(q,q^{\alpha+1};q\right)_{n}}\int_{-\infty}^{\infty}\exp\left(\frac{x^{2}}{2\log q}+i(n-\alpha/2)x\right)\frac{dx}{2}\\
 & =\frac{\left(q^{\alpha+1};q\right)_{\infty}}{(q;q)_{\infty}}\sum_{n=0}^{\infty}\frac{q^{n^{2}/2}\left(-z^{2}q^{1/2}/4\right)^{n}}{\left(q,q^{\alpha+1};q\right)_{n}}\sqrt{\frac{\pi\log q^{-1}}{2}}q^{(n-\alpha/2)^{2}/2}\\
 & =q^{\alpha^{2}/8}\sqrt{\frac{\pi\log q^{-1}}{2}}
 \frac{\left(q^{\alpha+1};q\right)_{\infty}}{(q;q)_{\infty}}
 \sum_{n=0}^{\infty}\frac{q^{n^{2}+\alpha n}\left(-z^{2}
 q^{(1-3\alpha)/2}/4\right)^{n}}{\left(q,q^{\alpha+1};q\right)_{n}}\\
 & =q^{(7\alpha^{2}-2\alpha)/8}\sqrt{\frac{\pi\log q^{-1}}{2}}
 \left(\frac{z}{2}\right)^{-\alpha}J_{\alpha}^{(2)}\left(zq^{(1-3\alpha)/4}; 
 q\right),
\end{aligned}
\]
which completes the proof of the first pair. We proceed to prove the 
next pair of formulas. It is clear that 
\[
\frac{q^{\left(\alpha+n\right)^{2}/2}}{\left(q^{\alpha+\nu+1};q\right)_{n}}
=\frac{1}{\sqrt{\pi\log q^{-2}}}\int_{-\infty}^{\infty}
\frac{\exp\left(\frac{x^{2}}{\log q^{2}}+i\left(\alpha+n\right)x\right)}
{\left(q^{\alpha+\nu+1},-q^{\nu+1/2}e^{ix};q\right)_{\infty}}dx
\]
 leads to
\[
\begin{aligned} & q^{\alpha^{2}/2}\sum_{n=0}^{\infty}\frac{q^{\binom{n+1}{2}}}{\left(q,q^{\alpha+\nu+1};q\right)_{n}}\left(-\frac{z^{2}q^{\alpha}}{4}\right)^{n}\\
 & =\sum_{n=0}^{\infty}\frac{\left(-z^{2}q^{1/2}/4\right)^{n}}{(q;q)_{n}}\int_{-\infty}^{\infty}\frac{\exp\left(\frac{x^{2}}{\log q^{2}}+i\left(\alpha+n\right)x\right)}{\left(q^{\alpha+\nu+1},-q^{\nu+1/2}e^{ix};q\right)_{\infty}}\frac{dx}{\sqrt{\pi\log q^{-2}}}\\
 & =\int_{-\infty}^{\infty}\frac{\exp\left(\frac{x^{2}}{\log q^{2}}+i\alpha x\right)}{\left(q^{\alpha+\nu+1},-q^{\nu+1/2}e^{ix},-z^{2}q^{1/2}e^{ix}/4;q\right)_{\infty}}\frac{dx}{\sqrt{\pi\log q^{-2}}},
\end{aligned}
\]
 This implies \eqref{eq:b3-3} and \eqref{eq:b3-4} follows from the 
 Fourier inversion formula. 
 \end{proof}
 It is worth noting that \eqref{eq:b1-b3} and \eqref{eq:b3-b2} are equivalent to 
 \begin{equation}
q^{\alpha(3\alpha+2)/8}J_{\alpha}^{(3)}\left(zq^{-(\alpha+1)/4};q\right)=\sqrt{\frac{2}{\pi\log q^{-1}}}\int_{-\infty}^{\infty}J_{\alpha}^{(1)}\left(ze^{ix};q\right)\exp\left(\frac{x^{2}}{\log q^{1/2}}-i\alpha x\right)dx, \label{eq:b3-6}
\end{equation}
 and 
 \begin{equation}
q^{(7\alpha^{2}-2\alpha)/8}J_{\alpha}^{(2)}
\left(zq^{(1-3\alpha)/4};q\right)=\sqrt{\frac{2}{\pi\log q^{-1}}}
\int_{-\infty}^{\infty}J_{\alpha}^{(3)}\left(ze^{ix};q\right)
\exp\left(\frac{x^{2}}{\log q^{1/2}}-i\alpha x\right)dx,\label{eq:b3-7}
\end{equation}
 respectively.  Moreover \eqref{eq:b3-3} is equivalent to 
 \begin{equation}
\frac{J_{\alpha+\nu}^{(3)}\left(2q^{\alpha/2}z;q\right)}{q^{\alpha \nu/2} z^{\alpha+\nu}}
=\frac{1}{\sqrt{\pi\log q^{-2}}}\int_{-\infty}^{\infty}
\frac{\exp\left(\frac{x^{2}}{\log q^{2}}+i\alpha x\right)dx}
{\left(q,-q^{\nu+1/2}e^{ix},-z^{2}q^{1/2}e^{ix};q\right)_{\infty}}.\label{eq:b3-8}
\end{equation}

We now come to the basic confluent hypergeometric function. 

\begin{thm}\label{thm5.5.3}
The basic confluent hypergeometric function has t he Fourier pair
\begin{equation}
\left(bq^{\alpha};q\right)_{\infty}q^{\alpha^{2}/2}
 {}_{1}\phi_{1}\left(a;bq^{\alpha};q,zq^{\alpha}\right)
 =\sqrt{\frac{\log q^{-1}}{2\pi}}\int_{-\infty}^{\infty}
 \frac{\left(-azq^{iy-1/2};q\right)_{\infty}q^{y^{2}/2+i\alpha y}dy}
  {\left(-bq^{iy-1/2},-zq^{iy-1/2};q\right)_{\infty}},\label{eq:fourier pair confluent 1}
\end{equation}
\begin{eqnarray}
\bg
\frac{\left(-azq^{iy-1/2};q\right)_{\infty}q^{y^{2}/2}}{\left(-bq^{iy-1/2},-zq^{iy-1/2};q\right)_{\infty}} 
 \qquad 
 \qquad  \qquad \qquad  \qquad\\
 \qquad  \qquad \qquad=\sqrt{\frac{\log q^{-1}}{2\pi}}
  \int_{-\infty}^{\infty}\left(bq^{\alpha};q\right)_{\infty}q^{\alpha^{2}/2-i\alpha y}{}_{1}\phi_{1}\left(a;bq^{\alpha};q,zq^{\alpha}\right)d\alpha.
 \label{eq:fourier pair confluent 2}
\eg
\end{eqnarray}
The double and half modulus pair is 
\begin{eqnarray}
\label{eq:fourier pair confluent 3}
\bg
{}_{1}\phi_{1}\left(a^{2};-b^{2};q^{2},-z^{2}q\right)  =  \sqrt{\frac{\log q^{-1}}{\pi}}\int_{-\infty}^{\infty}
\frac{\left(bq^{\alpha},-bq^{-\alpha};q\right)_{\infty}q^{\alpha^{2}}}{\left(-b^{2};q^{2}\right)_{\infty}}
\\
 \times {} _{1}\phi_{1}\left(a;bq^{\alpha};q,zq^{1/2+\alpha}\right){}_{1}
 \phi_{1}  \left(a;-bq^{-\alpha};q,- zq^{1/2-\alpha}\right) d\alpha, 
  \eg
\end{eqnarray}
\begin{eqnarray}
\label{eq:fourier pair confluent 4}
\bg
_{1}\phi_{1}\left(a;b;q,z\right)  =  \sqrt{\frac{\log q^{-2}}{\pi}}\int_{-\infty}^{\infty}\frac{\left(bq^{2\alpha+1/2},bq^{3/2-2\alpha};q^{2}\right)_{\infty}q^{2\alpha^{2}}}{\left(b;q\right)_{\infty}} \qquad 
\\
 \qquad  \times  _{1}\phi_{1}\left(a;bq^{2\alpha+1/2};q^{2},zq^{1/2+2\alpha}\right){}_{1}\phi_{1}\left(a;bq^{3/2-2\alpha};q^{2},zq^{3/2-2\alpha}\right)d\alpha
  \eg
\end{eqnarray}
\end{thm} 

\begin{cor}\label{cor5.5.4}
The $q$-Laguerre polynomials have the ${}_1\phi_1$ representation 
\begin{eqnarray}
L_{n}^{(\alpha)}\left(x;q\right) = \frac{\left(-xq^{n+\alpha+1};q\right)_{\infty}}
  {\left(q^{\alpha+n+1};q\right)_{\infty}\left(q;q\right)_{n}}\; {} _{1}\phi_{1}\left(-x;-xq^{n+\alpha+1};q,q^{\alpha+1}\right).
  \label{eq:fourier pair confluent laguerre}  
\end{eqnarray}
\end{cor}

\begin{proof}[Proof of Theorem \ref{thm5.5.3}]
Observe that
\begin{eqnarray*}
 &  & \frac{1}{\sqrt{\pi\log q^{-2}}}\int_{-\infty}^{\infty}\frac{\left(-aze^{ix};q\right)_{\infty}
  \exp\left(\frac{x^{2}}{\log q^{2}}+i\alpha x\right)dx}{\left(-bq^{-1/2}e^{ix},-ze^{ix};q\right)_{\infty}}\\
 &  & \qquad = \sum_{k=0}^{\infty}\frac{\left(a;q\right)_{k}\left(-z\right)^{k}}{\left(q;q\right)_{k}}
  \frac{1}{\sqrt{\pi\log q^{-2}}}\int_{-\infty}^{\infty}\frac{\exp\left(\frac{x^{2}} 
   {\log q^{2}}+i\left(\alpha+k\right)x\right)dx}{\left(-bq^{-1/2}e^{ix};q\right)_{\infty}}\\
 &  &  \qquad = \left(bq^{\alpha};q\right)_{\infty}q^{\alpha^{2}/2}
  \sum_{k=0}^{\infty}\frac{\left(a;q\right)_{k}\left(-zq^{\alpha+1/2}\right)^{k}q^{\binom{k}{2}}}{\left(q,bq^{\alpha};q\right)_{k}},
\end{eqnarray*}
 that is
\begin{eqnarray}
\bg
\left(bq^{\alpha};q\right)_{\infty}q^{\alpha^{2}/2}{}_{1}\phi_{1}\left(a;bq^{\alpha};q,zq^{\alpha+1/2}\right) 
\qquad\qquad \qquad
\\
\qquad \qquad \qquad =\frac{1}{\sqrt{\pi\log q^{-2}}}\int_{-\infty}^{\infty}\frac{\left(-aze^{ix};q\right)_{\infty}\exp\left(\frac{x^{2}}{\log q^{2}}+i\alpha x\right)dx}{\left(-bq^{-1/2}e^{ix},-ze^{ix};q\right)_{\infty}},
\eg
\label{eq:fourier pair confluent 5}
\end{eqnarray}
 which gives \eqref{eq:fourier pair confluent 1}. The  inverse is given by 
transform 
\begin{eqnarray}
\bg
\frac{\left(-aze^{ix};q\right)_{\infty}\exp\left(\frac{x^{2}}{\log q^{2}}\right)}{\left(-bq^{-1/2}e^{ix},-ze^{ix};q\right)_{\infty}}  \qquad \qquad\qquad\qquad \qquad \\
\qquad  
=\sqrt{\frac{\log q^{-1}}{2\pi}}\int_{-\infty}^{\infty}\left(bq^{\alpha};q\right)_{\infty}q^{\alpha^{2}/2}{}_{1}\phi_{1}\left(a;bq^{\alpha};q,zq^{\alpha+1/2}\right)e^{-i\alpha x}d\alpha.
\eg
\label{eq:fourier pair confluent 6}
\end{eqnarray}
This  is \eqref{eq:fourier pair confluent 2}.

From \eqref{eq:fourier pair confluent 5} to obtain 
\begin{eqnarray}
\bg
\left(bq^{\alpha};q\right)_{\infty}q^{\alpha^{2}/2}{}_{1}\phi_{1}\left(a;bq^{\alpha};q,zq^{\alpha+1/2}\right)
 \qquad \qquad \qquad   \\
 \qquad \qquad=\frac{1}{\sqrt{\pi\log q^{-2}}}\int_{-\infty}^{\infty}\frac{\left(-aze^{ix};q\right)_{\infty}\exp\left(\frac{x^{2}}{\log q^{2}}+i\alpha x\right)dx}{\left(-bq^{-1/2}e^{ix},-ze^{ix};q\right)_{\infty}},
\eg
\notag
\end{eqnarray}

\begin{eqnarray}
\bg
\left(-bq^{-\alpha};q\right)_{\infty}q^{\alpha^{2}/2}{}_{1}\phi_{1}\left(a;-bq^{-\alpha};q,-zq^{1/2-\alpha}\right)           \qquad \qquad    \qquad         \\
\qquad \qquad   =\frac{1}{\sqrt{\pi\log q^{-2}}}\int_{-\infty}^{\infty}\frac{\left(aze^{-ix};q\right)_{\infty}\exp\left(\frac{x^{2}}{\log q^{2}}+i\alpha x\right)dx}{\left(bq^{-1/2}e^{-ix},ze^{-ix};q\right)_{\infty}},
\eg
\notag
\end{eqnarray}

 and
\begin{eqnarray*}
 &  & \int_{-\infty}^{\infty}{}_{1}\phi_{1}\left(a;bq^{\alpha};q,zq^{\alpha+1/2}\right){}_{1}\phi_{1}\left(a;-bq^{-\alpha};q,-zq^{1/2-\alpha}\right)\\
 & & \quad  \times  \left(bq^{\alpha},-bq^{-\alpha};q\right)_{\infty}q^{\alpha^{2}}d\alpha\\
 &  &  \quad = \frac{1}{\log q^{-1}}\int_{-\infty}^{\infty}\frac{\left(a^{2}z^{2}e^{2ix};q^{2}\right)_{\infty}\exp\left(\frac{x^{2}}{\log q}\right)dx}{\left(b^{2}q^{-1}e^{2ix},z^{2}e^{2ix};q^{2}\right)_{\infty}}\\
 &  &  \quad  = \frac{1}{\log q^{-2}}\int_{-\infty}^{\infty}\frac{\left(a^{2}z^{2}e^{ix};q^{2}\right)_{\infty}\exp\left(\frac{x^{2}}{\log q^{4}}\right)dx}{\left(b^{2}q^{-1}e^{ix},z^{2}e^{ix};q^{2}\right)_{\infty}}\\
 &  & \quad=  \sqrt{\frac{\pi}{\log q^{-1}}}\left(-b^{2};q^{2}\right)_{\infty}{}_{1}\phi_{1}\left(a^{2};-b^{2};q^{2},-z^{2}q\right),
\end{eqnarray*}
 which gives \eqref{eq:fourier pair confluent 3}. 

From \eqref{eq:fourier pair confluent 5} to get
\begin{eqnarray}
\notag
\bg
\left(bq^{2\alpha};q^{2}\right)_{\infty}q^{\alpha^{2}}
{}_{1}\phi_{1}\left(a;bq^{2\alpha};q^{2},zq^{2\alpha+1}\right)  \qquad \qquad \qquad 
\\  
\qquad \qquad \qquad \qquad =\frac{1}{\sqrt{\pi\log q^{-4}}}\int_{-\infty}^{\infty}\frac{\left(-aze^{ix};q^{2}\right)_{\infty}\exp\left(\frac{x^{2}}{\log q^{4}}+i\alpha x\right)dx}{\left(-bq^{-1}e^{ix},-ze^{ix};q^{2}\right)_{\infty}},
\eg
\end{eqnarray}

\begin{eqnarray}
\notag
\bg
\left(bq^{1-2\alpha};q^{2}\right)_{\infty}q^{\alpha^{2}}
{}_{1}\phi_{1}\left(a;bq^{1-2\alpha};q^{2},zq^{2-2\alpha}\right)  
\qquad \qquad  \qquad \qquad   \\
\qquad \qquad  =\frac{1}{\sqrt{\pi\log q^{-4}}}
\int_{-\infty}^{\infty}\frac{\left(-azqe^{-ix};q^{2}\right)_{\infty}
\exp\left(\frac{x^{2}}{\log q^{4}}+i\alpha x\right)dx}{\left(-be^{-ix},-zqe^{-ix};q^{2}\right)_{\infty}},
\eg
\end{eqnarray}
 and
\begin{eqnarray*}
 &  & \int_{-\infty}^{\infty}{}_{1}\phi_{1}\left(a;bq^{2\alpha};q^{2},zq^{1+2\alpha}\right){}_{1}\phi_{1}\left(a;bq^{1-2\alpha};q^{2},zq^{2-2\alpha}\right)\\
 & & \qquad \qquad  \times  
  \left(bq^{2\alpha},bq^{1-2\alpha};q^{2}\right)_{\infty}q^{2\alpha^{2}}d\alpha\\
 &  & \qquad =  \frac{1}{\log q^{-2}}\int_{-\infty}^{\infty}
  \frac{\left(-aze^{ix};q\right)_{\infty}\exp\left(\frac{x^{2}}{\log q^{2}}\right)dx}
    {\left(-bq^{-1}e^{ix},-ze^{ix};q\right)_{\infty}}\\
 &  & \qquad =  \sqrt{\frac{\pi}{\log q^{-2}}}\left(bq^{-1/2};q\right)_{\infty}{}_{1}\phi_{1}\left(a;bq^{-1/2};q,zq^{+1/2}\right),
\end{eqnarray*}
which proves  \eqref{eq:fourier pair confluent 4}.
\end{proof}

\begin{proof}[Proof of Corollary \ref{cor5.5.4}]
Let $z=q^{\alpha+1/2}$ and $a=bq^{-n-\alpha-1}$ in \eqref{eq:fourier pair confluent 5}
to obtain 
\begin{eqnarray*}
\bg
\left(b;q\right)_{\infty}{}_{1}\phi_{1}\left(bq^{-n-\alpha-1};b;q,q^{\alpha+1}\right)  
 =  \frac{1}{\sqrt{\pi\log q^{-2}}}\int_{-\infty}^{\infty}\frac{\left(-bq^{-1/2-n}e^{ix};q\right)_{\infty}
  \exp\left(\frac{x^{2}}{\log q^{2}}\right)dx}{\left(-bq^{-1/2}e^{ix},-q^{\alpha+1/2}e^{ix};q\right)_{\infty}}\\
  =  \frac{1}{\sqrt{\pi\log q^{-2}}}\int_{-\infty}^{\infty}\frac{\left(-bq^{-1/2-n}e^{ix};q\right)_{n}
  \exp\left(\frac{x^{2}}{\log q^{2}}\right)dx}{\left(-q^{\alpha+1/2}e^{ix};q\right)_{\infty}} \\
  =  \left(q^{\alpha+n+1};q\right)_{\infty}\left(q;q\right)_{n}L_{n}^{(\alpha)}\left(-bq^{-n-\alpha-1};q\right)
\eg
\end{eqnarray*}
to get \eqref{eq:fourier pair confluent laguerre}. Another proof is to write the ${}_1\phi_1$ function in 
  \eqref{eq:fourier pair confluent laguerre} as 
  \bea
  \lim_{y\to \infty} {}_{2}\phi_{1}\left(\begin{array}{cc}
 \begin{array}{c}
-x, y\\
-xq^{n+\alpha+1}
\end{array}\bigg|  q,\frac{q^{\alpha+1}}{y}\end{array}\right).
  \notag
  \eea
Then apply the transformation in \cite[(III.2)]{Gas:Rah} or  \cite[(12.4.13)]{Ismbook}, namely
\bea
{}_{2}\phi_{1}\left(\begin{array}{cc}
\begin{array}{c}
A, B\\
C
\end{array}\bigg|  q, Z\end{array}\right) = \frac{(AZ, C/A;q)_\infty}{(C, Z;q)_\infty} 
\, {}_{2}\phi_{1}\left(\begin{array}{cc}
\begin{array}{c}
A, ABZ/C\\
Az
\end{array}\bigg|  q, \frac{C}{A} \end{array}\right).
\notag
\eea
This reduces the right-hand side of \eqref{eq:fourier pair confluent laguerre} to 
\bea
\notag
\frac{1}{(q;q)_n}\, {}_2\phi_1(q^{-n}, -x; 0; q; q^{n+\alpha+1})
\eea
which is the representation (3.21.1) in \cite{Koe:Swa} for $q$-Laguerre polynomials. 
\end{proof}

\subsection{Plancherel Type Identities} 
One can use Parseval's formula \eqref{eqPar} and the results of this section  to prove the 
 equivalence of several integrals. As an example we record the following theorem. 

\begin{thm}\label{thm5.7.1}
For $q_{1},q_{2},\left|z\right|,\left|w\right|\in\left(0,1\right)$
and $\lambda=\sqrt{\log q_{1}\log q_{2}}$we have

\begin{eqnarray}
 &  & \int_{-\infty}^{\infty}\left(q_{1}q_{2}\right)^{\alpha^{2}/2}E_{q_{1}}\left(zq_{1}^{\alpha+1/2}\right)E_{q_{2}}\left(wq_{2}^{\alpha+1/2}\right)d\alpha\label{eq:plancherel 1}\\
 & = & \int_{-\infty}^{\infty}\left(q_{1}q_{2}\right)^{y^{2}/2}e_{q_{1}}\left(ze^{\lambda iy};\right)e_{q_{2}}\left(we^{-i\lambda y }\right)dy,\nonumber 
\end{eqnarray}
\begin{eqnarray}
 &  & \int_{-\infty}^{\infty}\left(q_{1}q_{2}\right)^{\alpha^{2}/2}
  \left(-zq_{1}^{1/2+\alpha};q_{1}\right)_{\infty}
   \left(-wq_{2}^{1/2-\alpha};q_{2}\right)_{\infty}d\alpha\label{eq:plancherel 2}\\
 & = & \int_{-\infty}^{\infty}\left(q_{1}q_{2}\right)^{y^{2}/2}e_{q_{1}}\left(ze^{iy\lambda}\right)e_{q_{2}}\left(we^{iy\lambda}\right)dy\nonumber 
\end{eqnarray}
and
\begin{eqnarray}
 &  & \int_{-\infty}^{\infty}E_{q_{1}}\left(zq_{1}^{x}\right)
  e_{q_{2}}\left(we^{ix}\right)\exp\left\{ \frac{x^{2}}{2}\left(\log q_{1}
   +\frac{1}{\log q_{2}}\right)\right\} \frac{dx}{\sqrt{\log q_{2}^{-1}}}\label{eq:plancherel 3}\\
 & = & \int_{-\infty}^{\infty}E_{q_{2}}\left(wq_{2}^{\alpha+1/2}\right)e_{q_{1}}
  \left(ze^{i\alpha-1/2}\right)\exp\left\{ \frac{\alpha^{2}}{2}\left(\log q_{2} 
  +\frac{1}{\log q_{1}}\right)\right\} \frac{d\alpha}{\sqrt{\log q_{1}^{-1}}}\nonumber 
\end{eqnarray}
In particular, 
\begin{equation}
E_{q^{2}}\left(-t^{2}q\right)\mathcal{E}_{q}\left(x;t\right)=\frac{1}{\sqrt{\pi\log q^{-1}}}
 \int_{-\infty}^{\infty}q^{\alpha^{2}}E_{q}\left(t\sqrt{q}e^{i\theta}q^{\alpha}\right)
  E_{q}\left(t\sqrt{q}e^{-i\theta}q^{-\alpha}\right)d\alpha.\label{eq:plancherel 4}
\end{equation}
 \end{thm}
 
 The next theorem is an example of how Parseval's formula \eqref{eqPar} can be used and one side 
  is evaluated or expanded in a series.

\begin{thm}\label{thm5.7.2}
For $z\in\mathbb{C}$ we have
\begin{equation}
\int_{-\infty}^{\infty}q^{\alpha^{2}}\left(-zq^{\alpha+1};q\right)_{\infty}A_{q}\left(q^{-\alpha}z\right)d\alpha=\sqrt{\frac{\pi}{\log q^{-1}}},\label{eq:plancherel 5}
\end{equation}
\begin{equation}
\int_{-\infty}^{\infty}q^{2\alpha^{2}}A_{q^{2}}\left(q^{2\alpha}z\right)A_{q}\left(-q^{-2\alpha}z\right)d\alpha=\sqrt{\frac{\pi}{\log q^{-2}}}\sum_{k=0}^{\infty}\frac{q^{k^{2}/2}z^{k}}{\left(q^{2};q^{2}\right)_{k}}
\label{eq:plancherel 6}
\end{equation}
and
\begin{equation}
\int_{-\infty}^{\infty}q^{\alpha^{2}}A_{q}\left(q^{\alpha-1/2}z\right)A_{q^{2}}\left(-q^{-2\alpha-1}z^{2}\right)
 d\alpha=\sqrt{\frac{\pi}{\log q^{-1}}}\sum_{k=0}^{\infty}\frac{\left(-z\right)^{k}q^{k^{2}/4}}{\left(q;q\right)_{k}}.\label{eq:plancherel 7}
\end{equation}
 For $\theta\in[0,\pi]$ and $\left|t\right|<1$ we have

\begin{eqnarray}
 &  & \int_{-\infty}^{\infty}\frac{q^{\alpha^{2}}\left(-te^{i\theta}q^{\alpha+1};q\right)_{\infty}
  A_{q}\left(-te^{-i\theta}q^{-\alpha}\right)}{\left(t^{2}q^{2};q^{2}\right)_{\infty}}
   d\alpha\label{eq:plancherel 8}\\
 & = & \sqrt{\frac{\pi}{\log q^{-1}}}\sum_{k=0}^{\infty}   
 \frac{q^{2k^{2}}\left(-t^{2}e^{-2i\theta}\right)^{k}\mathcal{E}_{q}
 \left(\cos\theta;tq^{k+1}\right)}{\left(q^{2},t^{2}q^{2};q^{2}\right)_{k}},\nonumber 
\end{eqnarray}
\begin{eqnarray}
 &  & \int_{-\infty}^{\infty}q^{2\alpha^{2}}A_{q}\left(-q^{2\alpha-1}te^{i\theta}\right)\left(-q^{-2\alpha}te^{-i\theta};q^{2}\right)_{\infty}d\alpha\label{eq:plancherel 9}\\
 & = & \sqrt{\frac{\pi}{\log q^{-2}}}\sum_{k=0}^{\infty}\frac{\left(te^{i\theta}\right)^{k}q^{k^{2}/2}}{\left(q^{2};q^{2}\right)_{k}}\left(t^{2}q^{2k};q^{4}\right)_{\infty}\mathcal{E}_{q^{2}}\left(\cos\theta;tq^{k-1}\right),\nonumber 
\end{eqnarray}
 
\begin{eqnarray}
 &  & \int_{-\infty}^{\infty}q^{\alpha^{2}}\left(t^{2}q^{2+2\alpha};q^{4}\right)_{\infty}\mathcal{E}_{q^{2}}\left(x;tq^{\alpha}\right)A_{q}\left(q^{-\alpha-1/2}te^{i\theta}\right)d\alpha\nonumber \\
 & = & \sqrt{\frac{\pi}{\log q^{-1}}}\sum_{k=0}^{\infty}\frac{q^{k^{2}/4}t^{k}}{\left(q^{2};q^{2}\right)_{k}}\left(qe^{2i\theta};q^{2}\right)_{k}e^{-ik\theta},\label{eq:plancherel 10}
\end{eqnarray}
\begin{eqnarray}
\label{eq:plancherel 11}
 \bg
  \sqrt{\frac{\pi}{\log q^{-1}}}\left(t^{2}q;q^{2}\right)_{\infty}\mathcal{E}_{q}
 \left(\cos\theta;t\right)\\
  =  \int_{-\infty}^{\infty}q^{\alpha^{2}}\left(-ate^{i\theta}q^{1/2+\alpha},-te^{-i\theta}q^{1/2-\alpha};q\right)_{\infty}  
  {} _{1}\phi_{1}\left(a;-te^{-i\theta}q^{1/2-\alpha};q,-te^{i\theta}q^{1/2-\alpha}\right)d\alpha, 
\eg
\end{eqnarray}
 \textup{
\begin{eqnarray}
\label{eq:plancherel 12}
 \bg
  \sqrt{\frac{\pi}{\log q^{-1}}}\left(t^{2}q;q^{2}\right)_{\infty}
  \mathcal{E}_{q}\left(\cos\theta;t\right) \\
  =  \int_{-\infty}^{\infty}{}_{1}\phi_{1}\left(w/z;-te^{i\theta}q^{\alpha+1/2};q,zq^{1/2+\alpha}\right)  
   {} _{1}\phi_{1}\left(z/w;-te^{-i\theta}q^{1/2-\alpha};q,wq^{1/2-\alpha}\right) 
   \\ \times 
     \left(-te^{i\theta}q^{1/2+\alpha},-te^{-i\theta}q^{1/2-\alpha};q\right)_{\infty}q^{\alpha^{2}}d\alpha,
  \eg
\end{eqnarray}
}
\begin{eqnarray}
\sqrt{\frac{\pi}{\log q^{-1}}}\left(t^{2}q;q^{2}\right)_{\infty}\mathcal{E}_{q}\left(\cos\theta;t\right) 
& = & \int_{-\infty}^{\infty}q^{\alpha^{2}}\left(-cte^{-i\theta}q^{1/2+\alpha},-ate^{i\theta}
 q^{1/2-\alpha};q\right)_{\infty}
 \label{eq:plancherel 13}\\
 & \times & _{1}\phi_{1}\left(a;-cte^{-i\theta}q^{1/2+\alpha};q,-te^{i\theta}q^{1/2+\alpha}\right)\nonumber \\
 & \times & _{1}\phi_{1}\left(c;-ate^{i\theta}q^{1/2-\alpha};q,-te^{-i\theta}q^{1/2-\alpha}\right)d\alpha,
  \nonumber 
\end{eqnarray}
  
\begin{eqnarray}
\bg
J_{\nu}^{(2)}\left(2w;q^{2}\right)  
 \sqrt{\frac{\log q^{-1}}{\pi}}\frac{w^{\nu}}{\left(q;q\right)_{\infty}}
 \int_{-\infty}^{\infty}q^{\alpha^{2}}\left(-iq^{\alpha+\nu+1};q\right)_{\infty}
\\ \times   A_{q}\left(iwq^{\nu-\alpha}\right)
  {}_{1}\phi_{1}\left(a;-iq^{\alpha+\nu+1};q,iq^{\alpha+\nu+1}\right)d\alpha,
  \eg
\label{eq:plancherel 14}
\end{eqnarray}
 
\begin{eqnarray}
 \bg
  \int_{-\infty}^{\infty}q^{\alpha^{2}+\alpha\left(\nu_{1}-\nu_{2}\right)}
  J_{\nu_{1}+\alpha}^{(2)}\left(2iq^{\left(\nu_{2}-\nu_{1}\right)/2};q\right)
   J_{\nu_{2}-\alpha}^{(2)}\left(2iq^{\left(\nu_{1}-\nu_{2}\right)/2};q\right)
    d\alpha  \\
  =  \sqrt{\frac{\pi}{\log q^{-1}}}\frac{e^{\pi i\left(\nu_{1}+\nu_{2}\right)/2}}
  {\left(q;q\right)_{\infty}^{2}q^{\left(\nu_{1}-\nu_{2}\right)^{2}/2}},
  \eg
   \label{eq:plancherel 15}
\end{eqnarray}
\end{thm}

The proofs of Theorems \ref{thm5.7.1} and  \ref{thm5.7.2} are left as 
an exercise for the reader. 

The following beta type integral will once again confirm the fact
that $A_{q}$ is just another q-analogue of exponential function.
\begin{thm}
\label{airybeta}
For $\left|u\right|,\left|v\right|<1$ we have
\begin{equation}
\int_{-\infty}^{\infty}A_{q}\left(q^{2\alpha}\overline{u}\right)A_{q}
\left(vq^{2\alpha}\right)q^{2\alpha^{2}}d\alpha
=\sqrt{\frac{\pi}{\log q^{-2}}}
\frac{\left(q^{1/2}\overline{u},q^{1/2}v;q\right)_{\infty}}
{\left(\overline{u}v;q\right)_{\infty}}.
\end{equation}
 \end{thm}
\begin{proof}
From \ref{eq:fourier airy 11} we obtain 
\begin{eqnarray*}
&{}& \int_{-\infty}^{\infty}A_{q}(q^{2\alpha}\overline{u})A_{q}(vq^{2\alpha})q^{2\alpha^{2}}d\alpha  =  \frac{1}{\log q^{-2}}\int_{-\infty}^{\infty}\frac{\exp\left(\frac{x^{2}}{\log q^{2}}\right)}
{\left(-\overline{u}e^{-ix},-ve^{ix};q\right)_{\infty}} \qquad \qquad \\
 & =&  \frac{1}{\log q^{-2}}\sum_{j,k=0}^{\infty}\frac{(-\overline{u})^{j}(-v)^{k}}{(q;q)_{j}(q;q)_{k}}\int_{-\infty}^{\infty}\exp\left(\frac{x^{2}}{\log q^{2}}+i(k-j)x\right)dx
  =  \sqrt{\frac{\pi}{\log q^{-2}}}\sum_{j,k=0}^{\infty}\frac{(-\overline{u})^{j}(-v)^{k}}{(q;q)_{j}(q;q)_{k}}q^{(j-k)^{2}/2}\\
 & = & \sqrt{\frac{\pi}{\log q^{-2}}}\sum_{j=0}^{\infty}
 \frac{(-\overline{u})^{j}q^{j^{2}/2}}{(q;q)_{j}}\sum_{k=0}^{\infty}
 \frac{q^{\binom{k}{2}}\left(-q^{1/2-j}v\right)^{k}}{\left(q;q\right)_{k}} 
  =  \sqrt{\frac{\pi}{\log q^{-2}}}\sum_{j=0}^{\infty} 
  \frac{(-\overline{u})^{j}q^{j^{2}/2}}{(q;q)_{j}}
  \left(q^{1/2-j}v;q\right)_{\infty}\\
 & = & \sqrt{\frac{\pi}{\log q^{-2}}}\left(q^{1/2}v;q\right)_{\infty}
 \sum_{j=0}^{\infty}\frac{(\overline{u}v)^{j}}{(q;q)_{j}}
 \left(q^{1/2}v^{-1};q\right)_{j}
  =  \sqrt{\frac{\pi}{\log q^{-2}}}
  \frac{\left(q^{1/2}\overline{u},q^{1/2}v;q\right)_{\infty}}
  {\left(\overline{u}v;q\right)_{\infty}}.
\end{eqnarray*}
This proves our theorem. 
 \end{proof}

\section{Plancherel-Rotach Asymptotics Via Fourier Integral Representations}
\subsection{An Inequality}
In this section we prove the following lemma which contains useful estimates. 
\begin{lem}
\label{lem:q-exponential approximations}Given $q\in\left(0,1\right)$
we have 
\begin{equation}
\left|\frac{1}{\left(z;q\right)_{\infty}}-1\right|\le\frac{2\left|z\right|}{1-q}\label{eq:eq approximation}
\end{equation}
 for all $|z|<\frac{1-q}{2}$ and 
\begin{equation}
\left|\left(z;q\right)_{\infty}-1\right|\le\frac{2\left|z\right|}{1-q}\label{eq:Eq approximation}
\end{equation}
 for all $z\in\mathbb{C}$.\end{lem}
\begin{proof}
Observe that for $|z|<\frac{1-q}{2}$ and $q\in\left(0,1\right)$
we have 
\[
0<\frac{\left(1-q\right)^{n}}{\left(q;q\right)_{n}}<1
\]
 and 
\[
\left|\frac{1}{\left(z;q\right)_{\infty}}-1\right|=\left|\sum_{n=1}^{\infty}\frac{z^{n}}{\left(q;q\right)_{n}}\right|\le\sum_{n=1}^{\infty}\left(\frac{\left|z\right|}{1-q}\right)^{n}\le\frac{2\left|z\right|}{1-q}.
\]
 On the other hand, we always have 
\[
\left|\frac{q^{\binom{n}{2}}}{\left(q;q\right)_{n}}\right|\le\frac{1}{k!\left(1-q\right)^{k}}
\]
 and 
\[
\left|\left(z;q\right)_{\infty}-1\right|=\left|\sum_{n=1}^{\infty}
\frac{q^{\binom{n}{2}}}{\left(q;q\right)_{n}}\left(-z\right)^{n}\right|\le
\sum_{n=1}^{\infty}\frac{1}{k!}\left(\frac{\left|z\right|}{1-q}\right)^{k}=
\frac{\left|z\right|}{1-q}\exp\left(\frac{\left|z\right|}{1-q}\right).
\]
\end{proof}

\subsection{Asymptotics of Orthogonal Polynomials}  
In this section we use the integral representations of in  Section 5 to 
derive asymptotic expansions of the Stieltjes-Wigert, $q^{-1}$-Hermite and $q$-Laguerre polynomials. Our  results are stead as equations  \eqref{eq*1}, 
 \eqref{eqS2n}, \eqref{eqS2n+1},  \eqref{eqaymbhn1}, \eqref{eqaymbh2n}, \eqref{eqaymbh2n+1},    
   \eqref{eqqLagasym1},  \eqref{eqqLagasym2}, \eqref{eqqLagasym2n}, and  \eqref{eqqLagasym2n+1}. 

We first consider the Stieltjes--Wigert polynomials.  It is easy to see that 
\begin{eqnarray*}
\frac{\left(xe^{iy}q^{1/2};q\right)_{n}}{\left(q;q\right)_{n}}  
 =  \frac{\left(xe^{iy}q^{1/2};q\right)_{\infty}}{\left(q;q\right)_{\infty}}
  \frac{\left(q^{n+1};q\right)_{\infty}}{\left(xe^{iy}q^{1/2+n};q\right)_{\infty}}\\
  =  \frac{\left(xe^{iy}q^{1/2};q\right)_{\infty}}{\left(q;q\right)_{\infty}}
   \left\{ 1+\mathcal{O}\left(q^{n}\right)\right\} 
\end{eqnarray*}
 as $n\to\infty$, uniformly for $y\in\mathbb{R}$ and $x$ in any
compact subset of complex plane.  This observation and \eqref{eq:fouriersw 1} lead to 
\begin{eqnarray*}
S_{n}\left(x;q\right) & = & \frac{1}{\sqrt{\pi\log q^{-2}}}\int_{-\infty}^{\infty}
  \frac{\left(xe^{iy}q^{1/2};q\right)_{n}}{\left(q;q\right)_{n}}
   \exp\left(\frac{y^{2}}{\log q^{2}}\right)dy\\
 & = & \left\{ \int_{-\infty}^{\infty}\frac{\left(xe^{iy}q^{1/2};q\right)_{\infty}}
  {\sqrt{\pi\log q^{-2}}}\exp\left(\frac{y^{2}}{\log q^{2}}\right)dy\right\}  
 \frac{1}{\left(q;q\right)_{\infty}}\left\{ 1+\mathcal{O}\left(q^{n}\right)\right\}.  
 \end{eqnarray*}
 Therefore we proved that 
 \begin{eqnarray} 
 \label{eq*1}
 S_{n}\left(x;q\right) =  \frac{A_{q}\left(x\right)}{\left(q;q\right)_{\infty}}
   \left\{ 1+\mathcal{O}\left(q^{n}\right)\right\}.  
\end{eqnarray}
 as $n\to\infty$, uniformly for $x$ in any compact subset of the
complex plane.  Similarly, from \eqref{eq:fouriersw 5} and \eqref{eqJtriple}
we  obtain 
\begin{eqnarray*}
&{}& q^{n^{2}/2}S_{2n}\left(xq^{-2n};q\right)  =  \frac{1}{\sqrt{\pi\log q^{-2}}}
 \int_{-\infty}^{\infty}\frac{\left(xe^{iy}q^{1/2-n};q\right)_{2n}}{\left(q;q\right)_{2n}}
  \exp\left(\frac{y^{2}}{\log q^{2}}-iny\right)dy\\
 &  & \quad =  \frac{\left(-x\right)^{n}q^{-n^{2}/2}}{\sqrt{\pi\log q^{-2}}}
  \int_{-\infty}^{\infty}\frac{\left(q^{1/2}/\left(xe^{iy}\right),q^{1/2}\left(xe^{iy}\right);q\right)_{n}}
    {\left(q;q\right)_{2n}}\exp\left(\frac{y^{2}}{\log q^{2}}\right)dy\\
 & &  \quad=  \int_{-\infty}^{\infty}\left(q,q^{1/2}/\left(xe^{iy}\right),q^{1/2}\left(xe^{iy}\right);q\right)_{\infty}
  \exp\left(\frac{y^{2}}{\log q^{2}}\right)dy\\
 & &  \qquad \qquad   \times  \frac{\left(-x\right)^{n}q^{-n^{2}/2}}
   {\left(q;q\right)_{\infty}^{2}\sqrt{\pi\log q^{-2}}}\left\{ 1+\mathcal{O}\left(q^{n}\right)\right\} \\
 &   &  \quad=  \left\{ \sum_{k=-\infty}^{\infty}q^{k^{2}/2}\left(-x\right)^{k}\int_{-\infty}^{\infty}
  \frac{\exp\left(\frac{y^{2}}{\log q^{2}}+iky\right)}{\sqrt{\pi\log q^{-2}}}dy\right\} \; 
\frac{\left(-x\right)^{n}q^{-n^{2}/2}}{\left(q;q\right)_{\infty}^{2}}
 \left\{ 1+\mathcal{O}\left(q^{n}\right)\right\} \\
 &   & \quad  =  \frac{\left(-x\right)^{n}q^{-n^{2}/2}\sum_{k=-\infty}^{\infty}q^{k^{2}}
   \left(-x\right)^{k}}{\left(q;q\right)_{\infty}^{2}}\left\{ 1+\mathcal{O}\left(q^{n}\right)\right\}.  
\end{eqnarray*}
 This shows that 
\begin{eqnarray}
\label{eqS2n}
\frac{S_{2n}\left(xq^{-2n};q\right)q^{n^{2}}}{\left(-x\right)^{n}}
  =\frac{\left(q^{2},qx,q/x;q^{2}\right)_{\infty}}{\left(q;q\right)_{\infty}^{2}}
    \left\{ 1+\mathcal{O}\left(q^{n}\right)\right\} 
\end{eqnarray}
 as $n\to\infty$, uniformly for $x$ in a compact subset of the punctured
complex plane $\mathbb{C}\backslash\left\{ 0\right\} $. Similarly the asymptotic result 
\begin{eqnarray}
\label{eqS2n+1}
\frac{S_{2n+1}\left(xq^{-2n};q\right)q^{n^{2}}}{\left(-x\right)^{n}}
  =\frac{\left(q^{2},qx,q/x;q^{2}\right)_{\infty}}{\left(q;q\right)_{\infty}^{2}}
    \left\{ 1+\mathcal{O}\left(q^{n}\right)\right\},  
\end{eqnarray}
holds,  as $n\to\infty$, uniformly for $x$ in any  compact subset of the punctured
complex plane $\mathbb{C}\backslash\left\{ 0\right\}$.

\begin{rem}
All the results derived so far can be obtained from the power series definition of the 
Stieltjes-Wigert polynomials, \cite{Ism}, but the error term is $\mathcal{O}(q^{cn})$ for $c < 1$. 
It does not seem possible to use the series definition and get an error term which is $\mathcal{O}(q^{n})$. 
The same remark applies to the remaining cases of the $q^{-1}$-Hermite and $q$-Laguerre polynomials. 
\end{rem}

We now consider the $q^{-1}$-Hermite polynomials. 

It is clear that 
\begin{eqnarray*}
 \left(e^{2\xi-iy}q^{1/2};q\right)_{n}  =  \frac{\left(e^{2\xi-iy}q^{1/2};q\right)_{\infty}}{\left(e^{2\xi-iy}q^{1/2+n};q\right)_{\infty}} = \left(e^{2\xi-iy}q^{1/2};q\right)_{\infty}
 \left\{ 1+\mathcal{O}\left(q^{n}\right)\right\} 
\end{eqnarray*}
 as $n\to\infty$, uniformly for $y\in\mathbb{R}$ and $\xi$ in any
compact subset of the complex plane, put $\alpha=-n$ in \eqref{eq:fourierh 5}
to obtain 
\begin{eqnarray*}
&{}& q^{n^{2}/2}\left(-e^{\xi}\right)^{n}h_{n}\left(\sinh\left(\xi+\log q^{n/2}\right)\vert q\right)  \\
&{}& \qquad =  \frac{1}{\sqrt{\pi\log q^{-2}}}\int_{-\infty}^{\infty}\exp\left(\frac{y^{2}}
 {\log q^{2}}\right)\left(e^{2\xi-iy}q^{1/2};q\right)_{n}dy\\
 &  &  \qquad =  \frac{\left\{ 1+\mathcal{O}\left(q^{n}\right)\right\} }{\sqrt{\pi\log q^{-2}}}
  \int_{-\infty}^{\infty}\exp\left(\frac{y^{2}}{\log q^{2}}\right)\left(e^{2\xi-iy}q^{1/2};q\right)_{\infty}dy. 
\end{eqnarray*}
 This implies  that 
\bea
\label{eqaymbhn1}
q^{n^{2}/2}\left(-e^{\xi}\right)^{n}h_{n}\left(\sinh\left(\xi+\log q^{n/2}\right)\vert q\right)=A_{q}\left(e^{2\xi}\right)\left\{ 1+\mathcal{O}\left(q^{n}\right)\right\} 
\eea
 as $n\to\infty$, uniformly for $\xi$ in any compact subset of the
complex plane.   In \eqref{eq:fourierh 5}  we reparametrize the parameters as 
\[
n \to 2n,\alpha\to n,\xi\to\xi+\log q^{-n/2}
\]
and apply \eqref{eqJtriple}.  Then we find that  

\begin{eqnarray*}
&{}& q^{n^{2}/2}e^{2n\xi}h_{2n}\left(\sinh\left(\xi\right)\vert q\right) \\
&{}& \qquad =   \frac{1}{\sqrt{\pi\log q^{-2}}}\int_{-\infty}^{\infty}
 \exp\left(\frac{y^{2}}{\log q^{2}}-iny\right)\left(e^{2\xi+iy}q^{1/2-n};q\right)_{2n}dy\\
 &  & \qquad =  \frac{\left(-e^{2\xi}\right)^{n}q^{-n^{2}/2}}{\sqrt{\pi\log q^{-2}}}
  \int_{-\infty}^{\infty}\exp\left(\frac{y^{2}}{\log q^{2}}\right)
   \left(q^{1/2}/\left(e^{2\xi+iy}\right),q^{1/2}\left(e^{2\xi+iy}\right);q\right)_{n}dy\\
 &  & \qquad =   \int_{-\infty}^{\infty}\exp\left(\frac{y^{2}}{\log q^{2}}\right)
  \left(q,q^{1/2}/\left(e^{2\xi+iy}\right),q^{1/2}\left(e^{2\xi+iy}\right);q\right)_{\infty}dy\\
 & &  \qquad \qquad \times \frac{\left(-e^{2\xi}\right)^{n}q^{-n^{2}/2}}
  {\left(q,q\right)_{\infty}\sqrt{\pi\log q^{-2}}}\left\{ 1+\mathcal{O}\left(q^{n}\right)\right\} \\
 &  & \qquad =   \frac{\left\{ \sum_{k=-\infty}^{\infty}q^{k^{2}}\left(-e^{2\xi}\right)^{k}\right\} 
  \left(-e^{2\xi}\right)^{n}\left\{ 1+\mathcal{O}\left(q^{n}\right)\right\} }{\left(q,q\right)_{\infty}
   q^{n^{2}/2}}.
\end{eqnarray*}
 Thus we found the asymptotic relationship 
\bea
\label{eqaymbh2n}
q^{n^{2}}\left(-1\right)^{n}h_{2n}\left(\sinh\left(\xi\right)\vert q\right)=\frac{\left(q^{2},qe^{2\xi},qe^{-2\xi};q^{2}\right)_{\infty}
  \left\{ 1+\mathcal{O}\left(q^{n}\right)\right\} }
 {\left(q,q\right)_{\infty}},
\eea
 as $n\to\infty$, uniformly for $\xi$ in any compact subset of the
complex plane. If instead we use the parameterization 
\[
n\to2n+1,\quad\alpha\to-n-1,\quad\xi\to\xi+\log q^{-n/2}
\]
 in \eqref{eq:fourierh 5} we  obtain
\begin{eqnarray*}
\bg
h_{2n+1}\left(\sinh\left(\xi+\log q^{1/2}\right)\vert q\right)  \qquad \\
=  \frac{-q^{-\left(n+1\right)^{2}/2}e^{-\left(2n+1\right)\xi}}{\sqrt{\pi\log q^{-2}}}
 \int_{-\infty}^{\infty}\exp\left(\frac{y^{2}}{\log q^{2}}-iny\right)
  \left(e^{2\xi+iy}q^{1/2-n};q\right)_{2n+1}dy\\
  =  \frac{\left(-1\right)^{n-1}e^{-\xi}}{q^{n^{2}+n+1/2}\sqrt{\pi\log q^{-2}}}\int_{-\infty}^{\infty}
   \exp\left(\frac{y^{2}}{\log q^{2}}\right)\left(q^{1/2}/\left(e^{2\xi+iy}\right);q\right)_{n}
    \left(q^{1/2}\left(e^{2\xi+iy}\right);q\right)_{n+1}dy\\
  =  \int_{-\infty}^{\infty}\exp\left(\frac{y^{2}}{\log q^{2}}\right)
   \left(q,q^{1/2}/\left(e^{2\xi+iy}\right),q^{1/2}\left(e^{2\xi+iy}\right);q\right)_{\infty}dy\\
  \times  \frac{\left(-1\right)^{n-1}e^{-\xi}\left\{ 1+\mathcal{O}\left(q^{n}\right)\right\} }
   {\left(q,q\right)_{\infty}q^{n^{2}+n+1/2}\sqrt{\pi\log q^{-2}}}\\
  =  \frac{ \left(-1\right)^{n-1}e^{-\xi}} {\left(q,q\right)_{\infty}q^{n^{2}+n+1/2}}
  \left\{ \sum_{k=-\infty}^{\infty}q^{k^{2}}\left(-e^{2\xi}\right)^{k}\right\}
   \left\{ 1+\mathcal{O}\left(q^{n}\right)\right\} 
\eg
\end{eqnarray*}
We sum the series by the Jacobi triple product identity and conclude that  
\begin{eqnarray}
\bg
q^{n^{2}+n+1/2}\left(-1\right)^{n-1}e^{\xi}h_{2n+1}\left(\sinh\left(\xi+\log q^{1/2}\right)\vert q\right) 
\\
=\frac{\left(q^{2},qe^{2\xi},qe^{-2\xi};q^{2}\right)_{\infty}
 \left\{ 1+\mathcal{O}\left(q^{n}\right)\right\} }{\left(q,q\right)_{\infty}}
\eg
\label{eqaymbh2n+1}
\end{eqnarray}
 as $n\to\infty$, uniformly for $\xi$ in any compact set of the
complex plane.

We now study the asymptotics of  the  $q$-Laguerre polynomials.  
Given $\alpha>-\frac{1}{2}$, observe that 
\begin{eqnarray*}
&{}& \frac{\left(xe^{iy}q^{\alpha+1/2};q\right)_{n}}{\left(q;q\right)_{n}} 
  =  \frac{\left(xe^{iy}q^{\alpha+1/2};q\right)_{\infty}}{\left(q;q\right)_{\infty}}
    \frac{\left(q^{n+1};q\right)_{\infty}}{\left(xe^{iy}q^{\alpha+1/2+n};q\right)_{\infty}}\\
 &{} & \qquad =   \frac{\left(xe^{iy}q^{\alpha+1/2};q\right)_{\infty}}{\left(q;q\right)_{\infty}}
  \left\{ 1+\mathcal{O}\left(q^{n}\right)\right\} 
\end{eqnarray*}
 as $n\to\infty$, uniformly for $y\in\mathbb{R}$ and $x$ in any
compact set of the complex plane.  Then from \eqref{eq:fourier pair laguerre 5} we obtain 
\begin{eqnarray*}
&{}& \left(q^{\alpha+n+1};q\right)_{\infty}L_{n}^{(\alpha)}\left(x;q\right)  \\
  &{} & \qquad  = \frac{1}{\sqrt{\pi\log q^{-2}}}\int_{-\infty}^{\infty}
   \frac{\left(xe^{iy}q^{\alpha+1/2};q\right)_{n}}{\left(q;q\right)_{n}}
    \frac{\exp\left(\frac{y^{2}}{\log q^{2}}\right)}{\left(-q^{\alpha+1/2}e^{iy};q\right)_{\infty}}dy\\
 & {} &  \qquad = \int_{-\infty}^{\infty}\frac{\left(xe^{iy}q^{\alpha+1/2};q\right)_{\infty}
  \exp\left(\frac{y^{2}}{\log q^{2}}\right)}{\sqrt{\pi\log q^{-2}}
   \left(q,-q^{\alpha+1/2}e^{iy};q\right)_{\infty}}dy
 \left\{ 1+\mathcal{O}\left(q^{n}\right)\right\} \\
  &{} & \qquad  = J_{\alpha}^{(2)}\left(2\sqrt{x};q\right)x^{-\alpha/2}
   \left\{ 1+\mathcal{O}\left(q^{n}\right)\right\}. 
\end{eqnarray*}
 Thus we have established the asymptotic result 
\bea
\label{eqqLagasym1}
L_{n}^{(\alpha)}\left(x;q\right)=J_{\alpha}^{(2)}\left(2\sqrt{x};q\right)x^{-\alpha/2}
 \left\{ 1+\mathcal{O}\left(q^{n}\right)\right\} 
\eea
 as $n\to\infty$, uniformly for $y\in\mathbb{R}$ and $x$ in any
compact set of the complex plane cut along the closed negative real axis.

Let 
\[
\beta \to-n,\quad\alpha\to\alpha+n,\quad x\to xq^{-2n-\alpha}
\]
 in \eqref{eq:fourier pair laguerre 5} to obtain 
\begin{eqnarray*}
&{}& q^{n^{2}/2}\left(q^{\alpha+n+1};q\right)_{\infty}L_{n}^{(\alpha)}\left(xq^{-2n-\alpha};q\right) \\
 &{}& \qquad \qquad =  \frac{1}{\sqrt{\pi\log q^{-2}}}\int_{-\infty}^{\infty}
  \frac{\left(xe^{iy}q^{1/2-n};q\right)_{n}}{\left(q;q\right)_{n}}
   \frac{\exp\left(\frac{y^{2}}{\log q^{2}}-iny\right)}
  {\left(-q^{\alpha+1/2+n}e^{iy};q\right)_{\infty}}dy \\
 & {} & \qquad \qquad  = \frac{\left(-x\right)^{n}q^{-n^{2}/2}}{\sqrt{\pi\log q^{-2}}}
  \int_{-\infty}^{\infty}\frac{\left(q^{1/2}/\left(xe^{iy}\right);q\right)_{n}}{\left(q;q\right)_{n}}
  \frac{\exp\left(\frac{y^{2}}{\log q^{2}}\right)}{\left(-q^{\alpha+1/2+n}e^{iy};q\right)_{\infty}}dy\\
 &  & \qquad \qquad  =  \int_{-\infty}^{\infty}\frac{\left(q^{1/2}/\left(xe^{iy}\right);q\right)_{\infty}
  \exp\left(\frac{y^{2}}{\log q^{2}}\right)}{\sqrt{\pi\log q^{-2}}}dy\\
 &  &\qquad \qquad \qquad  \times  \frac{\left(-x\right)^{n}q^{-n^{2}/2}}{\left(q;q\right)_{\infty}}
  \left\{ 1+\mathcal{O}\left(q^{n}\right)\right\}.  
\end{eqnarray*}
 Therefore 
\begin{eqnarray}
\label{eqqLagasym2}
\frac{q^{n^{2}}L_{n}^{(\alpha)}\left(xq^{-2n-\alpha};q\right)}{\left(-x\right)^{n}}
  =\frac{A_{q}\left(x^{-1}\right)}{\left(q;q\right)_{\infty}}\left\{ 1+\mathcal{O}\left(q^{n}\right)\right\} 
\end{eqnarray}
 as $n\to\infty$, uniformly for $x$ in a compact set of the punctured
complex plane $\mathbb{C}\backslash\left\{ 0\right\} $.

We now let 
\[
n \to 2n,\quad  \beta \to  -n,  \quad\alpha\to\alpha+n,\quad x\to xq^{-2n-\alpha}
\]
 in \eqref{eq:fourier pair laguerre 5} and apply \eqref{eqJtriple} to get

\begin{eqnarray*}
&{}& q^{n^{2}/2}\left(q^{\alpha+2n+1};q\right)_{\infty}L_{2n}^{(\alpha)}\left(xq^{-2n-\alpha};q\right) \\
&{}  & \qquad \qquad  =  \frac{1}{\sqrt{\pi\log q^{-2}}}\int_{-\infty}^{\infty} 
 \frac{\left(xe^{iy}q^{1/2-n};q\right)_{2n}}{\left(q;q\right)_{2n}}\frac{\exp\left(\frac{y^{2}}
   {\log q^{2}}-iny\right)}{\left(-q^{\alpha+n+1/2}e^{iy};q\right)_{\infty}}dy\\
 & {} & \qquad \qquad    =\frac{\left(-x\right)^{n}q^{-n^{2}/2}}{\sqrt{\pi\log q^{-2}}}\int_{-\infty}^{\infty}\frac{\left(q^{1/2}/\left(xe^{iy}\right),q^{1/2}\left(xe^{iy}\right);q\right)_{n}}{\left(q;q\right)_{2n}}
   \frac{\exp\left(\frac{y^{2}}{\log q^{2}}\right)}{\left(-q^{\alpha+n+1/2}e^{iy};q\right)_{\infty}}dy\\
 & {} & \qquad \qquad  = \int_{-\infty}^{\infty}\frac{\left(q,q^{1/2}/\left(xe^{iy}\right),q^{1/2}
  \left(xe^{iy}\right);q\right)_{\infty}}{\sqrt{\pi\log q^{-2}}}\exp\left(\frac{y^{2}}{\log q^{2}}\right)dy\\
 &{} &   \qquad \qquad \qquad \times  \frac{\left(-x\right)^{n}q^{-n^{2}/2}}{\left(q;q\right)_{\infty}^{2}}
   \left\{ 1+\mathcal{O}\left(q^{n}\right)\right\} \\
 &  {} & \qquad \qquad = \left\{ \sum_{k=-\infty}^{\infty}q^{k^{2}}\left(-x\right)^{k}\right\} 
   \frac{\left(-x\right)^{n}q^{-n^{2}/2}}{\left(q;q\right)_{\infty}^{2}}
    \left\{ 1+\mathcal{O}\left(q^{n}\right)\right\}.  
\end{eqnarray*}
 This establishes the asymptotic result
\begin{eqnarray}
\label{eqqLagasym2n}
\frac{q^{n^{2}}L_{2n}^{(\alpha)}\left(xq^{-2n-\alpha};q\right)}{\left(-x\right)^{n}}
  =\frac{\left(q^{2},qx,q/x;q^{2}\right)_{\infty}}{\left(q;q\right)_{\infty}^{2}}
   \left\{ 1+\mathcal{O}\left(q^{n}\right)\right\} 
\end{eqnarray}
 as $n\to\infty$, uniformly for $x$ in a compact set of the punctured
complex plane $\mathbb{C}\backslash\left\{ 0\right\} $. Similarly, we  
make the parameter identification 
\[
n \to2n+1,\quad \beta \to  -n,\quad\alpha\to\alpha+n,\quad x\to xq^{-2n-\alpha}
\]
 in \eqref{eq:fourier pair laguerre 5} and discover that  
\begin{eqnarray*}
&{}& q^{n^{2}/2}\left(q^{\alpha+2n+2};q\right)_{\infty}L_{2n+1}^{(\alpha)}
 \left(xq^{-2n-\alpha};q\right) 
\\
& {}& \qquad = \frac{1}{\sqrt{\pi\log q^{-2}}}\int_{-\infty}^{\infty}
  \frac{\left(xe^{iy}q^{1/2-n};q\right)_{2n+1}}{\left(q;q\right)_{2n+1}}
    \frac{\exp\left(\frac{y^{2}}{\log q^{2}}-iny\right)}{\left(-q^{\alpha+1/2+n}e^{iy};q\right)_{\infty}}dy\\
 & {}& \qquad = \int_{-\infty}^{\infty}\frac{\left(q^{1/2}/\left(xe^{iy}\right);q\right)_{n}
   \left(q^{1/2}\left(xe^{iy}\right);q\right)_{n+1}\left(-x\right)^{n}}{q^{n^{2}/2}
    \left(q;q\right)_{2n+1}\sqrt{\pi\log q^{-2}}}
   \frac{\exp\left(\frac{y^{2}}{\log q^{2}}\right)}{\left(-q^{\alpha+1/2+n}e^{iy};q\right)_{\infty}}dy\\
 & {}& \qquad = \int_{-\infty}^{\infty}
  \frac{\left(q,q^{1/2}/\left(xe^{iy}\right),q^{1/2}\left(xe^{iy}\right);q\right)_{\infty}
   \exp\left(\frac{y^{2}}{\log q^{2}}\right)}{\sqrt{\pi\log q^{-2}}}dy\\
 & {} &\qquad \qquad \times  \frac{q^{-n^{2}/2}\left(-x\right)^{n}}{\left(q;q\right)_{\infty}^{2}}
  \left\{ 1+\mathcal{O}\left(q^{n}\right)\right\} \\
 & {}& \qquad = \left\{ \sum_{k=-\infty}^{\infty}q^{k^{2}}\left(-x\right)^{k}\right\} \frac{q^{-n^{2}/2}
  \left(-x\right)^{n}}{\left(q;q\right)_{\infty}^{2}}\left\{ 1+\mathcal{O}\left(q^{n}\right)\right\}. 
\end{eqnarray*}
 This gives the following asymptotic result
\begin{eqnarray}
\label{eqqLagasym2n+1}
\frac{q^{n^{2}}L_{2n+1}^{(\alpha)}\left(xq^{-2n-\alpha};q\right)}{\left(-x\right)^{n}}
 =\frac{\left(q^{2},qx,q/x;q^{2}\right)}{\left(q;q\right)_{\infty}^{2}}
  \left\{ 1+\mathcal{O}  \left(q^{n}\right)\right\} 
\end{eqnarray}
 as $n \to \infty$, uniformly for $x$ in a compact set of the punctured
complex plane $\mathbb{C}\backslash\left\{ 0\right\} $. Clearly \eqref{eqqLagasym2n+1}  is a companion formula to \eqref{eqqLagasym2n}. 

\subsection{Asymptotics of Transcendental Functions}

The functions covered in this section are the Ramanujan function $A_{q}\left(z\right)$, 
the $q$-Bessel function $J_\nu^{(2)}(z;q)$ and the ${}_1\phi_1$ function.  The asymptotic results 
of this section are \eqref{eqAqAsymp},  \eqref{eqJnuqAsymp1},   \eqref{eqJnuqAsymp2}, 
 \eqref{eqfirst1phi1asym}, and \eqref{eqsecond1phi1asym}.

We first consider the  function $A_{q}\left(z\right)$. 
Let $\alpha\to-n,\quad z\to zq^{-n}$ in \eqref{eq:fourier airy 9}
to get 
\begin{eqnarray*}
q^{n^{2}/2}A_{q}\left(zq^{-2n}\right) & = & \int_{-\infty}^{\infty}\frac{\left(zq^{1/2-n}e^{ix};q\right)_{\infty}\exp\left(\frac{x^{2}}{\log q^{2}}-inx\right)}{\sqrt{2\pi\log q^{-1}}}dx\\
 & = & \left(-z\right)^{n}\int_{-\infty}^{\infty}\frac{\left(q^{1/2}/\left(ze^{ix}\right);q\right)_{n}\left(q^{1/2}\left(ze^{ix}\right);q\right)_{\infty}\exp\left(\frac{x^{2}}{\log q^{2}}\right)}{q^{n^{2}/2}\sqrt{2\pi\log q^{-1}}}dx\\
 & = & \int_{-\infty}^{\infty}\frac{\left(q,q^{1/2}/\left(ze^{ix}\right),q^{1/2}\left(ze^{ix}\right);q\right)_{\infty}\exp\left(\frac{x^{2}}{\log q^{2}}\right)}{\sqrt{2\pi\log q^{-1}}}dx\\
 & &  \quad \times  \frac{\left(-z\right)^{n}}{\left(q;q\right)_{\infty}q^{n^{2}/2}}\left\{ 1+\mathcal{O}\left(q^{n}\right)\right\} \\
 & = & \left\{ \sum_{k=-\infty}^{\infty}q^{k^{2}}\left(-z\right)^{k}\right\} \frac{\left(-z\right)^{n}}{\left(q;q\right)_{\infty}q^{n^{2}/2}}\left\{ 1+\mathcal{O}\left(q^{n}\right)\right\} 
\end{eqnarray*}
 Therefore we established the asymptotic formula 
\begin{eqnarray}
\label{eqAqAsymp}
\frac{q^{n^{2}}A_{q}\left(zq^{-2n}\right)}{\left(-z\right)^{n}}=\frac{\left(q^{2},qz,q/z;q^{2}\right)_{\infty}}{\left(q;q\right)_{\infty}}\left\{ 1+ \mathcal{O}\left(q^{n}\right)\right\} 
\end{eqnarray}
 as $n\to\infty$, uniformly for $z$ in a compact subset of the punctured
complex plane $\mathbb{C}\backslash\left\{ 0\right\} $.

It must be noted that R. Zhang \cite{Zha1} used power series techniques to prove  a weaker version 
of \eqref{eqAqAsymp} where $\mathcal{O}\left(q^{n}\right)$ is replaced by 
$\mathcal{O}\left(q^{n/2}\right)$.  It may be possible to use power series techniques to obtain 
$\mathcal{O}\left(q^{n(1-\epsilon)}\right)$ for any $\epsilon \in (0,1)$ but we have been unable to 
prove that the error term is  $\mathcal{O}\left(q^{n}\right)$  through the use of power series techniques.

We now consider the $q$-Bessel function $J_{\nu}^{(2)}\left(z;q\right)$.   
Let 
\[
\alpha\to0,\quad\nu\to\nu+n,\quad z\to2\sqrt{w}q^{-\left(n+\nu\right)/2}
\]
 in \eqref{eq:fourier pair bessel 10} to obtain 
\begin{eqnarray*}
J_{n+\nu}^{(2)}\left(2\sqrt{w}q^{-\left(n+\nu\right)/2};q\right)\left(\frac{q^{n+\nu}}{w}\right)^{n/2} 
 & = & \frac{1}{\sqrt{2\pi\log q^{-1}}}\int_{-\infty}^{\infty}\frac{\left(wq^{1/2}e^{ix};q\right)_{\infty}
  \exp\left(\frac{x^{2}}{\log q^{2}}\right)}{\left(q,-q^{n+\nu+1/2}e^{ix};q\right)_{\infty}}dx\\
 & = & \frac{\left\{ 1+\mathcal{O}\left(q^{n}\right)\right\} }{\sqrt{2\pi\log q^{-1}}}
  \int_{-\infty}^{\infty}\frac{\left(wq^{1/2}e^{ix};q\right)_{\infty}\exp\left(\frac{x^{2}}
   {\log q^{2}}\right)}{\left(q;q\right)_{\infty}}dx. 
\end{eqnarray*}
 This leads to the asymptotic formula  
\bea
\label{eqJnuqAsymp1}
\frac{q^{n\left(n+\nu\right)/2}J_{n+\nu}^{(2)}\left(2\sqrt{w}q^{-\left(n+\nu\right)/2};q\right)}
 {w^{n/2}}=A_{q}\left(w\right)\left\{ 1+\mathcal{O}\left(q^{n}\right)\right\} 
\eea
 as $n\to\infty$, uniformly for $w$ in any compact set of the complex
plane with the segment $(-\infty+0i, 0]$ deleted. 

Let 
\[
\alpha\to-n\quad\nu\to\nu+n,\quad z\to2\sqrt{w}q^{-n-\nu/2}
\]
 in \eqref{eq:fourier pair bessel 10} and apply \eqref{eqJtriple} to obtain 
\begin{eqnarray*}
&{}&  q^{n^{2}/2}J_{\nu}^{(2)}\left(2\sqrt{w}q^{-n-\nu/2};q\right)
\left(\frac{q^{n+\nu/2}}{\sqrt{w}}\right)^{\nu}   \\
&{}& \qquad =   \frac{1}{\sqrt{2\pi\log q^{-1}}}\int_{-\infty}^{\infty}
 \frac{\left(wq^{1/2-n}e^{ix};q\right)_{\infty}\exp\left(\frac{x^{2}}{\log q^{2}}-nxi\right)}
 {\left(q,-q^{\nu+1/2+n}e^{ix};q\right)_{\infty}}dx\\
 & {} &   \qquad =   \frac{\left(-w\right)^{n}q^{-n^{2}/2}}{\sqrt{2\pi\log q^{-1}}}
  \int_{-\infty}^{\infty}\frac{\left(q^{1/2}/\left(we^{ix}\right);q\right)_{n}
   \left(q^{1/2}we^{ix};q\right)_{\infty}\exp\left(\frac{x^{2}}{\log q^{2}}\right)}
    {\left(q,-q^{\nu+1/2+n}e^{ix};q\right)_{\infty}}dx\\
 & {} &  \qquad =   \int_{-\infty}^{\infty}
  \frac{\left(q,q^{1/2}/\left(we^{ix}\right),q^{1/2}we^{ix};q\right)_{\infty}
   \exp\left(\frac{x^{2}}{\log q^{2}}\right)}{\left(q,q\right)_{\infty}^{2}\sqrt{2\pi\log q^{-1}}}dx\\
 & &   \qquad \qquad  \times \left(-w\right)^{n}q^{-n^{2}/2}
  \left\{ 1+\mathcal{O}\left(q^{n}\right)\right\} \\
 & {} &  \qquad =   \frac{\left\{ \sum_{k=-\infty}^{\infty}q^{k^{2}}
  \left(-w\right)^{k}\right\} \left(-w\right)^{n}q^{-n^{2}/2}
   \left\{ 1+\mathcal{O}\left(q^{n}\right)\right\} }{\left(q,q\right)_{\infty}^{2}}. 
\end{eqnarray*}
This proves the following asymptotic result 
\begin{eqnarray}
\label{eqJnuqAsymp2}
\frac{q^{n^{2}+n\nu+\nu^{2}/2}J_{\nu}^{(2)}\left(2\sqrt{w}q^{-n-\nu/2};q\right)}
 {\left(-1\right)^{n}w^{n+\nu/2}}=\frac{\left(q^{2},qw,q/w;q^{2}\right)_{\infty}}
  {\left(q,q\right)_{\infty}^{2}}\left\{ 1+\mathcal{O}\left(q^{n}\right)\right\} 
\end{eqnarray}
 as $n\to\infty$, uniformly for $w$ in any compact set of the complex
plane cut along negative real axis including the origin.

Finally we come to the asymptotics of  the ${}_{1}\phi_{1}$ function.  
With the choices 
\[
\alpha\to0,\quad z\to zq^{n},\quad b\to bq^{n},\quad a\to aq^{-n}
\]
 in \eqref{eq:fourier pair confluent 5}  it follows  that 
\begin{eqnarray*}
\bg
{}
_{1}\phi_{1}\left(\begin{array}{ccc}
\begin{array}{c}
aq^{-n}\\
bq^{n}
\end{array}  \vert  q,-zq^{1/2-n}\end{array}\right)   =  \frac{1}{\sqrt{\pi\log q^{-2}}}\int_{-\infty}^{\infty}\frac{\left(aze^{ix};q\right)_{\infty}\exp\left(\frac{x^{2}}{\log q^{2}}\right)dx}{\left(b,-bq^{n-1/2}e^{ix},zq^{n}e^{ix};q\right)_{\infty}}  \qquad \qquad \quad  \;\; \\
   \qquad  \; \;  
  =  \frac{\left\{ 1+\mathcal{O}\left(q^{n}\right)\right\} }{\sqrt{\pi\log q^{-2}}}\int_{-\infty}^{\infty}\frac{\left(aze^{ix};q\right)_{\infty}\exp\left(\frac{x^{2}}{\log q^{2}}\right)dx}{\left(b;q\right)_{\infty}}
  =  \frac{A_{q}\left(az\right)}{\left(b;q\right)_{\infty}}
   \left\{ 1+\mathcal{O}\left(q^{n}\right)\right\}.  
   \qquad \qquad \qquad  \qquad  \qquad
\eg
\end{eqnarray*}
In other words the asymptotic formula 
\bea
\label{eqfirst1phi1asym}
_{1}\phi_{1}\left(\begin{array}{ccc}
\begin{array}{c}
aq^{-n}\\
bq^{n}
\end{array} & \vert & q,-zq^{1/2-n}\end{array}\right)=\frac{A_{q}\left(az\right)}{\left(b;q\right)_{\infty}}\left\{ 1+\mathcal{O}\left(q^{n}\right)\right\} 
\eea
 holds  as $n\to\infty$, uniformly for $b$ in any compact subset of $\left|bq^{-1/2}\right|<1$,
$z$ in any compact subset of $\left|z\right|<1$ and $a$ in any
compact set of the complex plane.

Next we make the parameter identification 
\[
\alpha\to-n,\quad b\to bq^{n},\quad z\to zq^{n+1/2},\quad a\to aq^{-2n}
\]
 in \eqref{eq:fourier pair confluent 5} and find that 
\begin{eqnarray*}
\bg
_{1}\phi_{1}\left(\begin{array}{ccc}
\begin{array}{c}
aq^{-2n}\\
bq^{n}
\end{array}  \vert  q,-zq\end{array}\right)q^{n^{2}/2}   \qquad  \qquad \qquad  \qquad 
  \qquad  \qquad
 \qquad  \qquad  \\
 =  \frac{1}{\sqrt{\pi\log q^{-2}}}\int_{-\infty}^{\infty}
  \frac{\left(aze^{ix}q^{1/2-n};q\right)_{\infty}\exp\left(\frac{x^{2}}{\log q^{2}}-inx\right)dx}
   {\left(b,-bq^{n-1/2}e^{ix},zq^{n}e^{ix};q\right)_{\infty}}\\
  =  \frac{\left(-az\right)^{n}q^{-n^{2}/2}}{\sqrt{\pi\log q^{-2}}}\int_{-\infty}^{\infty}
    \frac{\left(q^{1/2}/\left(aze^{ix}\right);q\right)_{n}\left(q^{1/2}\left(aze^{ix}\right);q\right)_{\infty}
     \exp\left(\frac{x^{2}}{\log q^{2}}\right)dx}{\left(b,-bq^{n-1/2}e^{ix},zq^{n}e^{ix};q\right)_{\infty}}\\
  =  \int_{-\infty}^{\infty}\frac{\left(q,q^{1/2}/\left(aze^{ix}\right),q^{1/2}
   \left(aze^{ix}\right);q\right)_{\infty}\exp\left(\frac{x^{2}}{\log q^{2}}\right)dx}
    {\sqrt{\pi\log q^{-2}}}\\
  \times  \frac{\left(-az\right)^{n}q^{-n^{2}/2}}{\left(q,b;q\right)_{\infty}}
   \left\{ 1+\mathcal{O}\left(q^{n}\right)\right\} \\
  =  \left\{ \sum_{k=-\infty}^{\infty}q^{k^{2}}\left(-az\right)^{k}\right\} 
   \frac{\left(-az\right)^{n}q^{-n^{2}/2}}{\left(q,b;q\right)_{\infty}}
    \left\{ 1+\mathcal{O}\left(q^{n}\right)\right\} \\
  =  \frac{\left(q^{2},qaz,q/\left(az\right);q^{2}\right)_{\infty}
   \left(-az\right)^{n}}{\left(q,b;q\right)_{\infty}q^{n^{2}/2}}
    \left\{ 1+\mathcal{O}\left(q^{n}\right)\right\}.  
\eg
\end{eqnarray*}
 Therefore 
\bea
\label{eqsecond1phi1asym}
{}_{1}\phi_{1}\left(\begin{array}{ccc}
\begin{array}{c}
aq^{-2n}\\
bq^{n}
\end{array} & \vert & q,-zq\end{array}\right)\frac{q^{n^{2}}}{\left(-az\right)^{n}}
 =\frac{\left(q^{2},qaz,q/\left(az\right);q^{2}\right)_{\infty}}{\left(q,b;q\right)_{\infty}}
  \left\{ 1+\mathcal{O}\left(q^{n}\right)\right\},  
\eea
 as $n\to\infty$, uniformly for $b$ in any compact subset of $\left|bq^{-1/2}\right|<1$,
$z$ in any compact subset of $0<\left|z\right|<1$ and $a$ in any
compact set of the punctured complex plane $\mathbb{C}\backslash\left\{ 0\right\} $.

\section{Series  Representations and  Identities} 
In this section  we find several new identities involving the $q$-functions, 
${\mathcal{E}}_q(z;t),  A_q(z), J_\nu^{(2)}(z;q)$,  the polynomials $S_n, h_n, L_n^{(\al)}(z;a)$, 
as well as the basic confluent hypergeometric function ${}_1\phi_1$.   

\subsection{The Functions $\mathcal{E}_{q}\left(x;t\right)$ and $A_q(z)$}

An application of  \eqref{eq:2.32} and \eqref{eq:fourier11}  leads to 
\begin{eqnarray*}
&{}& \left(t^{2}q;q^{2}\right)_{\infty}\mathcal{E}_{q}\left(x;t\right)  
 =  \frac{1}{\sqrt{\pi\log q^{-1}}}\int_{-\infty}^{\infty}
  \frac{\exp\left(\frac{y^{2}}{\log q}\right)dy}
   {\left(te^{i\left(y+\theta\right)},te^{i\left(y-\theta\right)};q\right)_{\infty}}\\
 &{} & \quad =  \frac{1}{\sqrt{\pi\log q^{-1}}}\int_{-\infty}^{\infty}
  \frac{\left(-te^{i\left(y+\theta\right)};q\right)_{\infty}}
   {\left(te^{i\left(y-\theta\right)};q\right)_{\infty}}
    \frac{\exp\left(\frac{y^{2}}{\log q}\right)dy}{\left(t^{2}e^{2i\left(y+\theta\right)};q^{2}\right)_{\infty}}\\
 & {}& \quad =  \sum_{k=0}^{\infty}\frac{\left(-e^{2i\theta};q\right)_{k}}
  {\left(q;q\right)_{k}}\frac{\left(te^{-i\theta}\right)^{k}}{\sqrt{\pi\log q^{-1}}}\int_{-\infty}^{\infty}
   \frac{\exp\left(\frac{y^{2}}{\log q}+iky\right)dy}
    {\left(t^{2}e^{2i\left(y+\theta\right)};q^{2}\right)_{\infty}}\\
 & {} &\quad = 
  \sum_{k=0}^{\infty}\frac{\left(-e^{2i\theta};q\right)_{k}}
   {\left(q;q\right)_{k}}\frac{\left(te^{-i\theta}\right)^{k}}{\sqrt{\pi\log q^{-4}}}
    \int_{-\infty}^{\infty}\frac{\exp\left(\frac{y^{2}}{\log q^{4}}+iky/2\right)dy}
     {\left(t^{2}e^{2i\theta}e^{iy};q^{2}\right)_{\infty}}\\
 & {} & \quad =  \sum_{k=0}^{\infty}\frac{\left(-e^{2i\theta};q\right)_{k}}{\left(q;q\right)_{k}}
  \left(te^{-i\theta}\right)^{k}q^{k^{2}/4}\left(-t^{2}e^{2i\theta}q^{k+1};q^{2}\right)_{\infty}.
\end{eqnarray*}
The last line is 
\begin{equation}
\mathcal{E}_{q}\left(x;t\right)=\sum_{k=0}^{\infty}\frac{q^{k^{2}/4}}{\left(q;q\right)_{k}}
 \frac{\left(-t^{2}e^{2i\theta}q^{k+1};q^{2}\right)_{\infty}}{\left(t^{2}q;q^{2}\right)_{\infty}}
  \left(-e^{2i\theta};q\right)_{k}\left(te^{-i\theta}\right)^{k}.
\label{eqseries1}
\end{equation}

We do not know how to prove \eqref{eqseries1} in a different way. 

We now consider the Ramanujan Function $A_{q}\left(z\right)$. For $|a|<q^{-\frac{1}{2}}$, the integral representation  \eqref{eq:fourier airy 9}  leads to 
\begin{eqnarray*}
A_{q}\left(ab\right) & = & \int_{-\infty}^{\infty}
 \frac{\left(aq^{1/2}e^{ix},abq^{1/2}e^{ix};q\right)_{\infty}
  \exp\left(\frac{x^{2}}{\log q^{2}}\right)}{\left(aq^{1/2}e^{ix};q\right)_{\infty}\sqrt{\pi\log q^{-2}}}dx\\
 & = & \sum_{k=0}^{\infty}\frac{\left(b;q\right)_{k}}{\left(q;q\right)_{k}}
  \left(aq^{1/2}\right)^{k}\int_{-\infty}^{\infty}\frac{\left(aq^{1/2}e^{ix};q\right)_{\infty}
   \exp\left(\frac{x^{2}}{\log q^{2}}+ikx\right)}{\sqrt{\pi\log q^{-2}}}dx,
\end{eqnarray*}
that is
\begin{equation}
A_{q}\left(ab\right) = \sum_{k=0}^{\infty}\frac{\left(b;q\right)_{k}}{\left(q;q\right)_{k}}
 q^{\binom{k+1}{2}}a^{k}A_{q}\left(aq^{k}\right),
\label{eq:series and identities airy 1}
\end{equation}
in particular, 
\begin{equation}
1=\sum_{k=0}^{\infty}\frac{q^{\binom{k+1}{2}}}{\left(q;q\right)_{k}}a^{k}A_{q}\left(aq^{k}\right).\label{eq:series and identities airy 2}
\end{equation}

The expansion \eqref{eq:series and identities airy 1}  is what is called a multiplication formula 
for $A_q$. It    can be viewed as an expansion in the polynomial basis $\{(b;q)_k\}$. 
The theory of polynomial expansions in certain bases is developed in Boas and Buck's classic work 
\cite{Boa:Buc}.  Moreover it is a $q$-Taylor expansion where the operator used is $D_{q^{-1}}$, 
\cite{Ann:Man}, \cite{Gue}.  The function $A_q$ however does not satisfy the assumptions 
in the $q$-Taylor theorem, so it is interesting that  the $q$-Taylor expansion is valid in this case. 

\begin{proof}[Alternate Proof of \eqref{eq:series and identities airy 1}]
The right-hand side of \eqref{eq:series and identities airy 1} is 
\bea
\notag
\bg
\sum_{k,n=0}^\infty \frac{(b;q)_k \, a^{k+n}}{(q;q)_k(q;q)_n} (-1)^n q^{n^2+kn + \binom{k+1}{2}} 
= \sum_{s=0}^\infty \frac{(-a)^s}{(q;q)_s} q^{s^2} {}_2\phi_1(q^{-s}, b; 0; q,q)  
 = A_q(ab), 
\eg
\eea
where we used the $q$-analogue of the Chu-Vandermonde sum \eqref{eqqVander}. 
\end{proof}

For $|z|,|zw|<1$, from \eqref{eq:fourier airy 11} we get
\begin{eqnarray*}
A_{q}\left(z\right) & = & \int_{-\infty}^{\infty}\frac{\left(-zwe^{ix};q\right)_{\infty}\exp\left(\frac{x^{2}}
{\log q^{4}}\right)dx}{\left(-ze^{ix},-zwe^{ix};q\right)_{\infty}\sqrt{\pi\log q^{-4}}}\\
 & = & \sum_{k=0}^{\infty}\frac{\left(w;q\right)_{k}}{\left(q;q\right)_{k}}\int_{-\infty}^{\infty}
 \frac{\left(-z\right)^{k}\exp\left(\frac{x^{2}}{\log q^{4}}+ikx\right)dx}{\left(-zwe^{ix};q\right)_{\infty}
 \sqrt{\pi\log q^{-4}}}. 
\end{eqnarray*}
The result is 
\begin{equation}
A_{q}\left(z\right)=\sum_{k=0}^{\infty}\frac{q^{k^{2}}}{\left(q;q\right)_{k}}\left(-z\right)^{k}\left(w;q\right)_{k}A_{q}\left(wzq^{2k}\right).
\label{eq:series and identities airy 3}
\end{equation}
 One can also prove \eqref{eq:series and identities airy 3} directly as follows. 
 First expand $A_q(wzq^{2k})$ in power series then expand  $(w;q)_k$  using  the $q$-binomial 
 theorem in the form 
 \bea
 \label{eqqbinom}
 (z;q)_n = \sum_{k=0}^n \gauss{n}{k} (-z)^k q^{\binom{k}{2}}. 
 \eea
 Some manipulations show that the coefficient of $(-z)^s w^r$ in the right-hand side of 
 \eqref{eq:series and identities airy 3} is 
 \bea
 \notag
\frac{q^{s^2}}{(q;q)_r(q;q)_{s-r}} \sum_{j=0}^r \gauss{r}{j} (-1)^j q^{\binom{j}{2}}. 
 \eea
 In view of \eqref{eqqbinom}, the $j$-sum vanishes unless $r=0$ and 
 the identity \eqref{eq:series and identities airy 3}  follows.

 Similarly, for $|z|<1$ we have 
\begin{eqnarray*}
A_{q}\left(z\right) & = & \int_{-\infty}^{\infty}\frac{\exp\left(\frac{x^{2}}{\log q^{4}}\right)}{\left(-ze^{ix},-qze^{ix};q^{2}\right)_{\infty}\sqrt{\pi\log q^{-4}}}dx\\
 & = & \sum_{k=0}^{\infty}\frac{\left(-z\right)^{k}}{\left(q^2;q^2\right)_{k}}\int_{-\infty}^{\infty}\frac{\exp\left(\frac{x^{2}}{\log q^{4}}+ikx\right)}{\left(-qze^{ix};q^{2}\right)_{\infty}\sqrt{\pi\log q^{-4}}}dx\\
 & = & \sum_{k=0}^{\infty}\frac{\left(-z\right)^{k}q^{k^{2}}}{\left(q^2;q^2\right)_{k}}\left(zq^{2k+2};q^{2}\right)_{\infty}. 
\end{eqnarray*}
This proves that 
\begin{equation}
\frac{A_{q}\left(z\right)}{\left(zq^{2};q^{2}\right)_{\infty}}
=\sum_{k=0}^{\infty}\frac{\left(-z\right)^{k}q^{k^{2}}}
{\left(q^2,zq^2;q^2\right)_{k}}.
\label{eq:series and identities airy 4}
\end{equation}
It is clear that both sides of \eqref{eq:series and identities airy 4} are 
analytic in $|z|<q^{-2}$, then the equation  \eqref{eq:series and 
identities airy 4} holds inside this larger open disk.

\subsection{The Polynomials $S_{n}(x;q)$ and $L_n^{(\al)}(x;q)$}

From \eqref{eq:fouriersw 5} and \eqref{eq:fourier11} to get
\begin{eqnarray*}
S_{n}\left(x;q\right)\left(q;q\right)_{n} & = & \frac{1}{\sqrt{\pi\log q^{-2}}}\int_{-\infty}^{\infty}\frac{\left(xq^{1/2}e^{iy};q\right)_{\infty}}{\left(xq^{1/2+n}e^{iy};q\right)_{\infty}}
 \exp\left(\frac{y^{2}}{\log q^{2}}\right)dy\\
 & = & \sum_{k=0}^{\infty}\frac{q^{k^{2}/2}\left(-x\right)^{k}}{\left(q;q\right)_{k}}
  \frac{1}{\sqrt{\pi\log q^{-2}}}\int_{-\infty}^{\infty}\frac{\exp\left(\frac{y^{2}}{\log q^{2}}+iky\right)dy}
  {\left(xq^{1/2+n}e^{iy};q\right)_{\infty}}\\
 & = & \sum_{k=0}^{\infty}\frac{q^{k^{2}}\left(-x\right)^{k}}{\left(q;q\right)_{k}}
  \left(-xq^{n+k+1};q\right)_{\infty},
\end{eqnarray*}
that is,
\begin{equation}
\frac{S_{n}\left(x;q\right)\left(q;q\right)_{n}}{\left(-xq^{n+1};q\right)_{\infty}}=\sum_{k=0}^{\infty}\frac{q^{k^{2}}\left(-x\right)^{k}}{\left(q,-xq^{n+1};q\right)_{k}}.
\label{eq:series and identities sw 1}
\end{equation}

One can prove \eqref{eq:series and identities sw 1} directly by expanding $(-xq^{n+1+k};q)_\infty$ as a sum over $m$. The coefficient of $x^s$ in 
$$
\sum_{k=0}^{\infty}\frac{q^{k^{2}}\left(-x\right)^{k}}{\left(q;q\right)_{k}}\left(-xq^{n+k+1};q\right)_{\infty}
$$
is 
\bea
\notag
q^{ns+s(s+1)/2} \sum_{k=0}^s \gauss{s}{k} (-1)^k q^{-nk+k(k-1)/2} 
= q^{ns+s(s+1)/2} (q^{-n};q)_s,
\eea
where we used \eqref{eqqbinom}  in the last step. This proves  
\eqref{eq:series and identities sw 1}.

From \eqref{eq:fouriersw 5} and \eqref{eq:fourier airy 9} we obtain
\begin{eqnarray*}
S_{n}\left(x;q\right)\left(q;q\right)_{n} & = & \frac{1}{\sqrt{\pi\log q^{-2}}}
\int_{-\infty}^{\infty}\frac{\left(q^{1/2}xe^{iy};q\right)_{\infty}}
{\left(xe^{iy}q^{n+1/2};q\right)_{\infty}}\exp\left(\frac{y^{2}}
{\log q^{2}}\right)dy\\
 & = & \sum_{k=0}^{\infty}\frac{\left(xq^{n+1/2}\right)^{k}}
 {\left(q;q\right)_{k}}\int_{-\infty}^{\infty}
 \frac{\left(q^{1/2}xe^{iy};q\right)_{\infty}}{\sqrt{\pi\log q^{-2}}}
  \exp\left(\frac{y^{2}}{\log q^{2}}+iky\right)dy\\
 & = & \sum_{k=0}^{\infty}\frac{\left(xq^{n}\right)^{k}}{\left(q;q\right)_{k}}
  q^{\binom{k+1}{2}}A_{q}\left(q^{k}x\right),
\end{eqnarray*}
that is,
\begin{equation}
S_{n}\left(x;q\right)=\frac{1}{\left(q;q\right)_{n}} 
 \sum_{k=0}^{\infty}\frac{\left(xq^{n}\right)^{k}}{\left(q;q\right)_{k}}
 q^{\binom{k+1}{2}}
  A_{q}\left(q^{k}x\right).\label{eq:series and identities sw 3}
\end{equation}
Formula \eqref{eq:series and identities sw 3} can be thought of an 
inverse to the standard generating function \cite[(3.27.11)]{Koe:Swa}
\bea
\Sum S_n(x;q) w^n = \frac{A_q(xw)}{(w;q)_\infty}. 
\label{eq:series and identities sw 4}
\eea
Note that \eqref{eq:series and identities sw 3} also follows from 
equating like powers of $x$. 

The next group of formulas resemble connection relations for 
$q$-Laguerre polynomials. First 
we note  the generating function \cite[(3.21.12)]{Koe:Swa} 
\bea
\Sum L_{n}^{(\alpha)}\left(x;q\right)  t^nq^{-\al n}
= \frac{1}{(tq^{-\al};q)_\infty}   
 {}_{1}\phi_1\left(\left. \begin{matrix} 
 -x  \\
0
\end{matrix}\, \right|q,qt\right).  
\label{eqgenfunSW}  
\eea
This implies the connection relation 
\bea
\label{eqconnaLaguerre}
q^{-\al n} L_{n}^{(\alpha)}\left(x;q\right) = 
\sum_{k=0}^n  \frac{(q^{\al-\beta};q)_{n-k}}{(q;q)_{n-k}} 
q^{-\beta k} L_k^{(\beta)}\left(x;q\right).  
\eea
For convenience we collect the next few results in a theorem.
\begin{thm} \label{thm721}
The $q$-Laguerre polynomials have the following properties 
\bea
 \frac{\left(q^{\alpha+n+1};q\right)_{\infty}\left(q;q\right)_{n}}
 {(-x q^{\alpha+n+1};q)_\infty}
 L_{n}^{(\alpha)}\left(x;q\right)
 &=&   \sum_{k=0}^{\infty}
  \frac{q^{\binom{k}{2}+k(\al+1)}}{\left(q;q\right)_{k}}\ (-1)^{k}
\frac{(-x;q)_{k}}{(-xq^{\al+n+1};q)_k}, 
\label{eq:series and identities laguerre 1}\\
\frac{\left(q^{\alpha+1};q\right)_{n}}{\left(q;q\right)_{n}}
&=&\sum_{k=0}^{\infty}\frac{\left(q^{n};q\right)_{k}q^{\binom{k}{2}}
\left(xq^{\alpha+1}\right)^{k}}{\left(q,q^{\alpha+n+1};q\right)_{k}}
L_{n}^{(\alpha+k)}\left(x;q\right), 
\label{eq:series and identities laguerre 2} \\
\frac{(q^{\al+1};q)_\infty}{(q^{\bt+1};q)_\infty} L_{n}^{(\alpha)}\left(x;q\right) &=& \sum_{k=0}^{\infty}\frac{\left(q^{\beta-\alpha};q\right)_{k}q^{\binom{k}{2}}}{\left(q,q^{\beta+n+1};q\right)_{k}}\left(-q^{\alpha+1}\right)^{k}L_{n}^{(\beta+k)}\left(xq^{\alpha-\beta};q\right).
\label{eq:series and identities laguerre 3}
\eea
\end{thm}
\begin{proof}
Applying  \eqref{eq:fourier pair laguerre 5} we  obtain
\begin{eqnarray*}
&{}& \left(q^{\alpha+n+1};q\right)_{\infty}\left(q;q\right)_{n}L_{n}^{(\alpha)}\left(x;q\right)  
 =  \int_{-\infty}^{\infty}\frac{\left(xe^{iy}q^{\alpha+1/2};q\right)_{n}\exp\left(\frac{y^{2}}
  {\log q^{2}}\right)}{\sqrt{\pi\log q^{-2}}\left(-q^{\alpha+1/2}e^{iy};q\right)_{\infty}}dy\\
 & {}  &\quad =  \int_{-\infty}^{\infty}\frac{\left(xe^{iy}q^{\alpha+1/2};q\right)_{\infty}}
  {\left(-q^{\alpha+1/2}e^{iy};q\right)_{\infty}}\frac{\left(\pi\log q^{-2}\right)^{-1/2}
   \exp\left(\frac{y^{2}}{\log q^{2}}\right)}{\sqrt{\pi\log q^{-2}} 
   \left(xe^{iy}q^{\alpha+n+1/2};q\right)_{\infty}}dy\\
 & &\quad =   \sum_{k=0}^{\infty}\frac{\left(-x;q\right)_{k}}{\left(q;q\right)_{k}}
  \frac{\left(-q^{\alpha+1/2}\right)^{k}}{\sqrt{\pi\log q^{-2}}}\int_{-\infty}^{\infty}
   \frac{\exp\left(\frac{y^{2}}{\log q^{2}}+iky\right)dy}{\left(xe^{iy}q^{\alpha+n+1/2};q\right)_{\infty}}\\
 &{} &\quad =  \sum_{k=0}^{\infty}\frac{\left(-x;q\right)_{k}}{\left(q;q\right)_{k}}
  \left(-q^{\alpha}\right)^{k}q^{\binom{k+1}{2}}\left(-xq^{\alpha+n+k+1};q\right)_{\infty},
\end{eqnarray*}
where we used \eqref{eq:fourier11}  in the last step. 
This leads to the   identity \eqref{eq:series and identities laguerre 1}. 
It must be noted that one can prove \eqref{eq:series and identities laguerre 1}   by writing the 
right-hand side as 
\bea
\notag
\lim_{u\to 0} {}_2\phi_1  \left(\left. \begin{matrix} 
 -x , 1/u \\
-x q^{\al+n+1}
\end{matrix}\, \right|q,u q^{\alpha+1} \right), 
\eea
then apply   the transformation \eqref{eqHeine2} to see that the right-
hand side becomes
\bea
\notag
\frac{(q^{\al+1};q)_\infty}{(-x q^{\al+n+1};q)_\infty} \lim_{u\to 0}  {}_2\phi_1  \left(\left. \begin{matrix} 
 q^{-n} , 1/u \\
 q^{\al+1}
\end{matrix}, \right|q, -xu q^{\al+n+1} \right)
 \eea
and the result follows from the explicit representation \eqref{eqqL}.

Next we apply  \eqref{eq:fourier11} to establish 
\begin{eqnarray*}
\begin{gathered}
\frac{\left(q^{\alpha+1};q\right)_{\infty}}{\left(q;q\right)_{n}}  
=  \frac{1}{\sqrt{\pi\log q^{-2}}}\int_{-\infty}^{\infty}
\frac{\left(xe^{iy}q^{\alpha+1/2};q\right)_{n}
\exp\left(\frac{y^{2}}{\log q^{2}}\right)}
{\left(xe^{iy}q^{\alpha+1/2};q\right)_{n}\left(-q^{\alpha+1/2}
e^{iy};q\right)_{\infty}}\frac{dy}{\left(q;q\right)_{n}}\\
 =  \sum_{k=0}^{\infty}\frac{\left(q^{n};q\right)_{k}}{\left(q;q\right)_{k}}
 \frac{x^{k}q^{k\left(\alpha+1/2\right)}}{\sqrt{\pi\log q^{-2}}}
   \int_{-\infty}^{\infty}\frac{\left(xe^{iy}q^{\alpha+1/2};q\right)_{n}\exp
   \left(\frac{y^{2}}{\log q^{2}}+iky\right)}{\left(q;q\right)_{n}
   \left(-q^{\alpha+1/2}e^{iy};q\right)_{\infty}}dy\\
  =  \left(q^{\alpha+n+1};q\right)_{\infty}\sum_{k=0}^{\infty}
 \frac{\left(q^{n};q\right)_{k}q^{\binom{k}{2}}\left(xq^{\alpha+1}\right)^{k}}
 {\left(q,q^{\alpha+n+1};q\right)_{k}}L_{n}^{(\alpha+k)}\left(x;q\right),
\end{gathered}
\end{eqnarray*}
where we used \eqref{eq:fourier pair laguerre 5} in the last step. This  
proves  \eqref{eq:series and identities laguerre 2}.  An alternate proof 
of  \eqref{eq:series and identities laguerre 2} is to replace 
$L_n^{(\al+k)}(x)$ by its series expansion then rearrange the  terms. 
This proof is rather lengthy.

From \eqref{eq:fourier pair laguerre 5} to get
\begin{eqnarray*}
\begin{gathered}
\left(q^{\alpha+n+1};q\right)_{\infty}L_{n}^{(\alpha)}\left(x;q\right) 
 =  \int_{-\infty}^{\infty}\frac{\left(xe^{iy}q^{\alpha+1/2};q\right)_{n}}
{\sqrt{\pi\log q^{-2}}\left(q;q\right)_{n}}\frac{\exp\left(\frac{y^{2}}{\log 
q^{2}}\right)}{\left(-q^{\alpha+1/2}e^{iy};q\right)_{\infty}}dy\\
  =  \int_{-\infty}^{\infty}\frac{\left(xe^{iy}q^{\alpha+1/2};q\right)_{n}}
 {\sqrt{\pi\log q^{-2}}\left(q;q\right)_{n}}\frac{\exp\left(\frac{y^{2}}
 {\log q^{2}}\right)}{\left(-q^{\beta+1/2}e^{iy};q\right)_{\infty}}%
 \frac{\left(-q^{\beta+1/2}e^{iy};q\right)_{\infty}}
 {\left(-q^{\alpha+1/2}e^{iy};q\right)_{\infty}}dy\\
  =  \sum_{k=0}^{\infty}\frac{\left(q^{\beta-\alpha};q\right)_{k}}
 {\left(q;q\right)_{k}}\left(-q^{\alpha+1/2}\right)^{k}
   \int_{-\infty}^{\infty}
 \frac{\left(xe^{iy}q^{\alpha+1/2};q\right)_{n}}
 {\sqrt{\pi\log q^{-2}}\left(q;q\right)_{n}}\frac{\exp\left(\frac{y^{2}}
 {\log q^{2}}+iky\right)dy}{\left(-q^{\beta+1/2}e^{iy};q\right)_{\infty}}\\
  =  \sum_{k=0}^{\infty}\frac{\left(q^{\beta-\alpha};q\right)_{k}}
 {\left(q,q^{\beta+n+1};q\right)_{k}}
 \left(-q^{\alpha+1}\right)^{k}q^{\binom{k}{2}}
 \left(q^{\beta+n+1};q\right)_{\infty}L_{n}^{(\beta+k)}
 \left(xq^{\alpha-\beta};q\right),
\end{gathered}
\end{eqnarray*}
and \eqref{eq:series and identities laguerre 3} follows. Formula 
\eqref{eq:series and identities laguerre 3} also follows by direct 
computation. 
\end{proof}

\begin{thm} We have the connection relation 
\begin{equation}
L_{n}^{(\alpha)}\left(x;q\right)=\frac{1}
{\left(q^{\alpha+n+1};q\right)_{\infty}}\sum_{k=0}^{\infty}
\frac{q^{\binom{k}{2}}}
{\left(q;q\right)_{k}}\left(-q^{\alpha+1}\right)^{k}
S_{n}\left(xq^{k+\alpha};q\right).
\label{eq:series and identities laguerre 4}
\end{equation}
and its inverse
\begin{equation}
\frac{S_{n}\left(xq^{\alpha};q\right)}{\left(q^{\alpha+n+1};q\right)_{\infty}}
=\sum_{k=0}^{\infty}\frac{q^{k^{2}+k\alpha}L_{n}^{(\alpha+k)}
\left(x;q\right)}{\left(q,q^{\alpha+n+1};q\right)_{k}}.
 \label{eq:series and identities laguerre 5}
\end{equation}
\end{thm}
\begin{proof}
 We discovered these formulas by using  \eqref{eq:fourier pair laguerre 5}, \eqref{eq:fouriersw 5}, \eqref{eq:fourier pair laguerre 5} and \eqref{eq:fouriersw 5}. However a direct proof of 
 \eqref{eq:series and identities laguerre 4} is straight forward. 
 The proof of \eqref{eq:series and identities laguerre 5} uses 
 \bea
I_{\nu}^{(2)}\left(2z;q\right) 
= \frac{z^\nu}{(q;q)_\infty} \; {}_1\phi_1(z^2; 0; q, q^{\nu+1}).
\label{eq:specialvalugeneral}
\eea
which we recently proved in \cite{Ism:Zha3}. 
\end{proof}

\subsection{The $q$-Bessel Function }
 
 \begin{thm} The expansion formula 
 \begin{equation}
J_{\nu+\alpha}^{(2)}\left(z;q\right)q^{\alpha\nu/2}
=\left(\frac{z}{2}\right)^{\nu}\sum_{k=0}^{\infty}
\frac{\left(q^{-\nu};q\right)_{k}}{\left(q;q\right)_{k}}
q^{\binom{k+1}{2}}\left(-\frac{2q^{\left(\nu+2\alpha\right)/2}}{z}
\right)^{k}J_{\alpha+k}^{(2)}\left(zq^{\nu/2};q\right).
\label{eq:series and identities bessel 1}
\end{equation}
and the generating function 
\begin{equation}
\frac{J_{\nu}^{(2)}\left(2wz;q\right)}{\left(wz\right)^{\nu}}
=\frac{\left(w^{2},q^{\nu+1};q\right)_{\infty}}{\left(q,q\right)_{\infty}}
\sum_{n=0}^{\infty}\frac{L_{n}^{(\nu)}\left(z^{2};q\right)w^{2n}}
{\left(q^{\nu+1};q\right)_{n}}, 
\label{eq:series and identities bessel 3 a}
\end{equation}
hold. Moreover \eqref{eq:series and identities bessel 3 a}
has the inverse relation 
\begin{equation}
L_{n}^{(\alpha)}\left(\frac{z^{2}}{4};q\right)\left(\frac{z}{2}\right)^{\alpha}=\frac{\left(q^{n+1};q\right)_{\infty}}{\left(q^{\alpha+n+1};q\right)_{\infty}}\sum_{k=0}^{\infty}\frac{q^{\binom{k+1}{2}}}{\left(q;q\right)_{k}}\left(\frac{zq^{\alpha+n}}{2}\right)^{k}J_{k+\alpha}^{(2)}\left(z;q\right),\label{eq:series and identities bessel 3 b}
\end{equation}
which gives analytic continuation of $n$ to $\mathbb{C}$. Furthermore we have the multiplication formula 
\begin{equation}
J_{\nu}^{(2)}\left(wz;q\right)=z^{\nu}\sum_{k=0}^{\infty}\frac{\left(z^{2};q\right)_{k}}{\left(q;q\right)_{k}}q^{\binom{k+1}{2}}\left(\frac{q^{\nu}w}{2}\right)^{k}J_{\nu+k}^{(2)}\left(w;q\right).
\label{eq:series and identities bessel 2}
\end{equation}
\end{thm}
\begin{proof}
The first two identities can be proved using  the inverse pair 
\eqref{eq:fourier pair bessel 10}.  We include a more elementary proof. 
From the definition of $J_\nu^{(2)}$ we find that the right-hand side of 
\eqref{eq:series and identities bessel 1} is 
\bea
\notag
\begin{gathered}
\left(\frac{z}{2}\right)^{\nu+\al}\sum_{k=0}^{\infty}
\frac{\left(q^{-\nu};q\right)_{k}}{\left(q;q\right)_{k}}
q^{\binom{k+1}{2}} (-1)^k  q^{k(\nu+\al)}
  \frac{(q^{\al+1};q)_\infty}{(q;q)_\infty}
\sum_{j=0}^\infty \frac{(-1)^j q^{j(j+\al+k)}}
{(q;q)_j(q^{\al+1};q)_{j+k}}
 \left(\frac{z}{2}q^{\nu/2}\right)^{2j} \\
 =\left(\frac{z}{2}\right)^{\nu+\al}   \frac{(q^{\al+1};q)_\infty}{(q;q)_\infty}\; 
  \sum_{j=0}^\infty \frac{(-1)^j q^{j(j+\al)}}
{(q;q)_j(q^{\al+1};q)_{j}}
 \left(\frac{z}{2}q^{\nu/2}\right)^{2j} \lim_{u\to 0}
 {}_2\f_1(q^{-\nu}, 1/u;q^{\al+j+i};q, u q^{\al+\nu+j+1}).
\end{gathered}
\eea
The ${}_2\f_1$ sums to 
$(q^{\nu+\al+j+1};q)_\infty/(q^{\al+j+1};q)_\infty$ and the result 
simplifies and establishes \eqref{eq:series and identities bessel 1}. 
The generating function \eqref{eq:series and identities bessel 3 a}
   is known, see \cite[(3.21.13)]{Koe:Swa}.  The proofs of 
   \eqref{eq:series and identities bessel 3 b} and   
   \eqref{eq:series and identities bessel 2}      are straight forward and 
   will be omitted. 
\end{proof} 
The special case $n=0$ of 
\eqref{eq:series and identities bessel 3 b} is 
\begin{equation}
\frac{\left(q^{\nu+1};q\right)_{\infty}\left(\frac{z}{2}\right)^{\nu}}
{\left(q;q\right)_{\infty}}=\sum_{k=0}^{\infty}
\frac{q^{\binom{k+1}{2}}}{\left(q;q\right)_{k}}
\left(\frac{q^{\nu}z}{2}\right)^{k}J_{k+\nu}^{(2)}\left(z;q\right).
\label{eq:series and identities bessel 4 b}
\end{equation}
 By writing $J_{\nu}^{(2)}$ as a confluent limit of a ${}_2\f_1$ then 
 apply  the iterated Heine transformation \cite[(III.2)]{Gas:Rah} we find 
 that
\begin{equation}
J_{\nu}^{(2)}\left(z;q\right)=\frac{\left(\frac{z}{2}\right)^{\nu}}
{\left(q;q\right)_{\infty}}\sum_{k=0}^{\infty}
\frac{\left(-z^{2}/4;q\right)_{k}}{\left(q;q\right)_{k}}q^{\binom{k+1}{2}}
\left(-q^{\nu}\right)^{k}.
\label{eq:series and identities bessel 4 a}
\end{equation}
 
\begin{thm}
We have the inverse pair
\bea
J_{\nu}^{(2)}\left(2z;q\right)&=&\frac{z^{\nu}}{\left(q;q\right)_{\infty}}
\sum_{k=0}^{\infty}\frac{\left(-q^{\nu}\right)^{k}}{\left(q;q\right)_{k}}
q^{\binom{k+1}{2}}A_{q}\left(q^{\nu+k}z^{2}\right), 
\label{eq:series and identities bessel 5 a}\\
\frac{z^{\nu}A_{q}\left(q^{\nu}z^{2}\right)}{\left(q;q\right)_{\infty}} &=&
\sum_{k=0}^{\infty}\frac{q^{k^{2}}}{\left(q;q\right)_{k}}\left(\frac{q^{\nu}}{z}\right)^{k}J_{k+\nu}^{(2)}\left(2z;q\right).
\label{eq:series and identities bessel 5 b}
\eea
\end{thm}
 One can prove  \eqref{eq:series and identities bessel 5 a} from 
 \eqref{eq:fourier pair bessel 10} and prove 
 \eqref{eq:series and identities bessel 5 b}  by using 
 \eqref{eq:fourier airy 9}.  Both formulas however follow from series 
 manipulations and the $q$-Gauss sum. 
 
  The next group of expansions involve the $q$-Bessel function
   $J_\nu^{(3)}$. 
   \begin{thm}
For the $q$-Bessel function $J_{\nu}^{(3)}\left(z;q\right)$ we have
\bea
z^{\mu-\nu}J_{\nu}^{(3)}\left(2z;q\right)&=&\sum_{n=0}^{\infty}
\frac{(q^{\mu-\nu};q)_{n}}{(q;q)_{n}}\left(-\frac{q^{\nu+1/2}}{z}\right)^{n}
J_{\mu+n}^{(3)}\left(2q^{n/2}z;q\right),\label{eq:b3-13}\\
J_{\nu}^{(3)}\left(2z/w;q\right)&=&w^{-\nu}\sum_{n=0}^{\infty}
\frac{(w^{2};q)_{n}}{(q;q)_{n}}
\left(-\frac{q^{(1-\nu)/2}z}{w^{2}}\right)^{n}J_{n+\nu}^{(3)}\left(2q^{n/2}z;q\right),
\label{eq:b3-14}\\
\frac{\left(z^{2}q;q\right)_{\infty}}{(q;q)_{\infty}}&=&\sum_{n=0}^{\infty}
\frac{q^{n(n+\nu)/2}}{(q;q)_{n}}\frac{J_{\nu+n}^{(3)}
\left(2q^{n/2}z;q\right)}{z^{\nu+n}},\label{eq:b3-15}\\
L_{n}^{(\nu)}\left(-z^{2}q^{1/2};q\right)&=&
\frac{(q^{n+1};q)_{\infty}}{\left(q^{\nu+n+1};q\right)_{\infty}}
\sum_{k=0}^{\infty}\frac{q^{k(k-\nu -n)/2}}{(q;q)_{k}}z^{k}
\frac{J_{k+\nu}^{(3)}\left(2zq^{(n+k)/2};q\right)}{\left(zq^{n/2}\right)^{\nu}},
\label{eq:b3-18}\\
\frac{J_{\nu}^{(3)}\left(2zq^{n/2};q\right)}{z^{\nu}}
&=&\frac{\left(q^{\nu+n+1};q\right)_{\infty}}{(q^{n+1};q)_{\infty}}
\sum_{k=0}^{\infty}\frac{q^{\binom{k+1}{2}}\left(-z^{2}\right)^{k}}
{(q,q^{\nu+n+1};q)_{k}}L_{n}^{(\nu+k)}\left(-z^{2}q^{-\nu};q\right).
\label{eq:b3-19}
\eea
\end{thm}
\begin{proof}
From (\ref{eq:b3-8}) we get 
\begin{eqnarray*}
z^{-\nu}J_{\nu}^{(3)}\left(2z;q\right) & = & \frac{1}{\sqrt{\pi\log q^{-2}}}\int_{-\infty}^{\infty}\frac{\exp\left(\frac{x^{2}}{\log q^{2}}\right)dx}{\left(q,-q^{\nu+1/2}e^{ix},-z^{2}q^{1/2}e^{ix};q\right)_{\infty}}\\
 & = & \frac{1}{\sqrt{\pi\log q^{-2}}}\int_{-\infty}^{\infty}\frac{(-q^{\mu+1/2}e^{ix};q)_{\infty}}{(-q^{\nu+1/2}e^{ix};q)_{\infty}}\frac{\exp\left(\frac{x^{2}}{\log q^{2}}\right)dx}{\left(q,-q^{\mu+1/2}e^{ix},-z^{2}q^{1/2}e^{ix};q\right)_{\infty}}\\
 & = & \sum_{n=0}^{\infty}\frac{(q^{\mu-\nu};q)_{n}}{(q;q)_{n}}\frac{\left(-q^{\nu+1/2}\right)^{n}}{\sqrt{\pi\log q^{-2}}}\int_{-\infty}^{\infty}\frac{\exp\left(\frac{x^{2}}{\log q^{2}}+inx\right)dx}{\left(q,-q^{\mu+1/2}e^{ix},-z^{2}q^{1/2}e^{ix};q\right)_{\infty}}\\
 & = & \sum_{n=0}^{\infty}\frac{(q^{\mu-\nu};q)_{n}}{(q;q)_{n}}\left(-q^{\nu+1/2}\right)^{n}z^{-\mu-n}J_{\mu+n}^{(3)}\left(2q^{n/2}z;q\right),
\end{eqnarray*}
 and \eqref{eq:b3-13} follows. 
 Similarly, we find that  
\[
\begin{aligned} & \left(z/w\right)^{-\nu}J_{\nu}^{(3)}\left(2z/w;q\right)=\frac{1}{\sqrt{\pi\log q^{-2}}}\int_{-\infty}^{\infty}\frac{\exp\left(\frac{x^{2}}{\log q^{2}}\right)dx}{\left(q,-q^{\nu+1/2}e^{ix},-z^{2}q^{1/2}e^{ix}/w^{2};q\right)_{\infty}}\\
 & =\frac{1}{\sqrt{\pi\log q^{-2}}}\int_{-\infty}^{\infty}\frac{(-z^{2}q^{1/2}e^{ix};q)_{\infty}}{(-z^{2}q^{1/2}e^{ix}/w^{2};q)_{\infty}}\frac{\exp\left(\frac{x^{2}}{\log q^{2}}\right)dx}{\left(q,-q^{\nu+1/2}e^{ix},-z^{2}q^{1/2}e^{ix};q\right)_{\infty}}\\
 & =\sum_{n=0}^{\infty}\frac{(w^{2};q)_{n}}{(q;q)_{n}}\frac{q^{n/2}\left(-z^2/w^2\right)^{n}}{\sqrt{\pi\log q^{-2}}}\int_{-\infty}^{\infty}\frac{(-1)^{n}\exp\left(\frac{x^{2}}{\log q^{2}}+inx\right)dx}{\left(q,-q^{\nu+1/2}e^{ix},-z^{2}q^{1/2}e^{ix};q\right)_{\infty}}\\
 & =z^{-\nu}\sum_{n=0}^{\infty}\frac{(w^{2};q)_{n}}{(q;q)_{n}}\(-\frac{q^{(1-\nu)/2}z}{w^2}\)^nJ_{n+\nu}^{(3)}\left(2q^{n/2}z;q\right),
\end{aligned}
\]
which establishes \eqref{eq:b3-14}.

To prove \eqref{eq:b3-15}  we use 
\[
q^{\alpha^{2}/2}\left(-zq^{\alpha+1/2};q\right)_{\infty}=\frac{1}{\sqrt{2\pi\log q^{-1}}}\int_{-\infty}^{\infty}\frac{\exp\left(\frac{x^{2}}{\log q^{2}}+i\alpha x\right)}{\left(ze^{ix};q\right)_{\infty}}dx
\]
 to get
\[
\begin{aligned} & \frac{q^{\alpha^{2}/2}\left(z^{2}q^{\alpha+1};q\right)_{\infty}}{(q;q)_{\infty}}=\frac{1}{\sqrt{2\pi\log q^{-1}}}\int_{-\infty}^{\infty}\frac{\exp\left(\frac{x^{2}}{\log q^{2}}+i\alpha x\right)}{\left(q,-z^{2}q^{1/2}e^{ix};q\right)_{\infty}}dx\\
 & =\frac{1}{\sqrt{2\pi\log q^{-1}}}\int_{-\infty}^{\infty}\frac{(-q^{\nu+1/2}e^{ix};q)_{\infty}\exp\left(\frac{x^{2}}{\log q^{2}}+i\alpha x\right)}{\left(q,-q^{\nu+1/2}e^{ix},-z^{2}q^{1/2}e^{ix};q\right)_{\infty}}dx\\
 & =\sum_{n=0}^{\infty}\frac{q^{n^{2}/2+n\nu}}{(q;q)_{n}}\frac{1}{\sqrt{2\pi\log q^{-1}}}\int_{-\infty}^{\infty}\frac{\exp\left(\frac{x^{2}}{\log q^{2}}+i(\alpha+n)x\right)dx}{\left(q,-q^{\nu+1/2}e^{ix},-z^{2}q^{1/2}e^{ix};q\right)_{\infty}}\\
 & =\sum_{n=0}^{\infty}\frac{q^{n^{2}/2+n\nu/2-\alpha \nu/2}}{(q;q)_{n}}z^{-\alpha-n-\nu}J_{\alpha+\nu+n}^{(3)}\left(2q^{(\alpha+n)/2}z;q\right),
\end{aligned}
\]
and the result follows.  
 We then proceed in a similar fashion 
to obtain 
\begin{eqnarray*}
 &  & \left(q^{\alpha+n+1};q\right)_{\infty}L_{n}^{(\alpha)}\left(-z^{2}q^{1/2};q\right)\\
 & = & \frac{1}{\sqrt{2\pi}}\int_{-\infty}^{\infty}\frac{\left(-z^{2}e^{iy}q^{1/2};q\right)_{n}}{\sqrt{\log q^{-1}}\left(q;q\right)_{n}}\frac{\exp\left(\frac{y^{2}}{\log q^{2}}\right)}{\left(-q^{\alpha+1/2}e^{iy};q\right)_{\infty}}dy\\
 & = & \frac{1}{\sqrt{2\pi}}\int_{-\infty}^{\infty}\frac{\left(-z^{2}e^{iy}q^{1/2};q\right)_{\infty}}{\sqrt{\log q^{-1}}\left(q;q\right)_{n}}\frac{\exp\left(\frac{y^{2}}{\log q^{2}}\right)}{\left(-q^{\alpha+1/2}e^{iy},-z^{2}e^{iy}q^{n+1/2};q\right)_{\infty}}dy\\
 & = & (q^{n+1};q)_{\infty}\sum_{k=0}^{\infty}\frac{q^{k^{2}/2}}{(q;q)_{k}}z^{2k}\int_{-\infty}^{\infty}\frac{\exp\left(\frac{y^{2}}{\log q^{2}}+iky\right)dy}{\left(q,-q^{\alpha+1/2}e^{iy},-\left(zq^{n/2}\right)^{2}e^{iy}q^{1/2};q\right)_{\infty}}\frac{dy}{\sqrt{\pi\log q^{-2}}}\\
 & = & \frac{(q^{n+1};q)_{\infty}}{z^{\alpha }q^{\alpha n/2}}\sum_{k=0}^{\infty}\frac{q^{k(k-\alpha-n)/2}}{(q;q)_{k}}z^{k})J_{k+\alpha}^{(3)}\left(2zq^{(n+k)/2};q\right),
\end{eqnarray*}
and we have established  \eqref{eq:b3-18}.  Finally  \eqref{eq:b3-19}
 follows from the following calculation
\begin{eqnarray*}
\frac{J_{\nu}^{(3)}\left(2zq^{n/2};q\right)}{z^{\nu}} & = & \frac{1}{\sqrt{\pi\log q^{-2}}}\int_{-\infty}^{\infty}\frac{\exp\left(\frac{x^{2}}{\log q^{2}}\right)dx}{\left(q,-q^{\nu+1/2}e^{ix},-z^{2}q^{1/2+n}e^{ix};q\right)_{\infty}}\\
 & = & \frac{1}{(q;q)_{\infty}}\sum_{k=0}^{\infty}\frac{\left(-z^{2}q^{1/2}\right)^{k}}{(q;q)_{k}}\int_{-\infty}^{\infty}\frac{(-z^{2}q^{1/2}e^{ix};q)_{n}\exp\left(\frac{x^{2}}{\log q^{2}}+ikx\right)dx}{\left(-q^{\nu+1/2}e^{ix};q\right)_{\infty}\sqrt{\pi\log q^{-2}}}\\
 & = & \sum_{k=0}^{\infty}\frac{\left(-z^{2}q^{1/2}\right)^{k}}{(q;q)_{k}}q^{k^{2}/2}\frac{\left(q^{k+\nu+n+1};q\right)_{\infty}}{(q^{n+1};q)_{\infty}}L_{n}^{(\nu+k)}\left(-z^{2}q^{-\nu};q\right).
\end{eqnarray*}
This completes the proof. 
\end{proof}

\subsection{The ${}_{1}\phi_{1}$ Function}
Our next theorem gives expansions of ${}_1\f_1$ into a series of 
Bessel or Ramanujan functions. 
\begin{thm}
We have the expansions 
\bea
{}_{1}\phi_{1}\left(-aq^{\nu+1};q^{\nu+1};q,z\right)
&=&\frac{\left(q;q\right)_{\infty}}
{\left(q^{\nu+1};q\right)_{\infty}}\sum_{k=0}^{\infty}
\frac{q^{\binom{k}{2}}(-z)^k}{\left(q;q\right)_{k}}\frac{J_{\nu+k}^{(2)}
\left(2\sqrt{az}\right)}{\left(az\right)^{(\nu+k)/2}}.  
\label{eq:series confluent 2}\\
\left(z;q\right)_{\infty}{}_{1}\phi_{1}\left(a;z;q,-z\right)
&=&\sum_{k=0}^{\infty}\frac{z^{2k}q^{2k^{2}-k}}
{\left(q^{2};q^{2}\right)_{k}}A_{q}\left(q^{2k-1}az\right). 
\label{eq:series confluent 4}
\eea
\end{thm}
\begin{proof}
One proof of \eqref{eq:series confluent 2} uses 
the case  $b=q^{\nu+1},\alpha=0$ of 
 \eqref{eq:fourier pair confluent 5}. Another proof follows by 
 interchanging sums and the use of the finite $q$-binomial 
 theorem.  One way to prove  \eqref{eq:series confluent 4} 
 is to take  $b=-z\sqrt{q}$ in \eqref{eq:fourier pair confluent 5} to get 
 \begin{eqnarray*}
&{}& \left(-z\sqrt{q};q\right)_{\infty}{}_{1}\phi_{1}\left(-a\sqrt{q};-z\sqrt{q};q,zq^{1/2}\right)  =  \frac{1}{\sqrt{\pi\log q^{-2}}}\int_{-\infty}^{\infty}\frac{\left(az\sqrt{q}e^{ix};q\right)_{\infty}\exp\left(\frac{x^{2}}{\log q^{2}}\right)dx}{\left(ze^{ix},-ze^{ix};q\right)_{\infty}}\\
 & = & \frac{1}{\sqrt{\pi\log q^{-2}}}\int_{-\infty}^{\infty}\frac{\left(az\sqrt{q}e^{ix};q\right)_{\infty}\exp\left(\frac{x^{2}}{\log q^{2}}\right)dx}{\left(z^{2}e^{2ix};q^{2}\right)_{\infty}}\\
&=&   \sum_{k=0}^{\infty}\frac{z^{2k}}{\left(q^{2};q^{2}\right)_{k}}\int_{-\infty}^{\infty}\frac{\left(az\sqrt{q}e^{ix};q\right)_{\infty}\exp\left(\frac{x^{2}}{\log q^{2}}+2ikx\right)dx}{\sqrt{\pi\log q^{-2}}}
  =  \sum_{k=0}^{\infty}\frac{z^{2k}q^{2k^{2}}}{\left(q^{2};q^{2}\right)_{k}}A_{q}\left(q^{2k}az\right),
\end{eqnarray*}
which is essentially  \eqref{eq:series confluent 4}.
 Another interesting proof is to write the left-hand side as 
\bea
\notag
\begin{gathered}
\Sum \frac{(a;q)_n}{(q;q)_n} (zq^n:q)_\infty z^n q^{\binom{n}{2}}
= \sum_{j\ge 0, n \ge k\ge0} \frac{(-a)^k q^{\binom{k}{2}}}{(q;q)_k(q;q)_{n-k}}\frac{(-zq^n)^j q^{\binom{j}{2}}}{(q;q)_j} 
z^n q^{\binom{n}{2}} \\
= \sum_{m=0}^\infty \frac{z^m}{(q;q)_m}
 A_q(azq^{m-1})q^{\binom{m}{2}}\sum_{j=0}^m 
\gauss{m}{j}(-1)^j.
\end{gathered}
\eea
The $j$ sum is $e^{-i\pi m/2}H_m(0|q)$. But $H_m(0|q) =0$ for $n$ 
odd  and $H_{2m}(0|q) = (-1)^m(q;q)_{2m}/(q^2;2q^2)_m$,  
  \cite[(13.2.19]{Ismbook}. This completes the proof. 
 \end{proof}
 
 \begin{thm}
 We have 
 \bea
\frac{\left(z;q\right)_{\infty}}{\left(b;q\right)_{\infty}}&=&
\sum_{k=0}^{\infty}\frac{\left(b/(az);q\right)_{k}
\left(-az\right)^{k}q^{\binom{k}{2}}}{\left(q,b;q\right)_{k}} \; 
{}_{1}\phi_{1}\left(a;bq^{k};q,zq^{k}\right).
\label{eq:series confluent 5}
\eea
and
\begin{eqnarray}
\begin{gathered}
J_{\nu}^{(2)}\left(z\sqrt{q};q\right)\left(\frac{z\sqrt{q}}{2}\right)^{-\nu}  \qquad \qquad  \qquad \qquad  \\
 = 
 \frac{\left(q^{\nu+1};q\right)_{\infty}}{\left(q;q\right)_{\infty}}\sum_{k=0}^{\infty}\frac{z^{k}q^{k^{2}}}{\left(q,q^{\nu+1};q\right)_{k}} \;
{} _{1}\phi_{1}\left(-\frac{q^{\nu+1}z}{4};q^{\nu+k+1};q,zq^{k+1}\right).  
  \end{gathered}
   \label{eq:series confluent 6}
\end{eqnarray}
 \end{thm}
\begin{proof} Using  \eqref{eq:fourier11} and the $q$-binomial 
theorem gives \eqref{eq:series confluent 5}. A more 
direct proof is to expand 
the ${}_{1}\phi_{1}$ and write  the right-hand side in the form 
\bea
\notag
\begin{gathered}
\sum_{n,k=0}^\infty \frac{(b/az;q)_k(-az)^k q^{\binom{k}{2}}}
{(q;q)_k(b;q)_{n+k}}\; \frac{(a;q)_n}{(q;q)_n} 
(-zq^k)^nq^{\binom{n}{2}} = \sum_{m=0}^\infty
 \frac{q^{\binom{m}{2}}z^m(a;q)_m}{(b, q;q)_m}\; {}_2\f_1
\left(\left.  \begin{array}{c}
   q^{-m},  b/az,  
  \\
 q^{1-m}/a
     \end{array} \right| q , q  \right) \\
     = \sum_{m=0}^\infty
 \frac{q^{\binom{m}{2}}z^m(a;q)_m}{(b, q;q)_m} \frac{(q^{1-m}z/b;q)_m}{(q^{1-m}/a;q)_m}\left(\frac{b}{az}\right)^m = 
 \lim_{u\to 0} {}_2\f_1
\left(\left.  \begin{array}{c}
    b/z,  1/u
  \\
 b
     \end{array} \right| q , zu  \right) = \frac{(z;q)_\infty}{(b;q)_\infty}, 
     \end{gathered}
\eea
where we used the Chu-Vandermonde sum and the $q$-Gauus sum.
This gives \eqref{eq:series confluent 5}.  To prove 
\eqref{eq:series confluent 6} start 
 with  \eqref{eq:fourier pair bessel 10}  and proceed as follows 
\begin{eqnarray*}
&{}& J_{\nu}^{(2)}\left(z;q\right)\left(\frac{z}{2}\right)^{-\nu}  = 
 \frac{1}{\sqrt{\pi\log q^{-2}}}\int_{-\infty}^{\infty}
 \frac{\left(\frac{q^{\nu+1/2}z^{2}e^{ix}}{4};q\right)_{\infty}
 \exp\left(\frac{x^{2}}{\log q^{2}}\right)}{\left(q,-q^{\nu+1/2} 
 e^{ix};q\right)_{\infty}}dx\\
 & = & \frac{1}{\sqrt{\pi\log q^{-2}}}\int_{-\infty}^{\infty}
 \frac{\left(-ze^{ix},\frac{q^{\nu+1/2}z^{2}e^{ix}}{4};q\right)_{\infty}\exp
 \left(\frac{x^{2}}
 {\log q^{2}}\right)}{\left(q,-q^{\nu+1/2}e^{ix},-ze^{ix};q\right)_{\infty}}dx\\
 & = & \sum_{k=0}^{\infty}\frac{z^{k}q^{\binom{k}{2}}}{\left(q;q\right)_{k}}
 \int_{-\infty}^{\infty}\frac{\left(\frac{q^{\nu+1/2}z^{2}e^{ix}}
 {4};q\right)_{\infty}\exp\left(\frac{x^{2}}{\log q^{2}}+ikx\right)}
 {\left(q,-q^{\nu+1/2}e^{ix},-ze^{ix};q\right)_{\infty}}dx\\
 & = & \frac{1}{\left(q;q\right)_{\infty}}\sum_{k=0}^{\infty}\frac{z^{k}
 q^{\binom{k}{2}}}{\left(q;q\right)_{k}}\left(q^{\nu+k+1};q\right)_{\infty}
 q^{k^{2}/2} {} _{1}\phi_{1}
 \left(-\frac{q^{\nu+1/2}z}{4};q^{\nu+k+1};q,zq^{k+1/2}\right).
\end{eqnarray*}
An alternate proof follows by expanding the ${}_1\f_1$ as a sum over 
$n$, say, then expand $(-q^{\nu+1}z/4;q)_n$ by the $q$-binomial 
theorem. as a sum over $j$. After repleacing $n$ by $n+j$ and $n+k$ 
by m  the right-hand side of  \eqref{eq:series confluent 6} is 
\bea
\notag
\frac{(q^{\nu+1};q)_\infty}{(q;q)_\infty} 
\sum_{m,j=0}^\infty \frac{z^{m+2j}(-1)^{m+j} q^{j(j+\nu+1)}}
{4^j(q^{\nu+1};q)_{m+j}(q;q)_j(q;q)_m} q^{jm+ \binom{m+1}{2}}
\sum_{n=0}^m\gauss{m}{k} (-1)^k q^{\binom{k}{2}}.
\eea
The $k$ sum is $\delta_{m,0}$. This completes the proof. 
\end{proof}
  \begin{thm} \label{thm7.8}
  The   expansions  
 \bea
{}_{1}\phi_{1}\left(a;b;q,z\right)&=&\frac{\left(bd;q\right)_{\infty}}
{\left(b;q\right)_{\infty}}
 \sum_{k=0}^{\infty}\frac{\left(d;q\right)_{k}\left(-b\right)^{k}
 q^{\binom{k}{2}}}{\left(q,bd;q\right)_{k}}
  {}_{1}\phi_{1}\left(a;bdq^{k};q,zq^{k}\right),
  \label{eq:series confluent 9}
  \\
{}_{1}\phi_{1}\left(aw;b;q,z\right)&=&\sum_{k=0}^{\infty}
\frac{\left(w;q\right)_{k}\left(-z\right)^{k}
 q^{\binom{k}{2}}}{\left(q,b;q\right)_{k}}
 {}_{1}\phi_{1}\left(a;bq^{k};q,wzq^{k}\right),
  \label{eq:series confluent 10}
\eea
and
 \begin{eqnarray}
  \label{eq:series confluent 7}
  \begin{gathered}
\frac{\left(q;q\right)_{n}L_{n}^{(\alpha)}\left(x;q\right)}{\left(q^{\alpha+1};q\right)_{n}}  =  \sum_{k=0}^{\infty}\frac{\left(-q^{-\alpha-n-1/2};q\right)_{k}x^{k}q^{\left(k+2\alpha+2n+1\right)k/2}}{\left(q,q^{\alpha+1};q\right)_{k}}\\
 \times  {} _{1}\phi_{1}\left(-q^{\alpha+1/2};q^{\alpha+k+1};q,xq^{k+1/2}\right), 
 \end{gathered} 
\end{eqnarray}
hold. 
  \end{thm}
  \begin{proof} Formula \eqref{eq:series confluent 9} easily follows from 
  writing the ${}_1\f_1$ on the right side as a series then interchange 
  the sums.  We originally discovered \eqref{eq:series confluent 9}
   by applying \eqref{eq:fourier pair confluent 5}.  To prove 
   \eqref{eq:series confluent 10} write the ${}_1\f_1$ 
  on the right-hand side as a sum over $n$, then let $n+k =m$ and fix 
  $m$ and evaluate  the $k$ sum using the Chu-Vandermonde sum. 
  The result   simplifies to the left-hand side. Formula 
  \eqref{eq:series confluent 10} 
  also follows from \eqref{eq:series confluent 9}. 
  We now come to 
  \eqref{eq:series confluent 7}. We express the ${}_1\f_1$ on the 
  right-hand side as a sum over $j$ say, then let $m =j+k$ so the 
  right hand side is $\sum_{m=0}^\infty \sum_{k=0}^m \cdots$. 
  We then evaluate the  $k$ sum by the Chu-Vandermonde theorem 
  and the result follows. A more constructive proof of 
  \eqref{eq:series confluent 7} follows from 
   \eqref{eq:fourier pair laguerre 5}  
  \end{proof}

{\bf Acknowledgements} The authors are grateful to the department 
of mathematics and the Grants Council of Hong Kong for support 
and hospitality during the initial stages of this work. We also thank 
the Academia Sinica in Taipei for hospitality their support of the 
authors visits which allowed us to complete this work. We thank 
George Andrews and Kathrin Bringmann for their remarks and encouragement.


 \end{document}